\documentclass{amsart}
\usepackage{geometry, comment}   % See geometry.pdf to learn the layout options. There are lots.
\usepackage[parfill]{parskip}    % Activate to begin paragraphs with an empty line rather than an indent
\usepackage{graphicx}
\usepackage{amssymb}
\usepackage{hyperref}
\usepackage{pdfsync}
\usepackage{footmisc}
\usepackage[all]{xy}

\newtheorem{thm}{Theorem}[section]
\newtheorem{prop}[thm]{Proposition}
\newtheorem{lem}[thm]{Lemma}
\newtheorem{cor}[thm]{Corollary}
\newtheorem*{thm*}{Theorem}
\newtheorem*{thmmain}{Theorem \ref{thm:main}}
\newtheorem*{cormain}{Corollary \ref{cor:main}}
\newtheorem*{thmsollast}{Theorem \ref{thm:sollast}}

\theoremstyle{definition}

\newtheorem{defn}[thm]{Definition}

\newtheorem{prob}{Problem}
\newtheorem{rem}[thm]{Remark}
\newtheorem{expl}[thm]{Example}

\newtheorem*{defn*}{Definition}
\newtheorem*{rem*}{Remark}
\newcommand{\cool}{co-irreducible }

\newcommand{\UU}{\mathfrak{U}}

\newcommand{\p}{\mathfrak{p}}

\newcommand{\cA}{\mathcal{A}}

\newcommand{\cc}{\mathfrak{c}}
\DeclareMathOperator{\sol}{sol}
\DeclareMathOperator{\SL}{SL}
\DeclareMathOperator{\PSL}{PSL}
\DeclareMathOperator{\nf}{nf}

\DeclareMathOperator{\Hom}{Hom}
\DeclareMathOperator{\id}{id}
\DeclareMathOperator{\Ker}{Ker}
\DeclareMathOperator{\ncl}{ncl}

\DeclareMathOperator{\tp}{tp}
\DeclareMathOperator{\ET}{ET}
\DeclareMathOperator{\D}{D}
\DeclareMathOperator{\Sh}{Sh}

\DeclareMathOperator{\az}{alph}

\DeclareMathOperator{\comp}{comp}
\DeclareMathOperator{\var}{var}
\DeclareMathOperator{\Con}{Cone}
\DeclareMathOperator{\Cay}{Cay}
\DeclareMathOperator{\lk}{link}

\DeclareMathOperator{\dist}{dist}

\DeclareMathOperator{\CAT}{CAT}

\newcommand{\Epsilon}{\mathcal E}

\renewcommand{\AA}{\mathfrak{A}}

\renewcommand{\P}{\mathcal{P}}

\newcommand{\Ts}{\mathfrak{T}}

\newcommand{\m}{\mathfrak{m}}

\newcommand{\BS}{\mathcal{BS}}

\newcommand{\BN}{\mathbb{N}}
\newcommand{\BR}{\mathbb{R}}
\newcommand{\BZ}{\mathbb{Z}}
\newcommand{\factor}[2]{{\raise0.7ex\hbox{$#1$} \!\mathord{\left/ {\vphantom {#1 {#2}}}\right.\kern-\nulldelimiterspace}\!\lower0.7ex\hbox{${#2}$}}}
\newcommand{\KK}{\ensuremath{\mathbb{K}}}
\newcommand{\GG}{\ensuremath{\mathbb{G}}}
\newcommand{\EE}{\ensuremath{\mathbb{E}}}
\newcommand{\PP}{\ensuremath{\mathbb{P}}}
\newcommand{\HH}{\ensuremath{\mathbb{H}}}
\newcommand{\FF}{\ensuremath{\mathbb{F}}}
\newcommand{\Tr}{\ensuremath{\mathbb{T}}}

\newcommand{\ov}{\overline}
\newcommand{\BA}{\ensuremath{\mathbb{A}}}

\newcommand{\lra }{\leftrightarrows}

\newcommand{\CC}{\mathcal{C}}

\newcommand{\UG}{\GG^{\UU}}

\newcommand{\Hc}{\mathcal{H}}
\newcommand{\mm}{\mathfrak{m}}

\renewcommand{\b}{\mathfrak{b}}
\newcommand{\mmm}{\mathfrak{M}}
\newcommand{\cali}{\mathfrak{l}}

\newcommand{\caN}{\mathcal{N}}
\newcommand{\caM}{\mathcal{M}}
\newcommand{\calL}{\mathcal{L}}
\newcommand{\caK}{\mathcal{K}}

\title[Group Actions on Real Cubings]{Group Actions on Real Cubings and Limit Groups over Partially Commutative Groups}

\author[M. Casals-Ruiz]{Montserrat Casals-Ruiz}
\address{Department of Mathematics, 1326 Stevenson Center, Vanderbilt University, Nashville, TN 37240, USA}
\email{montsecasals@gmail.com}
\thanks{The first author is supported by Programa de Formaci\'{o}n de Investigadores del Departamento de Educaci\'{o}n, Universidades e Investigaci\'{o}n del Gobierno Vasco}

\author[I. Kazachkov]{Ilya Kazachkov}
\address{Department of Mathematics, 1326 Stevenson Center, Vanderbilt University, Nashville, TN 37240, USA}
\email{ilya.kazachkov@gmail.com}
\thanks{The second author is supported by NSERC PostDoctoral Fellowship}

\begin{document}

\begin{abstract}
We introduce a class of spaces, called real cubings, and study the stucture of groups acting nicely on these spaces. Just as cubings are a natural generalisation of simplicial trees, real cubings can be regarded as a natural generalisation of real trees. Our main result states that a finitely generated group $G$ acts nicely (essentially freely and co-specially) on a real cubing if and only if it is a subgroup of a graph tower (a higher dimensional generalisation of $\omega$-residually free towers and NTQ-groups). It follows that $G$ acts freely, essentially freely and co-specially on a real cubing if and only if $G$ is a subgroup of the graph product of cyclic and (non-exceptional) surface groups. In the particular case when the real cubing is a tree, it follows that $G$ acts freely,  essentially freely and co-specially on the real cubing if and only if it is a subgroup of the free product of abelian and surface groups. Hence, our main result can be regarded as a generalisation of the Rips' theorem on free actions on real trees. 

We apply our results to obtain a characterisation of limit groups over partially commutative groups as subgroups of graph towers. This result generalises the work of Kharlampovich-Miasnikov, \cite{KhMNull}, Sela, \cite{Sela1} and Champetier-Guirardel, \cite{CG} on limit groups over free groups.
\end{abstract}

\maketitle

\newpage
\setcounter{page}{2}
\section*{Table of Contents}
\renewcommand{\contentsname}{}
\tableofcontents
\newpage
\section{Introduction} \label{sec:1}

Rips' theory of group actions on real trees has been extremely important not only because of its intrinsic value but also as a tool and a nexus for different disciplines, see \cite{Bestv}. For instance, it was essential  for understanding the compactification of hyperbolic structures \cite{MS1, MS2, MS3}, in Sela's approach to acylindrical accessibility \cite{SelaAcy}, for the isomorphism problem for hyperbolic groups \cite{SelaIso, DaGr1, DGuir}, for the study of limit groups \cite{Sela1}, for the construction of the JSJ decomposition \cite{RSJSJ}, for analysing automorphisms of hyperbolic groups and free groups, \cite{Paulin1, Paulin2, SelaAut}, etc. 

The goal of this paper is to generalise Rips' theory to higher dimensional spaces, i.e. to determine the structure of finitely generated groups acting nicely on \emph{real cubings} (see Definition \ref{defn:rcub}).

In order to be more concrete on the type of actions and spaces under consideration, we first discuss the point of view on group actions on real trees that we adopt. In \cite{BF}, Bestvina and Feighn determined the structure of finitely presented groups acting stably on a real tree. Guirardel, in \cite{Guir}, showed that, in fact, any stable action of a finitely presented group on a real tree can be approximated by actions on simplicial trees while keeping control on arc stabilisers. From this perspective, the theory of (stable) actions on real trees is the theory of ultralimits of actions on simplicial trees, \cite{KL,Roe}.

A natural high dimensional generalisation of a (simplicial) tree is that of a $\CAT(0)$ cube complex. Following Sageev, \cite{S95}, we call a simply connected cube complex of (combinatorially) non-positive curvature a \emph{cubing} and we define a \emph{real cubing} to be an ultralimit of cubings (see Definition \ref{defn:rcub}). As Rips' theory can be viewed as the theory of ultralimits of actions on simplicial trees, our approach to its generalisation is to study group actions on the ultralimits of cubings.

It is obvious that in order to succeed in this approach, one needs to have a good understanding of the structure of groups acting on $\CAT(0)$ cube complexes. But already at this point, one encounters the first obstacle: Burger and Mozes, see \cite{BM}, described a series of finitely presented simple groups that act \emph{freely} on the direct product of two (simplicial) trees. These examples show that the problem of describing the structure of finitely generated groups acting on $\CAT(0)$ cube complexes is practically hopeless.

Intuitively, one would expect that a ``good'' action on a space should capture the structure of the space and so, in the particular case of a direct product of trees, one would expect to recover some type of direct product structure. As shown by the examples of Burger and Mozes, free actions fail this intuition. The notion that does respond to this expectation is the notion of co-special action introduced by Haglund and Wise, see \cite{HW}. Roughly speaking, a group $G$ acts co-specially on a cube complex $X$ if the quotient $G\backslash X$ is a cube complex without certain pathologies on its hyperplanes. In this way, a group acts freely co-specially on a direct product of trees if and only if it is a subgroup of a direct product of free groups. More generally, it follows from the work of Haglund and Wise, that a group acts freely co-specially on a $\CAT(0)$ cube complex if and only if it is a subgroup of a (possibly infinitely generated) partially commutative group (also known as right-angled Artin groups or graph groups).  Furthermore, every action of a group on a simplicial tree (without inversions of edges) is co-special, therefore co-special actions on $\CAT(0)$ cube complexes are a natural generalisation of actions on simplicial trees.

The aim of this paper is to determine the structure of finitely generated groups that act on real cubings via an ultralimit of free co-special actions on cube complexes, i.e. to determine the structure of finitely generated groups that act essentially freely co-specially on real cubings, see Definition \ref{defn:essfr}. 

We show that the class of groups that admit this type of action on a real cubing is, in fact, very large, containing as main example the class of finitely generated groups discriminated by an arbitrary (finitely generated) partially commutative group, or in other words, the class of limit groups over partially commutative groups. In particular, the class of groups we consider contains the class of finitely generated residually free groups and so the class
of limit groups over free groups.

Furthermore, the type of actions we consider is rather wide. In the case of real trees (viewed as real cubings), one can show that every very small action of a free group, every faithful, non-trivial action of an abelian group and every small action of the fundamental group of an orientable surface is of our type. (Note that if one could show that very small actions of the fundamental group of a non-orientable surface on real trees can be approximated by very small simplicial actions, then we could conclude that any free action of a finitely generated group on a real tree is of our type, see \cite{Guir} for further discussion).

We now explain in what terms we determine the structure of finitely generated groups acting essentially freely and co-specially on real cubings. One family of groups that turned out to be key in the study of both the universal and the elementary theory of free groups is the class of $\omega$-residually free towers or, equivalently, the class of NTQ-groups. This class of groups is constructed inductively from free abelian groups and surface groups by taking free products and amalgamations over abelian subgroups. One of the main properties of this family is that, on the one hand, any NTQ-group is discriminated by non-abelian free groups (see \cite{KhMNull, Sela1}) and, on the other hand, any finitely generated freely discriminated group is a subgroup of an NTQ-group (see \cite{KhMNull, CG}). Hence studying limit groups over free groups is equivalent to  studying NTQ-groups and their subgroups. 

Along these lines, we prove that groups acting essentially freely co-specially on real cubings are subgroups of a nice class of groups, which we call \emph{graph towers}. Graph towers are natural generalisations of NTQ-groups ($\omega$-residually free towers) and, just as them, they are also defined inductively. Graph towers of height $0$ are partially commutative groups. A graph tower $\Ts^l$ of height $l$ is constructed by taking an amalgamated product of a graph tower $\Ts^{l-1}$ of height $l-1$ and an extension of either a free abelian group, or of a free group or of a surface group with boundary, by the centraliser $C_{\Ts^{l-1}}(D)$ of (a certain) subgroup $D$ of $\Ts^{l-1}$, and taking the amalgamation over this centraliser $C_{\Ts^{l-1}}(D)$. More precisely,

\begin{defn*}[see Lemma \ref{lem:prtower}]
Let partially commutative groups be graph towers of height $0$. Assume that graph towers $\Ts^{l-1}$ of height $l-1$ have been constructed. Then, a graph tower of height $l$ has one of the following presentations:
\begin{itemize}
\item [a1)] $\Ts^{l-1} \ast_{C_{\Ts^{l-1}}(D)} (C_{\Ts^{l-1}}(D)\times\langle x_1^l, \dots, x_{m_l}^l \rangle)$ \\
(basic type, $D$ is a certain non-abelian subgroup of $\Ts^{l-1}$);

\item [a2)] $\Ts^{l-1} \ast_{C_{\Ts^{l-1}}(D)} (C_{\Ts^{l-1}}(D) \times \langle x_1^l, \dots, x_{m_l}^l \mid [x_i^l,x_j^l]=1, 1\le i,j \le m_l, i\ne j \rangle)$ \\ (basic type, $D$ is a certain abelian subgroup of $\Ts^{l-1}$);

\item [b1)] $\Ts^{l-1} \ast_{C_{\Ts^{l-1}}(u)} (C_{\Ts^{l-1}}(u) \times \langle x_1^l, \dots, x_{m_l}^l \mid [x_i^l,x_j^l]=1, 1\le i,j \le m_l, i\ne j \rangle)$ \\ (abelian type, $u$ is a non-trivial irreducible root element);

\item [b2)] $\Ts^{l-1} \ast_{C_{\Ts^{l-1}}(D)} (C_{\Ts^{l-1}}(D) \times \langle x_1^l, \dots, x_{m_l}^l \mid [x_i^l,x_j^l]=1, 1\le i,j\le m_l, i\ne j \rangle)$ \\ (abelian type, $D$ is a certain non-abelian subgroup of $\Ts^{l-1}$);

\item [c)] $\Ts^{l-1} \ast_{C_{\Ts^{l-1}}(D)\times \langle u_{2g+1},\dots, u_m\rangle} ( \langle u_{2g+1}, \dots, u_m, x_1^l, \dots, x_{m_l}^l \mid W \rangle \times C_{\Ts^{l-1}}(D))$ \\ (surface type, $W$ is a non-exceptional quadratic equation and $D$ is a certain non-abelian subgroup of $\Ts^{l-1}$).
\end{itemize}
\end{defn*}

In general, one does not extend centralisers of arbitrary subgroups. Roughly speaking, one only considers subgroups of $\Ts^{l-1}$ which behave as directly indecomposable canonical parabolic subgroups, i.e. are discriminated (in a ``minimal way'') into directly indecomposable canonical parabolic subgroups of the partially commutative group $\GG$, see discussion in Section \ref{sec:9}. 

The class of graph towers extends the class of $\omega$-residually free towers or NTQ-groups. Indeed, if the subgroup $D$  of an $\omega$-residually free tower is non-abelian, then since centralisers in a freely discriminated group $\Ts^{l-1}$ are commutative transitive, it follows that $C_{\Ts^{l-1}}(D)$ is trivial and so, in this case, the decompositions correspond to free products and amalgamated products (over cyclic). If the subgroup $D$ is abelian, then $C_{\Ts^{l-1}}(D)$ is a maximal abelian subgroup of $\Ts^{l-1}$ and the decomposition corresponds to an amalgamated product over a (finitely generated free) abelian group. This suggests that the splittings that occur in groups acting on real cubings can give a framework for a generalisation of the theory of (abelian) JSJ-decompositions.

Just as the class of $\omega$-residually free towers and NTQ groups plays a crucial role in the classification of groups elementarily equivalent to a free group, \cite{KhMTar, SelaTar}, we expect that the class of graph towers will play a similar role in the classification of groups elementarily equivalent to a given partially commutative group.

We are now in the position to state one of the main results of this paper.

\begin{thmmain}
Let $G$ be a finitely generated group. The group $G$ acts essentially freely co-specially on a real cubing if and only if it is a subgroup of a graph tower.
\end{thmmain}
Furthermore, given a finitely generated group $G$, the corresponding graph tower $\Ts$ and the embedding of $G$ into $\Ts$ can be constructed effectively.

In the case of free actions, the above theorem results in the following corollary, which can be likened to Rips' theorem on free actions on real trees.
\begin{cormain}
A finitely generated group $G$ acts freely, essentially freely and co-specially on a real cubing if and only if $G$ is a subgroup of the graph product of free abelian and {\rm(}non-exceptional{\rm)} surface groups.

In particular, if the real cubing is a real tree, then $G$ is a {\rm(}subgroup of{\rm)} the free product of free abelian groups and {\rm(}non-exceptional{\rm)} surface groups.
\end{cormain}

Notice that although graph towers are finitely presented, finitely generated groups acting essentially freely co-specially on real cubings (or limit groups over partially commutative groups) are, in general, \emph{not} finitely presented, since partially commutative groups (and so graph towers) are not coherent.

\medskip

The approach to group actions on real cubings we present in this paper is similar to the one originally presented by Rips for the study of free actions on real trees and to the one presented by Kharlampovich and Miasnikov for the study of limit groups over free groups, \cite{KhMNull}. Both approaches are based on an interpretation of the Makanin-Razborov process, a combinatorial algorithm that decides whether or not a system of equations over a free group is compatible as well as produces an effective description of the set of solutions. In our case, the proof will rely on the analogue of the Makanin-Razborov process that we developed in \cite{CKpc}, which gives an effective description of the solution set of a system of equations over a partially commutative group.

\subsection*{Organization of the paper}
We now outline the organization of the paper. In the first three sections, we review basic notions and fix notation on partially commutative groups, algebraic geometry over groups and cube complexes.

In Section \ref{sec:5}, we review and extend several basic results from the theory of cube complexes. An important notion for our approach is that of co-special action introduced by Haglund and Wise in \cite{HW}, since it links cube complexes and groups acting on them to partially commutative groups and their subgroups - the natural setting for our work. More precisely, given a special cube complex $X$, Haglund and Wise construct an ($A$-)typing map and a natural combinatorial local isometry from $X$ to the standard complex of a partially commutative group $\GG(X)$,  and prove that $\pi_1(X)$ embeds into $\GG(X)$. The main inconvenience of this construction is that, in general, the partially commutative group $\GG(X)$ is infinitely generated and this precludes the use of some of the techniques needed in our approach. 

An elegant way to overcome this technical obstacle would be by solving the following problem for partially commutative groups.

\begin{prob} \label{prob:1}
Given $k\in \BN$, does there exist a universal finitely generated partially commutative group $\GG_k$ that contains all $k$-generated subgroups of arbitrary partially commutative groups? 

Or, in a stronger form, given $k\in \BN$, does there exist a finitely generated partially commutative group $\GG_k$ so that for any $k$-generated subgroup $H$ of any partially commutative group $\GG$ there exists a homomorphism $\varphi_{H}$ from $\GG$ to $\GG_k$ which is injective on $H$?
\end{prob}

We note that by \cite{Baudisch}, every $2$-generated subgroup of any partially commutative group is either free or free abelian. Hence, there does exist a $2$-universal partially commutative group.

Unfortunately, we do not know a solution to this problem in the full generality. However, we can solve it if we restrict the family of partially commutative groups to which the subgroups $H$ belong. To this end, in Section \ref{sec:5}, we introduce the notion of width of a cube complex and show that the subclass of partially commutative groups of bounded width do have a universal object for the set of their $k$-generated subgroups. Note that there are other (even wider) classes that one could handle in a similar way such as the class of partially commutative groups that are (finite or infinite) free products of finitely many different partially commutative groups of finite width.

Due to the nature of the class of partially commutative groups we are considering, some constructions in Section \ref{sec:5} might seem somewhat artificial. The reader interested primarily in understanding the structure of limit groups over partially commutative groups, may wish to skip Sections \ref{sec:5}, \ref{sec:6} and \ref{sec:7} and proceed to Section \ref{sec:9}. Our choice of approach is justified by the fact that, on the one hand, it covers the case of trees, and, on the other hand, in the event that Problem \ref{prob:1} is solved in positive, the proofs and results of this paper could be straightforwardly carried over to the general case (i.e. the restriction on the width of cubings would be dropped).

In Section \ref{sec:6}, we introduce real cubings as ultralimits of cubings of uniformly bounded width, see Definition \ref{defn:rcub}. 

In Section \ref{sec:7}, we introduce the notion of essentially free co-special action - the class of actions that we are interested in, see Definition \ref{defn:essfr}. Roughly speaking, an action of a group $G$ on a real cubing (viewed as an ultralimit) is essentially free and co-special if it is faithful, non-trivial and is a limit of co-special actions. We show that if a group $G$ acts essentially freely and co-specially on a real cubing, then it acts essentially freely and co-specially on the asymptotic cone of a finitely generated partially commutative group.  We then reduce the study of groups acting essentially freely and co-specially on real cubings to the study of limit groups over partially commutative groups. The main result of this section is summarised in the following
\begin{thm*}[see Theorem \ref{thm:gen}]
A finitely generated group acts essentially freely and co-specially on a real cubing if and only if it acts essentially freely and co-specially on an asymptotic cone of a partially commutative group if and only if it is a limit group over some partially commutative group.
\end{thm*}

Just as the asymptotic cone of a non-abelian free group is the universal real tree, \cite{DP, MNO}, the above results suggest that the asymptotic cone of the partially commutative group $\GG_N$ is the universal real cubing for real cubings of width $N$. From this perspective, as Rips' theory can be viewed as the theory of groups acting on subspaces of the asymptotic cone of a free group, the theory developed in this paper can be regarded as the theory of groups acting on subspaces of asymptotic cones of partially commutative groups.

In Section \ref{sec:9}, we introduce the notion of graph tower, see Definition \ref{defn:tower}. As we briefly discussed, graph towers are a natural generalisation of $\omega$-residually free towers and NTQ-groups.

In Section \ref{sec:10}, we  construct a graph tower $\Ts_G$ associated to a limit group $G$ over a partially commutative group $\GG$ and in Section \ref{sec:11}, we show that $\Ts_G$ is discriminated by a family of homomorphisms induced by a discriminating family of $G$ and that $G$ embeds into $\Ts_G$.  The main result of these two sections can be formulated as follows
\begin{thm*}[see Corollary \ref{cor:limit}]
Given a limit group $G$ over a partially commutative group $\GG$, one can effectively construct a graph tower $\Ts_G$ and a monomorphism $i:G \to \Ts$. Furthermore, the graph tower $\Ts_G$ is discriminated by $\GG$.
\end{thm*}

Given a limit group $G$ over a  partially commutative group $\GG$, the construction of the graph tower $\Ts_G$ and the embedding $i:G\hookrightarrow \Ts_G$ essentially relies on the process we constructed in \cite{CKpc}, which can be regarded as the counterpart of the Makanin-Razborov process for free groups. In \cite{CKpc}, we show that the set of homomorphisms $\Hom(G, \GG)$ from a finitely generated group $G$ to a partially commutative  group $\GG$ can be (effectively) described by a finite, oriented, rooted tree $T_{\sol}$. Vertices $v$ of the tree are labelled by coordinate groups $G_{R(\Omega_v)}$ of generalised equations $\Omega_v$ and certain groups  $A(\Omega_v)$ of automorphisms of $G_{R(\Omega_v)}$. Edges $e:u\to v$ of the tree are labelled by proper epimorphisms $\pi(u,v):G_{R(\Omega_u)} \to G_{R(\Omega_v)}$. 

Any homomorphism from $G$ to $\GG$ determines a path $v_0\to v_1\to \dots \to v_k$ from the root to a leaf of this tree that allows to describe the homomorphism as a composition
$$
\sigma_0 \pi(v_0,v_1)\sigma_1 \pi(v_1,v_2)\cdots \sigma_{k-1}\pi(v_{k-1},v_k)\psi,
$$
where $\sigma_i\in A(G_{v_i})$ is an automorphism from the corresponding group assigned to the vertex $v_i$, $\pi(v_i,v_{i+1})$ is the epimorphism that labels the edge $e:v_i\to v_{i+1}$ and $\psi$ is a certain homomorphism from the group assigned to the leaf of the tree to $\GG$. (Note that both the group associated to a leaf and the set of homomorphisms from the leaf to $\GG$ are completely determined, see \cite[Theorem 9.2]{CKpc}).

In the free group case, Rips' machine allows one to analyse the dynamics of the foliated band complex and classify the corresponding measured foliations by taking them to a normal form, i.e. it allows one to show that minimal pieces of the corresponding action of the group on a real tree correspond to either thin, surface or toral case. 

In the case of partially commutative groups, the process we describe analyses the dynamics of a ``foliated constrained band complex'' or, in our terms, the dynamics induced by the family of solutions of a constrained generalised equation. Constrained generalised equations are generalised equations (or union of bands) with commutation constraints associated to the items. In other words, if the coordinate group associated to the generalised equation is defined to be the free group generated by the items quotient by the radical ideal generated by the set of relations, then the coordinate group associated to a constrained generalised equation is the quotient of some partially commutative group (generated by the items and commutation relations being induced by the constraints) by the radical ideal of the set of relations.

Since the tree $T_{\sol}(G)$ that describes the set $\Hom(G, \GG)$ is finite and since $G$ is discriminated by $\GG$, it follows (see Lemma \ref{lem:discfam}) that there exists a branch of the tree, called fundamental branch, so that the family of homomorphisms that factors through it, is a discriminating family for $G$. The idea is to use a fundamental branch and the family of automorphisms associated to the vertex of the branch to construct a graph tower $\Ts$ as well as the embedding of $G$ into $\Ts$. 

As in the free group case, the family of automorphisms associated to a vertex $v$ are intrinsically related to the types of dynamics of a constrained generalised equation associated to $v$. These dynamics were studied in \cite{CKpc} and can be classified into the following three types:
\begin{itemize}
\item linear-like (thin-like) case: case 7-10;
\item quadratic-like (surface-like) case: case 12;
\item general type (a combination of surface-like and axial-like cases): case 15.
\end{itemize}

If the automorphism group associated to a vertex is of abelian type, i.e. if the generalised equation associated to the vertex contains a periodic structure, then similarly to the free group case, one can show that there exists an automorphism of the coordinate group that brings it to a presentation that exhibits the splitting (see abelian floor in Definition \ref{defn:tower}).

There is an essential difference when studying the linear and quadratic cases over free groups or over partially commutative groups. In the free group case, given a generalised equation $\Omega$ of type 12 (quadratic type), the coordinate group associated to $\Omega$ is isomorphic to a surface group. Basically, one shows that bases of the generalised equation are generators of the corresponding group and that the group (in this new generating set) is a 1-relator group where the relation is a quadratic word. One concludes that there exists an automorphism of the group that brings the quadratic word into the normal form and hence the group is isomorphic to a surface group. In this case, the isomorphism only relies on the type of generalised equation.

This is not true in the case of partially commutative groups. The type of the generalised equation alone does not determine the structure of the corresponding coordinate group: in general, the coordinate group of a constrained generalised equation of type 12 is \emph{not} isomorphic to the quotient of a partially commutative group by a surface relation. In order to determine the structure of the group one essentially uses the fact that the corresponding generalised equation of type 12 repeats infinitely many times along an infinite branch of the tree $T(\Omega)$ constructed in \cite{CKpc}. In other words, the structure of the coordinate group associated to a generalised equation $\Omega_v$ of type 12 is intimately related with the presence of the automorphism group associated to the vertex $v$.

The analysis of the dynamics of an infinite branch of type 12 shows that one can gain control on the type of constraints and prove that, in this case, we can take the set of bases as generators of the group and preserve the underlying structure of a partially commutative group as well as prove the existence of an automorphism of the group that brings the quadratic word into the normal form. Therefore, if a constrained generalised equation of type 12 appears infinitely many times in an infinite branch of the tree $T(\Omega)$, then one can conclude that the group is isomorphic to the quotient of a partially commutative group by a surface relation (see surface floor in Definition \ref{defn:tower}).

A similar phenomena occurs if the generalised equation is of type 7-10 (linear type). In the case of free groups, the coordinate group splits as a free product where one of the factors is a free group. In the case of partially commutative groups, again the existence of automorphisms is crucial to determining the splitting. However, the dynamics of the linear case are much harder to analyse and we did not manage to follow this approach in order to find the splitting. Instead, we change the strategy and study this case by forcing it to behave as the cases we know how to deal with: the quadratic case and the abelian case. In order to carry out this new analysis of the branch, we do the contrary of what is done in the free group case: we use Tietze transformations to \emph{introduce} new  generators and relations and transform the generalised equation to a generalised equation where every item is covered at least twice and then study the obtained generalised equation as in the quadratic and general cases. 

As one might expect, the dynamics of the linear case, not always can be analysed in this way and it may happen that after a finite number of steps, we fall again into a linear case. To deal with this, we introduce an invariant (the height of the minimal tribe), which is bounded above by a constant that depends only on the underlying partially commutative group $\GG$, and show that every time we go back to the linear case, this invariant increases. Furthermore,  if this invariant is maximal, we prove the existence of an automorphism that brings the group to a presentation that exhibits the splitting (see basic floor in Definition \ref{defn:tower}).

In the last section, we generalise the construction of the graph tower associated to a fundamental branch to an arbitrary branch. This allows us to prove the following result.
\begin{thmsollast} 
Let $\GG$ be a partially commutative group and let $G$ be a finitely generated residually $\GG$ group. Then, one can effectively construct finitely many fully residually $\GG$ graph towers $\Ts_1,\dots,\Ts_k$ and homomorphisms $p_i$ from $G$ to $\Ts_i$, $i=1,\dots,k$ so that any homomorphism from $G$ to $\GG$ factors through a graph tower $\Ts_i$, for some $i=1,\dots,k$, i.e. for any homomorphism $\varphi:G\to \GG$ there exists $i\in \{1,\dots,k\}$ and a homomorphism $\varphi_i:\Ts_i\to \GG$ so that $\varphi=p_i\varphi_i$. In particular, $G$ is a subgroup of the direct product of the graph towers $\Ts_i$, $i=1,\dots,k$ and a subdirect product of the direct product of groups $p_i(G)<\Ts_i$, $i=1,\dots, k$.
\end{thmsollast}

In the case of finitely generated residually free group, an analogous result was proven by Kharlampovich and Miasnikov, \cite{KhMNull}. In a recent work on the structure of finitely presented residually free groups, Bridson, Howie, Miller and Short, \cite{BHMS}, gave an alternative construction of the embedding of a finitely presented residually free group into the direct product of finitely many limit groups and showed that their construction is canonical.

\subsubsection*{Acknowledgement} The authors would like to thank Mark Sapir and Vincent Guirardel for stimulating discussions.

\section{Partially commutative groups}\label{sec:2}
In this section we recall some preliminary results on partially commutative groups and introduce the notation we use throughout the text.

Let $\Gamma=(V(\Gamma), E(\Gamma))$ be a (undirected) simplicial graph. Then, the partially commutative group $\GG=\GG(\Gamma)$ defined by the (commutation) graph $\Gamma$ is the group given by the following presentation:
$$
\GG=\langle V(\Gamma)\mid [v_1,v_2]=1, \hbox{ whenever } (v_1,v_2)\in E(\Gamma)\rangle.
$$
We note that $\GG$ is not necessarily finitely generated. 

Let $\Gamma'=(V(\Gamma'), E(\Gamma'))$ be a full subgraph of $\Gamma$. It is not hard to show, see for instance \cite{EKR}, that the partially commutative group $\GG'=\GG(\Gamma')$ is the subgroup  of $\GG$ generated by $V(\Gamma')$, i.e. $\GG(\Gamma')=\langle V(\Gamma')\rangle$. Following \cite{DKRpar}, we call $\GG(\Gamma')$ a \emph{canonical parabolic subgroup} of $\GG$. 

We denote the length of a word $w$ by $|w|$. For a word $w \in \GG$, we denote by $\ov{w}$ a \emph{geodesic} of $w$. Naturally, $|\ov{w}|$ is called the length of an element $w\in \GG$. An element $w\in \GG$ is called \emph{cyclically reduced} if the length of $\ov{w^2}$ is twice the length of $\ov{w}$ or, equivalently, the length of $w$ is minimal in the conjugacy class of $w$.

For a given word $w$, denote by $\az(w)$ the set of letters occurring in $w$. For a  word $w\in \GG$, define $\BA(w)$ to be the subgroup of $\GG$ generated by all letters that do not occur in a geodesic $\ov w$ and commute with $w$.  The subgroup $\BA(w)$ is well-defined (independent of the choice of a geodesic $\ov w$), see \cite{EKR}. Let $v,w\in \GG$ be so that $[v,w]=1$ and $\az(v)\cap \az(w)=\emptyset$, or, which is equivalent, $v\in \BA(w)$ and $w\in \BA(v)$. In this case, we say that $v$ and $w$ \emph{disjointly commute} and write $v\lra  w$. Given a set of elements $S$ of $\GG$, define $\BA(S)=\bigcap\limits_{w\in S} \BA(w)$.

For a (not necessarily finitely generated) partially commutative  group $\GG(\Gamma)$, consider its non-commutation graph $\Delta=(V(\Delta), E(\Delta))$ defined as follows. The vertex set $V(\Delta)$ coincides with $V(\Gamma)$. There is an edge connecting $v_i$ and $v_j$ in $\Delta$ if and only if $i\ne j$ and there is no edge connecting $v_i$ and $v_j$ in $\Gamma$. Note that the graph $\Delta$ is the complement graph of the graph $\Gamma$. The graph $\Delta$ is a union of its connected components $I_1, \ldots , I_k$, which induce a decomposition of $\GG$ as the direct product
$$
\GG= \GG(I_1) \times \cdots \times \GG(I_k).
$$

Given $w \in \GG$ and the set $\az(w)$, just as above, consider the graph $\Delta (\az(w))$ (which is a full subgraph of $\Delta$). This graph can be either connected or not. If it is connected, we call $w$ a \emph{block}. If $\Delta(\az(w))$ is not connected, then we can decompose $w$ into the product
\begin{equation} \label{eq:bl}
w= w_{j_1} \cdot w_{j_2} \cdots w_{j_t};\ j_1, \dots, j_t \in J,
\end{equation}
where $|J|$ is the number of connected components of $\Delta(\az(w))$ and the word $w_{j_i}$ is a word in the letters from the $j_i$-th connected component. Clearly, the words $\{w_{j_1}, \dots, w_{j_t}\}$ pairwise disjointly commute. Each word $w_{j_i}$, $i \in {1, \dots,t}$ is a block and so we refer to presentation (\ref{eq:bl}) as the block decomposition of $w$.

Observe that the number of blocks of the block decomposition of $w\in \GG$ is bounded above by the rank of $\GG$.
An element $g\in \GG$ is called \emph{irreducible} if it is a conjugate of a cyclically reduced block element.

\begin{rem*}
Irreducible elements play a very important role in the theory of partially commutative groups. They are crucial in describing centralisers of elements; they are used to prove that free extensions of centralisers are discriminated by partially commutative groups; they are key to understand the cut-points in the asymptotic cones of partially commutative groups;  and they will be essentially used in this paper. In a sense, irreducible elements can be likened to irreducible automorphisms of free groups or surface groups.
\end{rem*}

An element $w\in \GG$ is called a \emph{least root} (or simply, root) of $v\in \GG$ if there exists a positive integer $1< m\in \BN$ such that $v=w^m$ and there does not exists $w'\in \GG$ and $1< m'\in \BN$ such that $w={w'}^{m'}$. In this case, we write $w=\sqrt{v}$. By a result from \cite{DK}, partially commutative groups have least roots, that is the root element of $v$ is defined uniquely.

Let $\Gamma$ be a simplicial graph. For any $x\in V(\Gamma)$, define $x^\perp$ to be the subset of all vertices $y\in V(\Gamma)$ so that there is an edge $(x,y)\in E(\Gamma)$. We note that $x\notin x^\perp$. Given a subset $X\subseteq V(\Gamma)$, set $X^\perp=\bigcap\limits_{x\in X} x^\perp$.

Let $w$ be a cyclically reduced element of $\GG$. It is not hard to see that $\langle\az(w)^\perp\rangle =\BA(w)$.

Introduce an equivalence relation $\sim$ on the set of vertices $V(\Gamma)$. For two vertices $v_1,v_2\in V(\Gamma)$, set $v_1\sim v_2$ if and only if $v_1^\perp=v_2^\perp$.  Since for every $v\in \Gamma$, we have  $v \notin v^\perp$, it follows that if $v_1\sim v_2$, then they are not connected by an edge in $\Gamma$. Define the graph $\Gamma'$ whose vertices are $\sim$-equivalence classes and there is an edge joining $[u]$ to $[v]$ if and only if $(u',v')$ is an edge of $\Gamma$ for some (and thus for all) $u' \in [u]$ and $v' \in [v]$.  The graph $\Gamma'$ is called the \emph{deflation} of $\Gamma$. Observe that the partially commutative group $\GG(\Gamma')$ is isomorphic to a canonical parabolic subgroup of $\GG(\Gamma)$.

\begin{defn}
A canonical parabolic subgroup $\KK$ of a partially commutative group is called \emph{closed} if $\KK^{\perp \perp}=\KK$. The subgroup $\KK$ is called \emph{\cool } if $\KK$  is closed and its complement $\KK^{\perp}$ is a directly indecomposable canonical parabolic subgroup. We denote by $\az(\KK)$ the set of canonical generators of $\GG$ that generate $\KK$.
\end{defn}

In our setting, the set of edges $E(\Gamma)$ of the graph $\Gamma$  will be decomposed into a disjoint union of two sets, $E(\Gamma)=E_d(\Gamma)\cup E_c(\Gamma)$, $E_d(\Gamma)\cap E_c(\Gamma)=\emptyset$. Let $\GG_d$ and $\GG_c$ be the partially commutative groups defined by the graphs $(V(\Gamma),E_d(\Gamma))$ and $(V(\Gamma),E_c(\Gamma))$, correspondingly. Then, any canonical parabolic subgroup $\KK$ of $\GG$ naturally defines canonical parabolic subgroups $\KK_d$ and $\KK_c$ of $\GG_d$ and $\GG_c$, correspondingly. 

We say that a canonical parabolic subgroup $\KK$ of $\GG$ is $E_d(\Gamma)$-\cool ($E_d(\Gamma)$-directly (in)decomposable) if so is the induced subgroup $\KK_d$ of $\GG_d$.  Similarly, a subgroup $\KK$ is called $E_c(\Gamma)$-abelian if so is the induced subgroup $\KK_c$ of $\GG_c$. 

The canonical parabolic subgroup $\BA(\KK_d)$ of $\GG_d$ naturally defines a canonical parabolic subgroup of $\GG$, which we shall denote by $\BA_{E_d(\Gamma)}(\KK_d)$.

Let $\GG$ be a partially commutative group given by the presentation $\langle \cA\mid R\rangle$. Let $\FF=\FF(\cA^{\pm 1})$ be the free monoid on the alphabet $\cA\cup \cA^{-1}$ and let $\Tr=\Tr(\cA^{\pm 1})$ be the partially commutative monoid with involution given by the presentation:
$$
\Tr(\cA^{\pm 1})=\langle \cA\cup \cA^{-1}\mid R_\Tr \rangle,
$$
where $\left[ a_i^{\epsilon}, a_j^\delta \right]\in R_\Tr$ if and only if $[a_i,a_j]\in R$, $\epsilon,\delta\in \{-1,1\}$. The involution on $\Tr$ is induced by the operation of inversion in $\GG$ and does not have fixed points. We refer to it as to the \emph{inversion} in $\Tr$ and denote it by $^{-1}$.

Following \cite{DM}, we define a \emph{clan} to be a maximal subset $C=\mathcal{C}\cup\mathcal{C}^{-1}$ of $\cA \cup \cA^{-1}$ such that $[a,c]\notin R_\Tr$ if and only if $[b,c]\notin R_\Tr$ for all $a,b\in  C$ and $c\in \cA^{\pm 1}$. A clan $C$ is called  \emph{thin} if there exist $a\in C$ and $b\in \cA^{\pm 1}\setminus C$ such that $[a,b]\in R_\Tr$ and is called \emph{thick}, otherwise. It follows from the definition that there is at most one thick clan and that the number of thin clans never equals 1.

Every element of $\cA\cup \cA^{-1}$ belongs to exactly one clan. If $\Tr$ is a direct product of $d$ free monoids, then the number  of thin clans is $d$ for $d>1$, and it is $0$ for $d=1$. In the following, we pick a thin clan and we make it thick by removing commutation. It might be that the number of clans does not change, but the number of thin clans decreases. This is the reason why the definition of DM-normal form below is based on thin clans (instead of considering all clans).

It is convenient to encode an element of the partially commutative monoid as a finite labelled acyclic oriented graph $[V,E,\lambda]$, where $V$ is the set of vertices, $E$ is the set of edges and $\lambda:V\to \cA^{\pm 1}$ is the labelling. Such a graph induces a labelled partial order $[V, E^*,\lambda]$. For an element $w\in \Tr$, $w=b_{1}\cdots b_{n}$, $b_{i}\in \cA^{\pm 1}$, we introduce the graph $[V,E, \lambda]$ as follows. The set of vertices of $[V,E, \lambda]$ is in one-to-one correspondence with the letters of $w$, $V=\{1,\dots, n\}$. For the vertex $j$ we set $\lambda(j)=b_{j}$. We define an edge from $b_{i}$ to $b_{j}$ if and only if both $i<j$ and $[b_{i},b_{j}]\notin R_\Tr$. The graph $[V,E, \lambda]$ thereby obtained is called the \index{dependence graph}\emph{dependence graph} of $w$. Up to isomorphism, the dependence graph of $w$ is unique, and so is its induced labelled partial order, which we further denote by $[V, \le, \lambda]$.

Let $c_{1}<\dots<c_{q}$ be the linearly ordered subset of $[V,\le,\lambda]$ containing all vertices with label in the clan $C$. For the vertex $v\in V$, we define the \emph{source point} $s(v)$ and the \emph{target point} $t(v)$ as follows:
$$
s(v)=\sup\{i\mid c_i\le v\}, \quad t(v)=\inf \{i \mid v \le c_i\}.
$$
By convention, $\sup\emptyset=0$ and $\inf\emptyset=q+1$. Thus, $0\le s(v)\le q$, $1\le t(v)\le q+1$ and $s(v)\le t(v)$ for all $v\in V$. Note that we have $s(v)=t(v)$ if and only if the label of $v$ belongs to $C$.

For $0\le s\le t\le q+1$, we define the median position $m(s,t)$ as follows. For $s=t$, we let $m(s,t)=s$. For $s<t$, by  \cite[Lemma 1]{DM}, there exist unique $l$  and $k$ such that $s\le l< t$, $k\ge 0$ and
$$
c_{s+1}\dots c_l\in \FF(\mathcal{C})(\mathcal{C}^{-1}\FF(\mathcal{C}))^k, \quad c_{l+1}\dots c_{t-1}\in (\FF(\mathcal{C}^{-1})\mathcal{C})^k \FF({\mathcal{C}^{-1}}).
$$
Then, we define $m(s,t)=l+\frac{1}{2}$ and we call $m(s,t)$ the \emph{median position}. Define the \index{global position} \emph{global position} of $v\in V$ to be $g(v)=m(s(v), t(v))$.

We define the normal form $\nf(w)$ of an element $w\in \Tr$ by introducing new edges into the dependence graph $[V, E, \lambda]$ of $w$. Let $u,v\in V$ be such that $\lambda(v)\in C$ and $[\lambda(u),\lambda(v)]\in R_\Tr$. We define a new edge from $u$ to $v$ if $g(u)<g(v)$, otherwise, we define a new edge from $v$ to $u$. The new dependence graph $[V, \hat{E}, \lambda]$ defines a unique element of the trace monoid $\hat{\Tr}$, where $\hat{\Tr}$ is obtained from $\Tr$ by omitting the commutativity relations of the form $[c,a]$ for any $c\in C$ and any $a\in \cA^{\pm 1}$. Note that the number of thin clans of $\hat{\Tr}$ is strictly less than the number of thin clans of $\Tr$. We proceed by designating a thin clan in $\hat{\Tr}$ and introducing new edges in the dependence graph $[V, \hat{E}, \lambda]$.

It is proven in \cite[Lemma 4]{DM}, that the normal form $\nf$ is a map from the trace monoid $\Tr$ to the free monoid $\FF(\cA\cup \cA^{-1})$, which is compatible with inversion, i.e.  it satisfies that $\pi(\nf(w))=w$ and $\nf(w^{-1})=\nf(w)^{-1}$, where $w\in \Tr$ and $\pi$ is the canonical epimorphism from $\FF(\cA\cup \cA^{-1})$ to $\Tr$.

We refer to this normal form as to the \emph{DM-normal form} or simply as to the \emph{normal form} of an element $w\in \Tr$.

Let $w\in \GG$ be a word and let $\HH<\GG$ be a canonical parabolic subgroup, $\HH=\HH_1\times \HH_2$, $\HH_i\ne 1$. We say that $w$ has $2k-1$ $\HH$-alternations if $w$ contains a subword $v_1 u_1\dots v_k u_{k}$, where $v_{i}\in \HH_1$ and $u_{i}\in \HH_2$ are non-trivial words, $i=1,\dots, k$. 

\begin{lem} \label{lem:pc}
Let $w\in \GG$ be written in the DM-normal form. Let $\HH<\GG$ be a canonical parabolic subgroup which decomposes as a non-trivial direct product of two canonical parabolic subgroups $\HH_1$ and $\HH_2$, $\HH=\HH_1\times \HH_2$. Then, the number of $\HH$-alternations in $w$ is bounded above by a constant that depends only on the number of clans of $\GG$.
\end{lem}
\begin{proof}
We use induction on the number of thin clans of $\GG$. If the number of thin clans equals one, then the statement is obvious. 

Suppose that the statement is true for all partially commutative groups with less than $n$ thin clans and let $\GG$ have precisely $n$ thin clans. Let $w\in \GG$ be some word, let $[V,E,\lambda]$ be the dependence graph of $w$ and let $C$ be a thin clan of $\GG$, write 
$$
w=w_1c_1w_2\cdots w_qc_qw_{q+1},
$$
where $c_1,\dots, c_q$ are all the letters in $w$ which belong to the clan $C$ and the dependence relations for $C$ have already been established. 

By definition of the DM-normal form,
$$
\nf(w)=\nf(w_1)c_1\nf(w_2)\cdots\nf(w_q)c_q\nf(w_{q+1}).
$$
Therefore, if $\HH\cap\langle C\rangle=1$, then the statement follows by induction.

Suppose that $\HH$ and $\langle C\rangle $ intersect non-trivially, then either $\HH\cap\langle C\rangle<\HH_1$ or $\HH\cap\langle C\rangle<\HH_2$. Let $v_1\cdots v_k$ be an $\HH$-alternation in $\nf(w)$. If $v_1\cdots v_k$ does not contain $c_i$ for all $i=1,\dots, q$, then the statement follows by induction. Let us assume that $v_1\cdots v_k$ contains $c_i$ for some $i=1,\dots, q$ and let $i$ be minimal so that $c_i$ is a letter of $v_j$. 

If $j=k$, then the bound on $k$ follows by induction. Suppose that $j\ne 1$, then $v_{j-1}$ is a subword of $\nf(w_{i})$ and $v_{j+1}$ is a subword of  $\nf(w_l)$ for some $l\ne i$. But, since $v_{j-1}\lra v_j$ and $v_{j+1}\lra v_j$, it follows that for every letters $x$ and $y$ from $v_{j-1}$ and $v_{j+1}$, the global positions $g(x)$ and $g(y)$ (with respect to $C$) coincide, contradicting the definition of the DM-normal form. It follows that $k\ne 2,\dots, k-1$.

Suppose that $j=1$.  Let $i'$ be minimal so that $i'>i$ and $c_{i'}$ is contained in $v_1\cdots v_k$. Observe that, without loss of generality, we may assume that such $i'$ exists since, otherwise, the bound on $k$ follows by induction. The letter $c_{i'}$ is a letter of some $v_{j'}$. If $j'\ne k$, then $v_{j'-1}$ is a subword of $w_{i'}$ and $v_{j'+1}$ is a subword of $\nf(w_{l'})$ for some $l'>i'$. But, since $v_{j'-1}\lra v_{j'}$ and $v_{j'+1}\lra v_{j'}$, it follows that for every letters $x$ and $y$ from $v_{j'-1}$ and $v_{j'+1}$, the global positions $g(x)$ and $g(y)$ (with respect to $C$) coincide, which contradict the definition of the DM-normal form. If $j'=k$, then the bound on $k$ follows by induction.
\end{proof}

\section{Algebraic geometry over groups}\label{sec:3}

The objective of this section is to establish the basics of algebraic geometry over groups. We refer the reader to \cite{BMR1, DMR} for details. Let $G=\langle A\rangle$ be a group and $F(X)$ be the free group on the alphabet $X$, $X = \{x_1, x_2, \ldots,  x_n\}$. Denote by $G[X]$ the free product $G*F(X)$.

For any element $s\in G[X]$, the formal equality $s=1$ can be treated, in an obvious way, as an \emph{equation} over $G$. In general, for a subset  $S \subset G[X]$, the formal equality $S=1$ can be treated as {\em a system of equations} over $G$ with coefficients in $A$.  Elements from $X$ are called {\em variables} and elements from $A^{\pm 1}$ are called {\em coefficients} or {\em constants}. To emphasize this we sometimes write $S(X,A) = 1$.

A {\em solution} $U$ of the system $S(X) = 1$ over a group $G$ is a tuple of elements $g_1, \ldots, g_n \in G$ such that every equation from $S$ vanishes at $(g_1, \ldots, g_n)$, i.e. $S_i(g_1, \ldots, g_n)=1$ in $G$, for all $S_i\in S$. Equivalently, a solution $U$ of the system $S = 1$ over $G$ is a $G$-homomorphism $\pi_U: G[X] \to G$ induced by the map $\pi_U: x_i\mapsto g_i$ such that $S\subseteq \ker(\pi_U)$. When no confusion arises, we abuse the notation and write $U(w)$, where $w\in G[X]$, instead of $\pi_U(w)$.

Denote by $\ncl\langle S\rangle$ the normal closure of $S$ in $G[X]$. Then, every solution of $S(X) = 1$ in $G$ gives rise to a $G$-homomorphism $\factor{G[X]}{\ncl\langle S\rangle} \to G$, and vice versa. The set of all solutions over $G$ of the system $S=1$ is denoted by $V_G(S)$ and is called the {\em algebraic set} or \emph{variety} defined by $S$.

For every system of equations $S$, we set the {\em radical of the system $S$}  to be the following subgroup of  $G[X]$:
$$
R(S) = \left\{ T(X) \in G[X] \ \mid \ \forall g_1,\dots,\forall g_n  \left( S(g_1,\dots,g_n) = 1 \rightarrow T(g_1,\dots, g_n) = 1\right) \right\}.
$$
It is easy to see that  $R(S)$ is a normal subgroup of $G[X]$ that contains $S$. There is a one-to-one correspondence between algebraic sets $V_G(S)$ and radical subgroups $R(S)$ of $G[X]$. Notice that if $V_G(S) = \emptyset$, then $R(S) = G[X]$.

It follows from the definition that
$$
R(S)=\bigcap\limits_{U\in V_G(S)}\ker(\pi_U).
$$

The quotient group
$$
G_{R(S)}=\factor{G[X]}{R(S)}
$$
is called the {\em coordinate group} of the algebraic set  $V_G(S)$ (or of the system $S$). There exists a one-to-one correspondence between algebraic sets and  coordinate groups. More formally, the categories of algebraic sets and coordinate groups are dual, see \cite[Theorem 4]{BMR1}. If a group $G=G_{R(S)}$ is given as the quotient $\factor{G[X]}{R(S)}$, we say that $G$ is given by its \emph{radical presentation}. If the corresponding system $S$ is finite, we say that $G$ is given by a \emph{finite radical presentation}.

Given a system of equations $S=1$, we denote by $G_S$ the group $\factor{G[X]}{\ncl\langle S\rangle}$.

A group $H$ is called  ($G$-)\emph{equationally Noetherian} if every system $S(X) = 1$ with coefficients from $G$ is equivalent over $G$ to a finite subsystem $S_0 = 1$, where $S_0 \subset S$, i.e. the system $S$ and its subsystem $S_0$ define the same algebraic set. If $G$ is $G$-equationally Noetherian, then we say that $G$ is equationally Noetherian. If $G$ is equationally Noetherian then the Zariski topology over $G^n$ is {\em Noetherian} for every $n$, i.e., every proper descending chain of closed sets in $G^n$ is finite. This implies that every algebraic set $V$ in $G^n$ is a finite union of irreducible subsets, called {\em irreducible components} of $V$, and such a decomposition of $V$ is unique. Recall that a closed subset $V$ is {\em irreducible} if it is not a union of two proper closed (in the induced topology) subsets.

We note that partially commutative groups are linear, see \cite{Heqn}, thus, equationally Noetherian, see \cite{BMR1}.

We say that a family of homomorphisms $\{\varphi_i\} \subset \Hom(H,K)$ {\em separates} ({\em discriminates}) $H$ into $K$ if for every non-trivial element $h \in H$ (every finite set of non-trivial elements $H_0 \subset H$) there exists $k$ so that $\varphi_k(h) \ne 1$ ($\varphi_k(h) \neq 1$ for every $h \in H_0$). In this case, we also say that $H$ is \emph{residually $K$} (that $H$ is \emph{fully residually $K$}) and call the family $\{\varphi_i\}$ \emph{separating} (\emph{discriminating}).

\begin{rem}
There is a natural epimorphism from $\factor{G[X]}{\ncl\langle S\rangle}$ onto $G_{R(S)}$. This epimorphism is an isomorphism if and only if $\factor{G[X]}{\ncl\langle S\rangle}$ is residually $G$.
\end{rem}

A finitely generated fully residually $H$ group $G$, is called a \emph{limit group} over $H$. The term limit group was introduced by Sela in \cite{Sela1} in the setting of free groups. The original definition is given in terms of the action of $G$ on a limiting real tree. One can prove, see \cite{Sela1}, that, in the case of free groups, the geometric and residual definitions are equivalent.

\begin{thm} \label{thm:limgrchar}
Let $H$ be a group. Then, for a finitely generated group $G$ the following conditions are equivalent: 
\begin{enumerate}
\item $G$ is fully residually $H$;
\item $G$ is the coordinate group of an irreducible variety over $H$.
\end{enumerate}
If any of the above two conditions holds, then
\begin{enumerate}
\item[(3)] $G$ embeds into an ultrapower of $H$.
\end{enumerate}
Furthermore if $H$ is equationally Noetherian, then all three conditions above are equivalent.
\end{thm}

\begin{lem} \label{lem:discfam}
Let $G$ be a limit group over $H$ and let $\{\varphi_i\}$ be a discriminating family for $G$. If the family $\{\varphi_i\}$  is a union of finitely many families, $\{\varphi_i\}=\{\varphi_{i,1}\} \cup \dots \cup \{\varphi_{i,n}\}$, then one of the families $\{\varphi_{i,k}\}$ is a discriminating family for $G$.
\end{lem}
\begin{proof}
Without loss of generality, we may assume that $n=2$ and $\{\varphi_i\}=\{\phi_i\}\cup \{\psi_i\}$. We show that either $\{\phi_i\}$ or $\{\psi_i\}$ is a discriminating family for $G$. 

Assume the contrary, then there exist finite sets of non-trivial elements $S=\{g_1,\dots, g_m\}\subset G$ and $T=\{h_1,\dots, h_k\}\subset G$  which can not be discriminated into $H$ by $\{\phi_i\}$ and $\{\psi_i\}$, correspondingly, i.e. for all $i$ we have $\psi_i(g_{l_i})=1$ and $\phi_i(h_{m_i})=1$ for some $l_i=1,\dots, m$, $m_i=1,\dots, k$. Since $\{\varphi_i\} =\{\phi_i\}\cup \{\psi_i\}$, the set $S\cup T$ can not be discriminated by $\{\varphi_i\}$ into $H$ - a contradiction.
\end{proof}

\begin{rem} \label{rem:discfam}
Let $G$ be a group discriminated by a finitely generated partially commutative group $\GG$. Notice that since $\GG$ has only finitely many different canonical parabolic subgroups, then, by Lemma \ref{lem:discfam}, for any $H< G$, there exists a canonical parabolic subgroup $\HH$ of $\GG$  with the two following properties: there exists a discriminating family $\{\varphi_i\}$ such that $\varphi_i(H)<\HH$ and there exist no proper subgroups $\HH'$ of $\HH$ with the first property. 
\end{rem}

\section{Special cube complexes}\label{sec:4}
In this section we review the theory of co-special actions on $\CAT(0)$ cube complexes, see \cite{HW}. 
\begin{defn}
A \emph{cube complex} $X$ is a CW-complex where each $n$-cell $\sigma$ is a standard Euclidean $n$-cube whose attaching map $\varphi_\sigma:\partial \sigma \to X^{(n-1)}$ satisfies the following conditions:
\begin{enumerate}
\item the restriction of $\varphi_e$ to every fact of $e$ is a linear homeomorphism onto a cube of one lower dimension;
\item $\varphi_\sigma$ is a homeomorphism onto its image.
\end{enumerate}
We give $X$ the standard CW-topology.
\end{defn}
We refer to 1-cells as edges or as 1-cubes and to 0-cells as vertices or 0-cubes. 

Let $X$ denote a cube complex. The link $\lk(\sigma)$ of a cube $\sigma$ is a \emph{simplicial} complex whose $n$-skeleta are defined inductively as follows:
\begin{itemize}
\item The set of vertices of $\lk(\sigma)$ is $\{\tau\in X^{(n+1)}\mid \sigma \in \partial \tau\}$;
\item The set of $n$-simplices, $n\ge 1$ of $\lk(\sigma)$ is 
$$
\{(\tau_0,\dots, \tau_n)\mid \tau_i\in \lk(\sigma)^0 \hbox{ and there is a cube $\nu$ such that $\tau_i \in \partial \nu$}\}.
$$
\end{itemize}

The cube complex $X$ is \emph{combinatorially non-positively} curved if each vertex link is flag (that is each complete subgraph is the 1-skeleton of a simplex). We say that $X$ is combinatorially $\CAT(0)$ whenever $X$ is combinatorially non-positively curved and simply-connected. Following Sageev, \cite{S95}, we call a simply connected combinatorially non-positively curved cube complex a \emph{cubing}.

Each cube of $X$ can be given the metric of a standard unit Euclidean cube in $\BR^{n}$. One can then put on a cubing $X$ a pseudo-metric, which, in fact, turns $X$  into a complete $\CAT(0)$ metric space $(X,d)$, \cite{BH}. 

\begin{lem}[see \cite{gromov}] 
Let $X$ be a cube complex. Then, $X$ is combinatorially non positively curved if and only if the length metric $d$ on $X$ is locally $\CAT(0)$. In particular, a cube complex is metrically $\CAT(0)$ if and only if it is combinatorially $\CAT(0)$.
\end{lem}

Hence, the metric and combinatorial geometry of $\CAT(0)$ cube complexes are closely connected, see \cite{Hagiso} for more details. We work with non-positively curved cube complexes from a combinatorial viewpoint and consider them as higher-dimensional analogues of graphs. However, to introduce certain notions, it will be convenient to think of the cube complex as a geometric object.

A midcube in the cube $I_n=[0,1]\times \dots \times[0,1]$ is the subset obtained by restricting one of the coordinates to $\frac{1}{2}$, so the midcube is parallel to two $(n-1)$-faces of $I_n$. The edges of $I_n$ dual to this midcube are the edges perpendicular to it. The center of an $n$-cube in a cube complex is the image of $(\frac{1}{2}, \dots , \frac{1}{2})$ from the corresponding $n$-cell. The center of an edge is its midpoint.

Given a cube complex $X$, we form a new cube complex $Y_X$, whose cubes are the midcubes of cubes of $X$. The vertices of $Y_X$ are the midpoints of edges of $X$. The restriction of a $(k+1)$-cell of $X$ to a midcube of $I_{k+1}$ defines the attaching map of a $k$-cell in $Y_X$. Each component of $Y_X$ is a \emph{hyperplane} of $X$. An edge of $X$ is dual to some hyperplane $H$ if its midpoint is a vertex of $H$. Each edge $e$ is dual to a unique hyperplane, which we will denote by $H(e)$.  Two hyperplanes $A$, $B$ of a cube complex $X$ \emph{intersect} if $A \cap B\ne \emptyset$, they \emph{cross} if they intersect but are not equal.

Set $\mathcal{H}=\factor{\Epsilon}{\parallel}$ and let $[e],[f]\in \mathcal{H}$. Elements of $\mathcal{H}$ are sometimes called \emph{{\rm (}unoriented{\rm)}  combinatorial hyperplanes}. Combinatorial hyperplanes are in one-to-one correspondence with hyperplanes. When no confusion arises we refer to combinatorial hyperplanes as simply hyperplanes.

We will use a combinatorial metric defined on the set of vertices of $X$. Let $\Epsilon$ be the set of oriented edges of $X$. An edge path or simply a \emph{path} in $X$ is a finite sequence of oriented edges such that the end of each edge is the origin of its successor. The length of a path is just the number of edges in the sequence. Given two vertices $p, q\in X$ we define the distance $\dist(p,q)$ between $p$ and $q$ as the infimum of the lengths of paths between them. One can check that $\dist$ is a metric, called the \emph{edge-path metric}.

Sometimes, we shall consider cube complexes with the rescaled edge-path metric. Let $c\in \BR$ and $c>0$, then we define the rescaled edge-path metric $\dist_c$ on a cube complex $X$ as follows. For any $p,q\in X$, set 
$$
\dist_c(p,q)=\frac{\dist(p,q)}{c}.
$$

As we have already mentioned, essentially, just as graphs, cube complexes are combinatorial objects. Just as any simplicial tree (graph) can be made into a metric tree (graph) by identifying every edge with $[0,1]$, given a cube complex $X$ with a (rescaled) edge-path metric  $\dist_c$, we shall make $X$ into a \emph{metric cube complex}. We identify every edge of $X$ with $[0,c]$ and endow every cube of $X$ with the $\ell^1$ metric. We thereby obtain a metric on $X$ that we denote by $d_c$. Observe that for any vertices $p,q\in X$ we have $d_c(p,q)=\dist_c(p,q)$. Furthermore, any combinatorial isometry of $(X, \dist_c)$ induces an isometry of $(X, d_c)$. We shall consider only combinatorial isometries of cube complexes and for the most part identify $(X,d_c)$ and $(X,\dist_c)$, and abusing the notation and terminology, refer to $(X,d_c)$ as to $(X,\dist_c)$, in particular, in Sections \ref{sec:6} and \ref{sec:7}.

Let $X$ be a cube complex and let $\Epsilon$ be the set of edges of $X$. We define the equivalence relation $\parallel$ as follows: set $e\parallel f$, $e,f \in \Epsilon$ if and only if there exists a finite sequence of edges $e=e_1, \dots, e_n=f$ such that for each $i=1,\dots, n-1$, the edges $e_i$ and $e_i+1$ are opposite sides of some 2-cube in $X$ oriented in the same direction.

Special cube complexes were introduced by F.~Haglund and D.~Wise, \cite{HW}. Following Haglund and Wise, we define a special cube complex as a nonpositively curved cube complex which does not have certain pathologies related to its immersed hyperplanes. An immersed hyperplane $D$ crosses itself if it contains two different midcubes from the same cube of $C$. An immersed hyperplane $D$ is $2$-sided if the map $D\to C$ extends to a map $D \times I \to C$ which is a combinatorial map of cube complexes. When $D$ is $2$-sided, it is possible to consistently orient its dual $1$-cubes so that any two dual $1$-cubes lying (opposite each other) in the same $2$-cube are oriented in the same direction.

An immersed $2$-sided hyperplane $D$ self-osculates if for one of the two choices of induced orientations on its dual $1$-cells, some $0$-cube $v$ of $C$ is the initial $0$-cube of two distinct dual $1$-cells of $D$.
A pair of distinct immersed hyperplanes $D, E$ cross if they contain distinct midcubes of the same cube of $C$. We say $D,E$ osculate, if they have dual $1$-cubes which contain a common $0$-cube, but do not lie in a common $2$-cube. Finally, a pair of distinct immersed hyperplanes $D, E$ inter-osculate if they both cross and osculate, meaning that they have dual $1$-cubes which share a $0$-cube but do not lie in a common $2$-cube.
A cube complex is special if all the following hold, see \cite[Figure 1]{HWCox}:
\begin{enumerate}
\item No immersed hyperplane crosses itself;
\item Each immersed hyperplane is $2$-sided;
\item No immersed hyperplane self-osculates;
\item No two immersed hyperplanes inter-osculate.
\end{enumerate}
All cube complexes we consider are special and connected unless stated otherwise.

Central examples of special cube complexes are cube complexes associated to partially commutative groups. 
Let $\GG$ be a (perhaps infinitely generated) partially commutative group. The 2-complex $X$ of the standard presentation of $\GG$ extends to a non-positively curved cube complex $C(\GG)$ by adding an $n$-cube (in the form of an $n$-torus) for each set of $n$ pairwise commuting generators. It is well-known, see \cite[Example 3.3]{HW}, that 
\begin{itemize}
\item Every graph is a special 1-dimensional cube complex;
\item Every $\CAT(0)$ cube complex is special;
\item For any (not necessarily finitely generated) partially commutative group $\GG$, the complex $C(\GG)$ is a special cube complex.
\end{itemize}

Given a partially commutative group, throughout this text we denote by $C(\GG)$ the standard complex of $\GG$ and by $\widetilde{C(\GG)}$ its universal cover. The universal cover $\widetilde{C(\GG)}$ is a $\CAT(0)$ cube complex whose $1$-skeleton is the Cayley graph $\Cay(\GG)$ of the standard presentation of $\GG$.

As the following series of results show, not only $C(\GG)$ and  $\widetilde{C(\GG)}$ are important examples of special cube complexes, but essentially these examples are universal.
\begin{prop}[Theorem 4.2, \cite{HW}] \label{prop:liso}
A cube complex $X$ is special if and only if it admits a combinatorial local isometry to the cube complex $C(\GG)$ of a {\rm(}not necessarily finitely generated{\rm)} partially commutative group $\GG$. Furthermore, if $X$ is special, then $\pi_1(X)$ is a subgroup of $\GG$.
\end{prop}

As we have already discussed, cubings are a natural generalisation of simplicial trees.  While the structure of groups acting (without inversions) on simplicial trees is well understood, the structure of groups acting (essentially) even freely on cubings may be extremely complex, as shown by examples of Burger and Mozes, \cite{BM}. Proposition \ref{prop:liso} motivates the definition of a co-special action. In the case of trees, if a group acts on a tree freely and without inversions (and, therefore, in this case, co-specially), the group is a subgroup of a free group. In the case of  co-special actions on cubings, Proposition \ref{prop:liso} can be reformulated as follows.
\begin{prop} \label{prop:HW}
A group $G$ acts freely co-specially on a {\rm(}not necessarily finite dimensional{\rm)} cubing $C$ if and only if it is a subgroup of a {\rm(}not necessarily finitely generated{\rm)} partially commutative group.
\end{prop}

Furthermore, by construction of the $A$-typing maps, see \cite[Sections 3 and 4]{HW}, one can extract a more specific result.
\begin{cor} \label{cor:HW}
Let $G$ be a group acting freely co-specially on a {\rm(}not necessarily finite dimensional{\rm)} cubing $C$. Then, 
\begin{itemize}
\item there exists a combinatorial isometric embedding of $C$ into the universal cover $\widetilde{C(\GG)}$ of the standard complex of a {\rm(}not necessarily finitely generated{\rm)} partially commutative group $\GG$;
\item the action of $G$ on $C$ extends to a free co-special action $\alpha$ of $G$ on $\widetilde{C(\GG)}$.
\item the action $\alpha$ is  induced by the action of $G$ by left multiplication on the Cayley graph $\Cay(\GG)$.
\end{itemize}  
\end{cor}

The following result is not hard to deduce from Proposition \ref{prop:HW}.

\begin{prop} \label{prop:corHW}
Let $G$ be a finitely generated group. Then, the group $G$ acts freely co-specially on a finite dimensional cubing $C$ if and only if it is a subgroup of a finitely generated partially commutative group $\bar\GG$.
\end{prop}
\begin{proof}
Let $G=\langle w_1,\dots, w_k\rangle$ be a finitely generated subgroup of a (perhaps, infinitely generated) partially commutative group $\GG$. Then, $G$ is a subgroup of $\bar \GG=\langle \az(w_1),\dots, \az(w_k)\rangle$, which in turn is a canonical parabolic subgroup of $\GG$. The group $G$ acts (freely) by left multiplication on the Cayley graph $\Cay(\bar\GG)$ of $\bar\GG$. The Cayley graph $\Cay(\bar\GG)$ is the $1$-skeleton of the universal cover $\widetilde{C(\bar \GG)}$ of the standard complex $C(\bar \GG)$ of $\bar\GG$ and the action of $G$ on $\Cay(\bar\GG)$ gives rise to a free co-special action on $\widetilde{C(\bar \GG)}$. 

Conversely, if $G$ is a finitely generated group acting freely and co-specially on a cubing, using an $A$-typing map, see \cite[Sections 3 and 4]{HW}, one obtains an embedding of $G$ into a (a priori infinitely generated) partially commutative group $\GG$. As above, we conclude that as $G$ is finitely generated, so it embeds into a finitely generated canonical parabolic subgroup $\bar\GG$ of $\GG$.
\end{proof}

\section{Cube complexes of finite width}\label{sec:5}

Finitely generated partially commutative groups enjoy many nice properties, which are not shared by infinitely generated ones. In our setting, it will be crucial that the  group $G$ acting co-specially on a cubing $C$ be a subgroup of a finitely generated partially commutative group. We have seen above that one can restrict the consideration to finitely generated groups in order to make sure that $G$ indeed is a subgroup of a finitely generated partially commutative group. Alternatively, as is done by Haglund and Wise, see \cite{HW, HWCox}, one can impose the condition that the quotient of $C$ by the action of $G$ have finitely many immersed hyperplanes. 

However, both of these approaches have their limitations. On the one hand, requiring that the quotient have only finitely many immersed hyperplanes, imposes strong restrictions not only on the group $G$, but also on the action of $G$ on $C$, for instance, one automatically excludes many non co-compact actions of subgroups of partially commutative groups. 

On the other hand, the assumption that the group $G$ is finitely generated is not sufficient in our setting. Suppose we are given a sequence of free actions of finitely generated groups $G_i$ on cubings $C_i$ (we shall see in the following sections that this is precisely the setting we are interested in). Each of these actions defines an embedding of $G_i$ into a partially commutative group $\GG_i$. It may happen that the ranks of $\GG_i$ tend to infinity as $i$ tends to infinity and therefore, we can not find a finitely generated universe for the groups $G_i$. As we have mentioned in the introduction, a positive solution of Problem \ref{prob:1} would resolve this technical issue and allow to consider arbitrary cubings.

Our next goal is to show that one can impose a natural geometric restriction on the cubing $C$ that allows to solve a restricted version of Problem \ref{prob:1}, namely, a group $G$ acts freely co-specially on a cubing of \emph{finite width} if and only if  $G$ is a subgroup of a finitely generated partially commutative group. Furthermore, it allows us to find a universal partially commutative group that contains all groups $G$ acting freely on cubings of bounded width. Moreover, we show that the free co-special action of a group $G$ on a cubing of finite width extends to a free co-special action of $G$ on the universal cover of the standard complex of a finitely generated partially commutative group.

Let $X$ be a special cube complex. Then, by Proposition \ref{prop:HW}, there always exists a partially commutative group $\GG$ so that there exists a combinatorial local isometry from $X$ to $C(\GG)$. However, such an isometry is by far non-unique. For example, the $A$-typing map for the Cayley graph of the free group $F_2$ of rank $2$ results in a free group of countable rank, but, of course, the Cayley graph  $\Cay(F_2)$ admits a combinatorial local isometry onto $C(F_2)$.

\begin{defn} \label{defn:width}
Let $[e],[f]\in \mathcal{H}$ be two hyperplanes of some cube complex. Set $[e]\approx[f]$ if and only if every hyperplane $[g]\in \mathcal{H}$ crosses $[e]$ if and only if it crosses $[f]$. One can check that $\approx$ is an equivalence relation. Denote by $\mathfrak{H}=\mathfrak{H}(X)$ the quotient $\factor{\mathcal{H}}{\approx}$.

We say that the cube complex $X$ is \emph{$n$-wide} or has \emph{width $n$} if and only if there exists a combinatorial local isometry from $X$ to $C(\EE)$, where $\EE$ is a partially commutative group such that $|\mathfrak{H}(C(\EE))|=n$, $n\in \BN$ and $n$ is minimal with this property, i.e. for any $\EE'$ so that $|\mathfrak{H}(C(\EE'))|=k<n$, the cube complex $X$ does not admit a combinatorial local isometry onto $C(\EE')$. In this case, we write $w(X)=n$. If  $|w(X)|=n$ for some $n\in \BN$, we say that $X$ is of \emph{finite width}, otherwise, we call $X$ \emph{infinitely wide}.

Throughout this text, given a special cube complex $X$ of width $n$, we denote by $\EE(X)$ a partially commutative group so that $w(C(\EE(X)))=n$ and there is a combinatorial local isometry from $X$ to $C(\EE(X))$.
\end{defn}

\begin{rem}
Let $X$ be a special cube complex of width $n$, then there is a combinatorial isometric embedding of the universal cover $\widetilde{X}$ into $\widetilde{C(\EE)}$ for some partially commutative group $\EE$, where $w(C(\EE))=n$.

Conversely, if $X$ is a cubing which admits a combinatorial isometric embedding into $\widetilde{C(\GG)}$ so that $w(C(\GG))=n$, and $X\to Y$ is a covering map, then $X, Y$ have finite width and $w(X), w(Y)\le n$.
\end{rem}

We now record the following observations about the width of cube complexes, which follow directly from the definition.
\begin{lem} \label{lem:width} \
\begin{enumerate}
\item The width of any graph is one. 
\item The width of the standard cubulation of the euclidean space $E^n$ is $n$. 
\item More generally, if $X$ and $Y$ are cube complexes, then $w(X\times Y)=w(X)+ w(Y)$. 
\item If the width of $X$ is finite, then $X$ is finite dimensional. 
\item The width of any compact cube complex is finite. If a cube complex contains only finitely many embedded hyperplanes, then it has finite width.
\item Let $\GG$ be a partially commutative group with the underlying graph $\Gamma$,  $w(C(\GG))=N$, then the deflation $\Gamma'$ of $\Gamma$ is a finite graph with $N$ vertices and $w(C(\GG))=w(C(\GG(\Gamma'))$.
\end{enumerate}
\end{lem}

The main result of this section is the following free actions theorem for special cube complexes of finite width. Viewing simplicial trees as $1$-dimensional cubings and graphs as $1$-dimensional special cube complexes, the next theorem is a natural generalisation of the free actions theorem for trees: a group acts freely without inversions of edges (and hence co-specially) on a tree if and only if it is a subgroup of a finitely generated free group.

\begin{thm}
Let $G$ be a group. Then, $G$ acts freely co-specially on a cubing of width $n$ if and only if  $G$ is a subgroup of a partially commutative group $\GG$ and $w(C(\GG))=n$. 
\end{thm}
\begin{proof}
Let $\widetilde{C(\GG)}$ be the universal cover of $C(\GG)$ and $w(\widetilde{C(\GG)})=n$. Every subgroup $G$ of $\GG$ acts (freely)  by left multiplication on the Cayley graph $\Cay(\GG)$ of $\GG$. Since $\Cay(\GG)$ is the $1$-skeleton of $\widetilde{C(\GG)}$, we get a free co-special action of $G$ on $\widetilde{C(\GG)}$.

Conversely, if $X$ is a cubing of width $n$, then there exists a combinatorial local isometry from $X$ to $C(\EE)$, where $w(C(\EE))=n$. Therefore, by Proposition \ref{prop:liso}, $G$ embeds into $\EE$.
\end{proof}

\begin{lem} \label{lem:2k}
Let $\GG=\GG(\Gamma)$ be a partially commutative group so that $w(C(\GG))=n<\infty$. Then, $\GG$ is a subgroup of a $2n$-generated partially commutative group.
\end{lem}
\begin{proof}
Let $\Gamma'$ be the deflation of $\Gamma$. Then, by Lemma \ref{lem:width}, $\Gamma'$ has exactly $n$ vertices. Define the graph $\Gamma^*$ as follows. The graph $\Gamma^*$ is a finite graph with exactly $2n$ vertices constructed from $\Gamma'$. For every vertex $v\in \Gamma'$, we introduce two vertices $v_1, v_2\in \Gamma^*$. We set $(v_i,u_j)\in E(\Gamma^*)$ if and only if $(v,u)\in E(\Gamma')$, for all  $v,u\in \Gamma'$, $i,j=1,2$.

Observe that both $\GG(\Gamma)$ and $\GG(\Gamma^*)$ can be viewed as a graph product of groups with the underlying graph $\Gamma'$ and free vertex groups (see \cite{Green,Goda} for definition and basic properties of graph products). For every vertex $v\in \Gamma'$, let $F([v])$ be the free group generated by $\{v_1',\dots, v_k', \dots\}=[v]$ associated to the corresponding vertex of the graph product $\GG(\Gamma)$ and let $F(v_1,v_2)$ be the  free group generated $\{v_1,v_2\}$ associated to the corresponding vertex of the graph product $\GG(\Gamma^*)$.  Let $\phi_{v}:F([v])\hookrightarrow F(v_1,v_2)$ be an embedding of $F([v])$ into $F(v_1,v_2)$. 

The embeddings $\phi_v$, $v\in \Gamma'$, give rise to an embedding $\phi:\GG(\Gamma) \to \GG(\Gamma^*)$. We conclude that  $\GG$ is a subgroup of a $2n$-generated partially commutative group $\GG(\Gamma^*)$.
\end{proof}

\begin{cor} \label{cor:finitewidth}
Let $G$ be a group. Then, $G$ acts freely co-specially on a cubing of finite width if and only if $G$ is a subgroup of a finitely generated partially commutative group. 
\end{cor}

Given a special cube complex $X$ of width $n$, the $2n$-generated partially commutative group  constructed from $\EE(X)$ in Lemma \ref{lem:2k} will be denoted by $\PP(X)$ .

As we have shown above, if a group $G$ acts freely co-specially on a cubing of finite width, then it is a subgroup of a finitely generated partially commutative group $\PP$. We now show that there is a universal partially commutative group that contains all groups that act freely co-specially on cubings of a given width.

By Lemma \ref{lem:width}, for any $N\in \BN$ there are only finitely many deflated graphs $\Gamma_1,\dots, \Gamma_k$, so that $w(C(\GG(\Gamma_i)))\le N$. Set $\Gamma_N$ to be the union of the graphs $\Gamma_i^*$, 
$$
\Gamma_N=\Gamma_1^*\cup \dots \cup \Gamma_k^*
$$
and $\GG_N=\GG(\Gamma_N)=\GG(\Gamma_1^*)*\dots*\GG(\Gamma_k^*)$, where the graphs $\Gamma_i^*$'s are defined as in the proof of Lemma \ref{lem:2k}. Observe that $\Gamma_N$ is a finite graph.

We now arrive to the following
\begin{cor} \label{cor:1pcg}
For every $N\in \mathbb N$, there exists a partially commutative group $\GG_N=\GG(\Gamma_N)$ such that every group $G$ acting freely co-specially on a cubing $C$ of width $N$ is a subgroup of $\GG_N$.
\end{cor}

Our next goal is to show that if a group $G$ acts freely co-specially on a cubing $C$ of width $n$, then one can construct an embedding from $G$ to $\PP(C)$ so that it induces an equivariant quasi-isometric embedding from $C$ to the universal cover of $C(\PP)$. Therefore, the study of groups acting freely co-specially on cubings of finite width reduces to the study of groups acting freely co-specially on the universal covers of standard complexes of partially commutative groups.

Let $G=\langle S\rangle$, $S=S^{-1}$ be a finitely generated group acting on a cubing $C$, and let $x\in C$. Define \emph{the displacement at $x$} to be
$$
\partial_x = \max\limits_{s\in S} \dist(s.x, x).
$$
If the cubing $C$ has a designated based point $b$, then the number $\partial_b$ is called the \emph{displacement of the action of $G$}.

\begin{lem}\label{lem:qie}
Let $F(A)=F(a_1,\dots, a_r, \dots)$, $F(B)=F(a_1,\dots, a_r)$ and $F(b,c)$ be free groups on the indicated alphabets. Let $w_1,\dots, w_r\in F(b,c)$ be $r$ words in $F(b,c)$ so that $|w_i|=r$, $w_i\ne w_j^{\pm 1}$ and $w_i$ does not begin or end with $b^{-1}$ and contains the letter $c$, $i=1,\dots, r$.  Define the homomorphism $\psi:F(a_1,\dots, a_r,\dots,)\to F(b,c)$ as follows:
$$
\psi(a_i)=w_i b^{L+i} a_i b^{L+i} w_i,
$$
where $L\in \BN$, $i=1,\dots, r$ and set 
$$
\psi(a_j)=b^{L+r+j} a_i b^{L+r+j},
$$
for all $j>r$. Then, the map $\psi$ is an embedding of $F(A)$ into $F(b,c)$, the restriction of $\psi$ onto $F(B)$ is a quasi-isometric embedding and the quasi-isometry constant is $4r+2L+1$.
\end{lem}
\begin{proof}
The proof that $\psi$ is an embedding is standard and left to the reader. Abusing the notation, we denote the restriction of $\psi$ onto $F(B)$ by $\psi$. We only need to show that $\psi$ is a quasi-isometry and compute the quasi-isometry constant.

By definition $2r+2L+3 \le |\psi(a_i)|\le 4r+2L+1$ and the equalities are attained for $i=1$ and $i=r$, correspondingly. On the other hand, by the definition of $\psi$, it follows that for all $v\in F(a_1,\dots, a_r)$ the length $|\psi(v)|_{F(b,c)}$ satisfies the following inequality 
$$
2L |v|_{F(a_1,\dots, a_r)}\le |\psi(v)|_{F(b,c)}\le (4r+2L+1)|v|_{F(a_1,\dots, a_r)}
$$
and the statement follows.
\end{proof}

\begin{lem} \label{lem:frgrqi}
Let $F(A)=F(a_1,\dots, a_l, \dots)$ and $F(b,c)$ be two free groups on the indicated generators. Let $H=\langle w_1,\dots, w_k\rangle <F$ be a finitely generated subgroup of $F(A)$. Re-enumerating the generators of $F(A)$, let $B=\{a_1,\dots, a_r\}=\{\az(w_1), \dots, \az(w_k)\}$. We view $\Cay(F(A))$, $\Cay(F(B))$ and $\Cay(F(b,c))$ as 1-dimensional cube complexes based at the identity. The embedding of $H$ into $F(B)$ defines a subcomplex $C(H)$ in $\Cay(F(B))$. We designate the identity as the basepoint of $C(H)$.

Let $\alpha$ be the action of $H$ on $\Cay(F(B))$ by left multiplication, let $\psi$ be the map defined in {\rm Lemma \ref{lem:qie}}, let $\beta$ be the corresponding action of $H$ on $F(b,c)$ and let $\partial(\alpha)$ and $\partial(\beta)$ be the displacements of the corresponding actions. Then, 
\begin{itemize}
\item the homomorphism $\psi$ induces an equivariant, quasi-isometric embedding $\psi^*$ of $\Cay(F(B))$ into $\Cay(F(b,c))$, the constant of the quasi-isometry $\psi^*$ does not exceed $6k\partial(\alpha)+1$;
\item the following inequality holds: $2k\partial(\alpha)^2\le \partial(\beta)\le 7k\partial(\alpha)^2$;
\item the homomorphism $\psi$ induces an equivariant, quasi-isometric embedding $\psi'$ of $(F(B), \dist_{\partial(\alpha)})$ into $(\Cay(F(b,c)), \dist_{\partial(\beta)})$, the quasi-isometry constant of $\psi'$ does not exceed $7k$.
\end{itemize}
\end{lem}
\begin{proof}
By Lemma \ref{lem:qie}, for  $L=k\partial(\alpha)$, the homomorphism $\psi$ induces the quasi-isometries $\psi^*$ and $\psi'$. We only need to compute the quasi-isometry constant.

By definition, the displacement of the action is the displacement of the basepoint of the complex. Therefore, in this case, the displacement $\partial(\alpha)$ equals $\max\limits_{i=1,\dots, k}\{ |w_i|_{F(B)}\}$. There are at most $k\partial(\alpha)$ different letters in the words $w_1,\dots, w_k$, therefore $r\le k\partial(\alpha)$. By Lemma \ref{lem:qie}, it follows that the quasi-isometry constant of $\psi^*$ does not exceed $6k\partial(\alpha)+1$.

The displacement of $\partial(\beta)$ equals $\max\limits_{i=1,\dots, k}\{ |\psi(w_i)|_{F(b,c)}\}$. Let $a_i\in B$, then $|\psi(a_i)|= 2(r+k\partial(\alpha) + i)+1$. Since $|w_i|_{F(B)}\le \partial(\alpha)$ and $r\le k\partial(\alpha)$, we have
\begin{align*}
\partial(\beta)=\max_{i=1,\dots, k}\{|\psi(w_i)|_{F(b,c)}\} & \le \max_{i=1,\dots, k}\{|w_i|_{F(B)}\}\cdot \max_{j=1,\dots,k\partial(\alpha)} |\psi(a_j)|_{F(b,c)}\\
& \le \max_{i=1,\dots, k}\{|w_i|_{F(B)}\}\cdot 7k\partial(\alpha)=7k \partial(\alpha)^2.
\end{align*}

On the other hand, since there exists $l$, $1\le l\le k$ so that $|w_l|_{F(B)}=\partial(\alpha)$, we have that
$$
\partial(\beta)\ge |\psi(w_l)|_{F(b,c)}\ge (2k\partial(\alpha)+1)\partial(\alpha)\ge 2k\partial(\alpha)^2.
$$

The last statement now follows from the first two.
\end{proof}

\begin{prop} \label{prop:uniqi}
Let $\alpha$ be a free co-special action of a $k$-generated group $G$ on a based cubing $(C,b)$ of width $n$. Let $\EE=\EE(C)$ and $\PP=\PP(C)$ be the corresponding partially commutative groups. Let $\chi:G\hookrightarrow \EE(C)$ be the induced embedding and let $\beta'$ be the action of $G$ induced by $\alpha$ on $\widetilde{C(\EE)}$. Let $G=\langle w_1,\dots, w_k\rangle<\EE(C)$ and, re-enumerating the generators of $\EE(C)$, let $\langle \az(w_1),\dots, \az(w_k)\rangle =\langle a_1,\dots, a_r\rangle$, where $a_i$ are canonical generators of $\EE(C)$, $i=1,\dots, r$. 

Let $\psi$ be the embedding of $\EE(C)$ into $\PP(C)$ defined as in {\rm Lemma \ref{lem:2k}} using the maps $\psi_v$ from {\rm Lemma \ref{lem:qie}}, where $L=k\partial(\alpha)$. Denote by $\phi$ the composition of $\chi$ and $\psi$ and let  $\beta$ be the action of $G$ induced by $\alpha$ on  $\widetilde{C(\PP)}$. Endow $C$, $\widetilde{C(\EE)}$ and $\widetilde{C(\PP)}$ with the edge-path metric and let the cubings $\widetilde{C(\EE)}$ and $\widetilde{C(\PP)}$ be based at lifts of the identity elements $\id'$ and $\id$ of $\EE$ and $\PP$, correspondingly. Then, 
\begin{itemize}
\item the embedding $\phi$ induces an equivariant, based, quasi-isometric embedding $\psi^*$ of $(C,b)$ into $(\widetilde{C(\PP)},\id)$;
\item the following inequality holds $2k\partial(\alpha)^2\le \partial(\beta)\le 7 k \partial(\alpha) ^2$;
\item the embedding $\phi$ induces an equivariant, based, $7k$-quasi-isometric embedding from $(C, \dist_{\partial(\alpha)}, b)$ to $(\widetilde{C(\PP)}, \dist_{\partial(\beta)}, \id)$.
\end{itemize}
\end{prop}
\begin{proof}
By Corollary \ref{cor:HW}, the embedding of $G$ into $\EE$ induces an equivariant, combinatorial isometric embedding of $C$ into $\widetilde{C(\EE)}$. Hence, it suffices to show that the embedding $\psi$ of $\EE$ into $\PP$ and, consequently, the induced map $\psi^*$ from $\widetilde{C(\EE)}$ into $\widetilde{C(\PP)}$ can be chosen in such a way that the restriction $\psi^*|_C$ of $\psi^*$ onto the image of $C$ in $\widetilde{C(\EE)}$ satisfies the statements of the proposition.
 
The statement now follows by the definition of $\psi$, see proof of Lemmas \ref{lem:2k} and \ref{lem:frgrqi}.
\end{proof}

\section{Real cubings}\label{sec:6}
The aim of this section is to introduce and give examples of the main object of our study, \emph{real cubings}.
We begin by recalling the notion of an ultralimit of metric spaces. We refer the reader to \cite{Roe} for more details.

Let $\UU$ be a non-principal ultrafilter on $\BN$. Let $(X_n,d_n)$ be a sequence of metric spaces with specified base-points $p_n\in X_n$. We say that a sequence $(x_n)_{n\in \BN}$, where $x_n\in X_n$, is \emph{admissible}, if the sequence of real numbers $(d_n(x_n,p_n))_{n\in \BN}$ is bounded. Denote the set of all admissible sequences by $\AA$. It is easy to see from the triangle inequality that, for any two admissible sequences $\mathbf x=(x_n)_{n\in \BN}$ and $\mathbf y=(y_n)_{n\in \BN}$, the sequence $(d_n(x_n,y_n))_{n\in \BN}$ is bounded and hence there exists a $\UU$-limit $\hat d_\infty(\mathbf x, \mathbf y)=\lim\limits_\UU d_n(x_n,y_n)$. 

Define a relation $\sim$ on the set $\AA$ as follows. Set $\mathbf x\sim \mathbf y$ if and only if $\hat d_\infty(\mathbf x, \mathbf y)=0$. It is easy to show that $\sim$ is an equivalence relation on $\AA$. 

The \emph{ultralimit} $(X_\infty, d_\infty)=\lim\limits_\UU(X_n,d_n, p_n)$ with respect to $\UU$ of the sequence $(X_n,d_n, p_n)_{n\in \BN}$ is a metric space $(X_\infty, d_\infty)$ defined as follows. As a set, we have $X_\infty=\factor{\AA}{\sim}$. For two $\sim$-equivalence classes $[\mathbf x]$ and $\mathbf{[y]}$ of admissible sequences  $\mathbf x$ and $\mathbf{y}$, we set $d_\infty([\mathbf x], [\mathbf y])=\hat d_\infty(\mathbf x,\mathbf y)=\lim\limits_\UU d_n(x_n,y_n)$. It is not hard to see that $d_\infty$ is well-defined and that it is a metric on the set $X_\infty$.

An important class of ultralimits are the so-called \emph{asymptotic cones} of metric spaces. Let $(X,d)$ be a metric space, let $\UU$ be a non-principal ultrafilter on $\BN$, let $p_n\in X$ be a sequence of base-points and let $\{j_n\}$ be a sequence of positive integers. Then, the $\UU$-ultralimit of the sequence $(X,\frac{d}{j_n},p_n)$ is called the asymptotic cone of $X$ with respect to $\UU$, $\{j_n\}_{n\in \BN}$ and $\{p_n\}$  and is denoted $\Con_\UU (X_n,j_n,p_n)$. The point $(p_n)_{n\in \BN}$ is called the \emph{observation point} and the sequence $\{j_n\}$ is called the \emph{scaling sequence}. An asymptotic cone of a group $G$ is simply an asymptotic cone of its Cayley graph.  Note that it is customary to require that $\lim\limits_{n\to \infty} j_n=\infty$. In our setting, we allow for the possibility that the asymptotic cone of a cubing be a (simplicial) metric cubing.

The following properties of ultralimits of metric spaces are well-known, see \cite{KL, Roe}.
\begin{enumerate}
\item If $(X_n,d_n)_{n\in \BN}$ are geodesic metric spaces, then $\lim\limits_\UU(X_n,d_n, p_n)$ is also a geodesic metric space.
\item The ultralimit $\lim\limits_\UU(X_n,d_n, p_n)$ of metric spaces is a complete metric space.
\item Let $\kappa\le 0$ and let $(X_n,d_n)_{n\in \BN}$ be a sequence of $\CAT(\kappa)$-metric spaces. Then, the ultralimit  is also a $\CAT(\kappa)$-space. 
\item Let $(X_n,d_n)_{n\in \BN}$ be a sequence of $\CAT(\kappa_n)$-metric spaces, where $\lim\limits_\UU \kappa_n=-\infty$. Then, $\lim\limits_\UU(X_n,d_n, p_n)$ is a real tree.
\end{enumerate}

We recall that every cubing $X$ can be turned into a metric cubing by endowing every cube of $X$ with the $\ell^1$ metric.

\begin{defn} \label{defn:rcub}
Let $(X_n,\dist_{c_n},b_n)_{n\in \BN}$, $\dist_{c_n}=\frac{\dist}{c_n}$ be a sequence of cubings with fixed based points $b_n$ endowed with the metric $\dist_{c_n}$, and let $\UU$ be a non-principal ultrafilter on $\BN$. Suppose that the widths of $X_i$ are uniformly bounded by a fixed $N\in \BN$. Then, we call the ultralimit $\CC=\lim\limits_\UU(X_n,\dist_{c_n},b_n)$ a \emph{real cubing}. 

Let $\omega\in \UU$ and suppose that $w(X_n)=N$, for all $n\in \omega$. Then, we say that $\CC$ is a real cubing of \emph{width $N$}. Note that the width of a real cubing is well-defined by the properties of the ultrafilters.
\end{defn}

The following proposition follows from well-known results, see \cite{BH, KL, Roe}
\begin{prop}
Let $\CC$ be a real cubing, then $\CC$ is a complete contractible $\CAT(0)$ space.
\end{prop}

As the following example demonstrates, the class of real cubings is very wide.

\begin{expl}\
\begin{enumerate}
\item Any asymptotic cone of a finitely generated partially commutative group is a real cubing since it is an ultralimit of the universal cover of the standard complex of the partially commutative group.
\item Asymptotic cones of nilpotent group are homeomorphic to $\BR^k$ for some $k$, see \cite{Pansu}, hence the asymptotic cone of any nilpotent group is homeomorphic to a real cubing.
\item It is well known that $\SL_2(\BR)$ is quasi-isometric to the direct product of the hyperbolic plane with the real line, hence every asymptotic cone of $\SL_2(\BR)$ is bi-Lipschitz equivalent to the direct product of a real tree and the real line, and, in fact, see \cite{Kar}, every asymptotic cone of the universal cover of $\SL_2(\BR)$ endowed with the Sasaki metric is isometric to a real cubing. 
\item Let $(M,g)$ be a compact connected Riemannian manifold of dimension $m$, where $g$ is a Riemannian metric on $M$. Let $d$ be the metric on $M$ corresponding to $g$, so that $(M,d)$ is a geodesic metric space. Choose a basepoint $p\in M$. Then, the ultralimit (and even the ordinary Gromov-Hausdorff limit) $\lim\limits_\UU (M,nd,p)$ is isometric to the tangent space $T_p M$ of $M$ at $p$ with the distance function on $T_p M$ given by the inner product $g(p)$. Therefore the ultralimit is isometric to the Euclidean space with the standard Euclidean metric and is bi-Lipschitz equivalent to a real cubing.
\item  The asymptotic cone of any toral relatively hyperbolic group is bi-Lipschitz equivalent to a real cubing, \cite{OS, Sisto}.
\end{enumerate}
\end{expl}

We shall need the following lemma, which is well-known and easy to prove.
\begin{lem}\label{lem:limiso}
Let $(X_n, d_n, x_n)$ and $(Y_n, d_n', y_n)$ be a sequence of based metric spaces and let $\UU$ be an ultrafilter. Suppose that $(X_n,d_n, x_n)$ and $(Y_n,d_n',y_n)$ are $(q_n,C)$ quasi-isometric and $\lim\limits_\UU q_n=B$. Then, $\lim\limits_\UU(X_n, d_n, x_n)$ and $\lim\limits_\UU(Y_n, d_n', y_n)$ are $B$-bi-Lipschitz equivalent.
\end{lem}

\begin{rem}
Real trees were introduced independently by Chiswell and Tits, see \cite{Chis} and \cite{Tits}. Chiswell proved the equivalence of the Bass-Serre theory of groups acting on simplicial trees (without inversions of edges) and integer-valued length functions on a group, see \cite{Chiswell}. Following this approach, for a group $G$ with a real-valued length function, Chiswell defined a real tree as a contractible metric space $X$ on which the group $G$ acts.

In his book \cite{serre}, Serre constructed the $\SL_2$-tree associated to a discrete rank one valuation. This tree is the $\SL_2$-case of Bruhat-Tits building. Having groups with more general valuations, Tits generalised the Tits building construction in the $\SL_2$-case to the case of more general valuations.

One can show that the spaces defined by Chiswell and Tits are complete and, in fact, are real cubings. 

Later, Alperin and Moss, \cite{AM} gave a more general definition of real trees as spaces where every two points can be joined by a unique arc or, equivalently, as $0$-hyperbolic spaces. In analogy to the work of Alperin and Moss, one can ask if there is a metric characterisation of real cubings.

There is yet another way to look at real trees: as metric spaces that are isometric to subspaces of the asymptotic cone of a free group, see \cite{MNO, DP}. In this fashion, one could define a real cubing to be a (convex) subspace of the asymptotic cone of a finitely generated partially commutative group. Just as Rips' theory can be viewed as the theory of groups acting on subspaces of the asymptotic cone of a free group, the theory developed in this paper can be regarded as the theory of groups acting on subspaces of asymptotic cones of partially commutative groups. Since the nature of the actions we consider is that of an ultralimit of actions, we restrict our consideration to real cubings that are ultralimits of cubings. 
\end{rem}

\section{Group actions on real cubings}\label{sec:7}
At this point we begin our study of group actions on real cubings. As we noticed in the introduction, intuitively, a ``good'' action on a space should capture its structure, hence in the case of direct product of trees, one would expect to recover some type of direct product structure. However, by examples constructed by Burger and Mozes, see \cite{BM}, studying general, even free actions on direct product of trees (and, more generally, on $\CAT(0)$ cube complexes) is practically hopeless. On the other hand, as we have seen in Sections \ref{sec:3} and \ref{sec:4}, free co-special group actions do preserve the structure of direct product and can be analysed.

It is clear that any (metric) cubing is a real cubing, hence, if one is to understand the structure of groups acting on real cubings, one needs to have a good understanding of group actions on cubings. We therefore consider the class of essentially free co-special actions on real cubings - a natural generalisation of co-special actions on cubings. In a sense, essentially free co-special actions can be likened to very small actions on real trees, where vertex stabilisers also have a very small action on a real tree.

Let $G$ be a group acting co-specially by isometries on a cubing $C$. Let $K=\{g\in G\mid g.x=x \hbox{ for all $x\in C$}\}$ be the kernel of the action. We shall say that the action of $G$ on a cubing is \emph{essentially free} if  $G\ne K$ and the induced action of the group $\factor{G}{K}$ on $C$ is free. 

\begin{defn} \label{defn:essfr}
Let $\CC=\lim\limits_\UU (C_i, b_i, \dist_{c_i})$ be a real cubing. Let $\{\alpha_i\}$ be a sequence of group actions of a group $G$ on the cubings $\{C_i\}$. Suppose that for all $(x_i)\in \CC$ and all $g\in G$ the sequence of elements $(g.x_i)$ is admissible and thus represents an element of $\CC$.  Define an action $\alpha$ of $G$ on $\CC$ as follows: 
$$
\alpha(g.(x_n)) = (\alpha_n(g.x_n)).
$$
We call such an action \emph{limiting} and the actions of $G$ on the $C_i$'s are called the \emph{components} of the limiting action.

Suppose further that we have a limiting action of $G$ on the real cubing $\CC$. We say that $G$ acts \emph{essentially freely and co-specially} on $\CC$ if the action of $G$ on $\CC$ is faithful, non-trivial (i.e. without a global fixed point) and the components of the action are essentially free and co-special actions. 
\end{defn}

In \cite{Guir}, Guirardel showed that, in fact, any stable action of a finitely presented group on a real tree can be approximated by actions on simplicial trees. From this perspective, the theory of (stable) actions on real trees is the theory of ultralimits of actions on simplicial trees, \cite{KL,Roe}. The main example of essentially free co-special action on a real tree (viewed as $1$-dimensional real cubings) is the action of a limit group $G$ via the (discriminating) family of homomorphisms from $G$ to $F$, see \cite{Sela1}. 

We now give the central example of essentially free co-special actions of finitely generated groups on real cubings. 
\begin{expl}
Let $G$ be a finitely generated group and let $\GG$ be a partially commutative group with trivial centre. Let $S=S^{-1}$ be a finite set generating $G$ and let $\dist$ be the edge path metric on $\widetilde{C(\GG)}$. Given  an infinite sequence of homomorphisms $\{\varphi_n: G \to \GG\}$, one can associate to it a sequence of positive integers defined by 
$$
c_n = \max_{x\in \GG} \min_{a\in S} \dist(\varphi_n(a).x, x)=\max_{g\in\GG}\min_{a\in S} \dist(\id, g^{-1}\varphi_n(a)g.\id).
$$
It is well-known that if $(\varphi_n)$ are pairwise non-conjugate in $\GG$, since $\GG$ is finitely generated, then $\lim\limits_{n\to \infty}c_n =\infty$. 
Since the image $\varphi_n(G)$ is a subgroup of $\GG$, the subgroup $\varphi_n(G)$ acts by left multiplication on the Cayley graph $\Cay(G)$ and thus every homomorphism $\varphi_n$ defines an essentially free co-special action of $G$ on $\widetilde{C(\GG)}$.

Choose an ultrafilter $\UU$ and let $a \in S$, $x_n \in \GG$ be so that $c_n = \dist(\varphi_n(a).x_n,x_n)$ for $\UU$-almost all $n$. We then obtain a limiting action of $G$ on the asymptotic cone (real cubing) $\Con_\UU(\GG; (x_n), (c_n))$. By definition, if this action is faithful, then, as we show in Proposition \ref{prop:conult}, it is essentially free and co-special.
\end{expl}

Given a finitely generated group $G=\langle S\rangle$ acting on a space $X$, it is customary to define the displacement of the action as $\sup\limits_{x\in X}\min\limits_{a\in S}\dist(x,a.x)$. As the following lemma shows, if we are \emph{given} an essentially free co-special group action on a real cubing, our definition of displacement and the usual one are basically equivalent. Furthermore, this lemma allows us to change the basepoint of the real cubing.

\begin{lem}
Let $G=\langle S\rangle$ be a finitely generated group acting essentially freely and co-specially on a real cubing $\CC =\lim\limits_\UU(C_n,b_n\dist_{c_n})$. Then, for any $x=(x_n)\in \CC$ one has that
$$
\lim\limits_\UU\left(\frac{\partial_{b_n}}{c_n}\right)=p \quad \hbox{ and } \quad \lim\limits_\UU\left(\frac{\partial_{b_n}}{\partial_{x_n}}\right)=q_x, 
$$
where $p, q_x\in \BR$, $p,q_x >0$.
\end{lem}
\begin{proof}
By definition of the ultralimit, $\lim\limits_\UU(\frac{\partial_{b_n}}{c_n})$ always exists. We are to show that $p=\lim\limits_\UU(\frac{\partial_{b_n}}{c_n})\ne 0, \infty$. Since the action  of $G$ on $\CC$ (and, in particular, the basepoint $(b_n)$) is well-defined, it follows that $\lim\limits_\UU(\frac{\partial_{b_n}}{c_n})\ne\infty$. On the other hand, since the action of $G$ is essentially free and co-special and therefore non-trivial, we conclude that $p\ne 0$.  

The fact that $q_x\in \BR$, $q_x>0$ follows from the first statement.
\end{proof}

Let $G$ be a finitely generated group acting essentially freely and co-specially on a real cubing $\CC$ by an action $\alpha=\{\alpha_i\}$ and let the width $w(\CC)$ of $\CC$ be $N$. Then, by Lemma \ref{lem:width} and Proposition \ref{prop:uniqi}, $\UU$-almost all components of the action $\alpha$ define a subgroup  $\factor{G}{K_i}$ of a partially commutative group $\PP_i=\PP(\factor{G}{K_i})$, where $w(C(\PP_i))=N$ and $\PP_i$ has $2N$ generators. Since there are only finitely many such partially commutative groups, there exists a partially commutative group $\GG(\CC)$ so that:  $\GG(\CC)$ is $2N$-generated, the width $w(C(\GG(\CC)))$ equals $N$ and $\factor{G}{K_i}$ is a subgroup of $\GG(\CC)$ for $\UU$-almost all $i$. Moreover, by Proposition \ref{prop:uniqi}, there is a quasi-isometric embedding of $C_i$ into $\widetilde{C(\GG(\CC))}$. Given an essentially free co-special action of $G$ on $\CC$, we call the group $\GG(\CC)=\GG(\CC,G, \alpha)$ \emph{the partially commutative group of $\CC$}.

\begin{defn}
Let $\GG$ be a partially commutative group and let $\Con_\UU(\GG)$ be its asymptotic cone.  An isometry $t$ of $\Con_\UU(\GG)$ is called a (left) \emph{translation} if there exists an admissible element  $w\in \UG$, $w=(w_i)_{i\in \BN}$ so that for any point $p=(p_i)_{i\in \BN}\in \Con_\UU(\GG)$ the image $t(p)$ equals $(w_ip_i)_{i\in \BN}$. The element $w$ is called the \emph{vector} of the translation $t$. Note that for a translation $t$, the vector $w$ may be defined non-uniquely.
\end{defn}

\begin{lem} \label{lem:translfaith}
Let $\tau_r$ be a translation of $\Con_\UU(\GG)$ with translation vector $r=(r_i)\in \UG$, where $\GG$ is a partially commutative group with trivial centre. Then, the action $\tau_{r}$ is trivial on all the points of the asymptotic cone $\Con_\UU(\GG)$ if and only if the element $r$ is the trivial element of $\GG^{\UU}$.
\end{lem}
\begin{proof}
By definition, the $i$-th component of $\tau_r$ is simply the action on $\widetilde{C(\GG)}$ defined by the left multiplication by $r_i$.

Assume by contradiction that the set $\{i \in \BN \mid r_i\ne 1\}$ belongs to the ultrafilter. For every non-trivial $r_i\in \GG$, we take a canonical generator $y_i$ of $\GG$ such that $y_i$ does not commute with $r_i$ and $y_i$ does not left-divide $r_i$ (see \cite{EKR} for definition). Notice that such a generator $y_i$ exists, since $\GG$ has trivial centre. 

Let $y=(y_i^{c_i})_{i\in \BN}$, where $(c_i)$ is the sequence of the scaling constants, be an element of $\Con_\UU(\GG)$. Consider the distance $\dist(y, r.y)$. Since the canonical generator $y_i$ neither divides nor commutes with $r_i$, it follows from \cite[Corrollary 4.8]{CK1} that ${y_i^{-c_i}}$ left-divides ${y_i^{-c_i}}r_iy_i^{c_i}$, hence $d(y, r.y)\ge 1$  and so the action $\tau_r$ is non-trivial on $\Con_\UU(\GG)$.
\end{proof}

We now show that the sequence of quasi-isometric embeddings  of $C_i$ into $\widetilde{C(\GG(\CC))}$ induces an equivariant bi-Lipschitz embedding of $\CC$ into the asymptotic cone of $\GG(\CC)$.

\begin{prop} \label{prop:uem}
Let $\alpha$ be an essentially free co-special group action of a $k$-generated group $G$ on a real cubing $\CC=\lim\limits_{\UU} (C_n,b_n,\dist_{c_n})$ and let $\{\alpha_n\}$ be the components of this action.  Let $\GG(\CC)$ be the partially commutative group of the real cubing $\CC$, let $\beta_n$ be the action of $G$ on $\widetilde{C(\GG(\CC))}$ induced by $\alpha_n$ and let $\phi_n$ be the embedding  of $\factor{G}{\Ker(\alpha_n)}$ into $\GG(\CC)$ constructed in  {\rm Proposition \ref{prop:uniqi}}. Then, there exists a $G$-equivariant bi-Lipschitz embedding $\psi$ of $\CC$ into $\Con_\UU(\GG(\CC), 1, \dist_{\partial(\beta_n)})$.
\end{prop}
\begin{proof}
Let $\psi_n$ be the equivariant quasi-isometric embedding of $(C_n,b_n,\dist_{\partial({\alpha_n})})$ into $(\widetilde{C(\GG(\CC))},1, \dist_{\partial({\beta_n})})$ constructed in Proposition \ref{prop:uniqi}. Identifying the $1$-skeleton of $\widetilde{C(\GG(\CC))}$ with the Cayley graph of $\GG(\CC)$, the embedding $\psi_n$ maps the base-point $b_n$ to the identity - the base-point of $\widetilde{C(\GG(\CC))}$. 

By Lemma \ref{lem:limiso}, the sequence of maps  $\{\psi_n\}$  gives rise to a bi-Lipschitz embedding $\psi$ of $\CC$ into $\Con_\UU(\GG(\CC), 1, \dist_{\partial(\beta_n)})$ defined by
$$
\psi((y_n))=(\psi_n(y_n)).
$$
Furthermore, since the actions $\beta_n$ are translations, the sequence of actions $\{\beta_n\}$ defines an action $\beta$ of $G$ on $\Con_\UU(\GG(\CC), 1, \dist_{\partial(\beta_n)})$:
$$
\beta(g.(x_n))=(\beta_n(g.x_n)).
$$
It now follows by Lemma \ref{lem:translfaith}, that the action $\beta$ is essentially free and co-special. Since the maps $\psi_n$ are equivariant, it follows that $\psi$ is a $G$-equivariant embedding.
\end{proof}

The corollary below follows by definition of a translation and by construction of the action of $G$  on $\Con_\UU(\GG_N)$ given in Proposition \ref{prop:uem}
\begin{cor}
Let $G$ be a finitely generated group acting essentially freely and co-specially on a real cubing $\CC$. Then, $G$ acts essentially freely and co-specially by translations on an asymptotic cone of a finitely generated partially commutative group.
\end{cor}
\begin{proof}
Observe that the action of $G$ on $\Con_\UU(\GG_N, 1, \dist_{\partial(\beta_n)})$ defined in the proof of Proposition \ref{prop:uem} is by left translations, hence the statement.
\end{proof}

We have shown that essentially free co-special actions on real cubings induce essentially free co-special actions by left translations on asymptotic cones of partially commutative groups. As we have mentioned above, limit groups over partially commutative groups act essentially freely and co-specially by translations on the corresponding asymptotic cones. We now show that, basically, these are the only examples.

\begin{prop} \label{prop:conult}
Let $\GG$ be a partially commutative group with trivial centre. Let $G$ be a finitely generated group acting  faithfully  by translations on $\Con_\UU(\GG)$. Then, $G$ embeds into the ultrapower $\UG$ of $\GG$ and so $G$ is fully residually $\GG$.  

Conversely, let $G$ be a finitely generated fully residually $\GG$ group. Then, $G$ acts essentially freely and co-specially by translations on the asymptotic cone $\Con_\UU(G)$ of $\GG$.
\end{prop}
Note that if $G$ acts essentially freely and co-specially by translations, then $G$ acts faithfully and so $G$ satisfies the assumptions of the above proposition.
\begin{proof}
Let $r=(r_i)\in \GG^\UU$ be an admissible sequence representing an element of the asymptotic cone $\Con_\UU(\GG)$. 
Then, the element $r$ defines an isometry $\tau_r$  of $\Con_\UU(\GG)$. The $i$-th component of $\tau_r$ is simply the action on $\widetilde{C(\GG)}$ defined by left multiplication by $r_i$.

Define the map $\psi: G\to \UG$ by sending every generator of $G$ to its translation vector. The map $\psi$ extends to a homomorphism from $G \to \UG$. Indeed, since $G$ acts on $\Con_\UU(\GG)$, it follows that for any relation $r$ of $G$ and any point $y \in \Con_\UU(\GG)$ we have that the action of $r$ on $y$ is trivial, i.e. $\dist(y,r.y)=0$. By Lemma \ref{lem:translfaith}, it then follows that the translation vector $r$ is, in fact, the trivial element of $\UG$ and hence $\psi$ is a homomorphism.

The homomorphism $\psi$ gives rise to an action of $G$ on $\UG$. Since $G$ acts faithfully on $\Con_\UU(\GG)$, the induced action of $G$ on  $\UG$ is faithful. Therefore, the homomorphism $\psi$ is injective. We conclude that $G$ is fully residually $\GG$ since so is every subgroup of $\UG$, see Theorem \ref{thm:limgrchar}.

On the other hand, any fully residually $\GG$ group $G$ can be viewed as a subgroup of $\UG$. Therefore, $G$ acts faithfully on $\UG$ by left translations. Furthermore, since the group $G$ is finitely generated one can choose scaling constants so that every translation vector corresponding to an element of $G$ is admissible in the asymptotic cone and hence the group $G$ acts on $\Con_\UU(\GG)$. Since the action of $G$ on $\UG$ is faithful, the translation vector $(g_i)$ corresponding to a non-trivial element $g$ of $G$ is non-trivial. By Lemma \ref{lem:translfaith}, the action of $(g_i)$ on the asymptotic cone is non-trivial. It follows that the action of $G$ on $\Con_\UU(\GG)$ is faithful.
\end{proof}

Let $\GG$ be the free abelian group of finite rank $n$ and let $\Con_\UU(\GG)$ be its asymptotic cone. Then, $\Con_\UU(\GG)$ is isometric to $\BR^n$ with the $\ell^1$ metric. It is not hard to see that if a finitely generated group $G$ acts faithfully by translations on $\BR^n$, then $G$ is free abelian.

\begin{cor}
Let  $\GG$ be an arbitrary partially commutative group and $Z(\GG)$ be its centre, i.e. $\GG=\GG'\times Z(\GG)$ and $Z(\GG)= \BZ^n$. Let $G$ be a finitely generated group acting essentially freely by translations on $\Con_\UU(\GG)= \Con_\UU(\GG')\times \BR^n$, then $G$ is a subgroup of a group of the form $G'\times \BZ^m$, where $G'$ is a finitely  generated group acting faithfully by translations on $\Con_\UU(\GG')$.  Hence, $G\simeq G''\times \BZ^l$, where $G''$ is a finitely generated group acting faithfully by translations on $\Con_\UU(\GG')$. Moreover, $G$ is fully residually $\GG$.
\end{cor}
\begin{proof}
Let $G=\langle x_1,\dots, x_k\rangle$ be a set of generators of $G$ and let $v^i$ be the translation vector of $x_i$, $i=1,\dots, k$. Write $v^i=((u^i_n, w^i_n))_{n\in \BN}$, where $u^i_n\in \GG'$ and $w^i_n\in \BZ^n$ for all $i$ and $n$. Consider the subgroup $H$ of isometries  of $\Con_\UU(\GG)$ generated by the $2k$ translations $\{(u^1_n),\dots, (u^k_n), \dots, (w^1_n), \dots, (w^k_n)\}$. It is clear that $H=\langle (u^1_n),\dots, (u^k_n)\rangle \times \langle (w^1_n), \dots, (w^k_n)\rangle$ and that $G$ is a subgroup of $H$. Furthermore, since if $u_n^i\ne 1$ for almost all $n$, by Lemma \ref{lem:translfaith}, there exists an element of the asymptotic cone $\Con_\UU(\GG')$ which is not fixed by the action by $(u_n^i)$. Therefore $G'=\langle (u^1_n),\dots, (u^k_n)\rangle$ acts essentially freely co-specially by translations on $\Con_\UU(\GG')$. Finally,  the group $\langle (w^1_n), \dots, (w^k_n)\rangle$ acts by translations on $\BR^n$ and thus is free abelian.
\end{proof}

We summarize the results of this section in the following theorem.

\begin{thm}\label{thm:gen}
Let $G$ be a finitely generated group. Then
\begin{enumerate}
\item  the group $G$ acts essentially freely and co-specially on a real cubing if and only if $G$ acts essentially freely and co-specially on an asymptotic cone of a partially commutative group  $\GG$ if and only if $G$ is fully residually $\GG$ for some finitely generated partially commutative group;
\item 
the group $G$ acts essentially freely and co-specially on a real cubing  of width $n$ if and only if $G$ acts essentially freely and co-specially on an asymptotic cone of an $n$-wide and $2n$-generated partially commutative group  $\GG$ if and only if $G$ is fully residually $\GG$. 
\end{enumerate}
\end{thm}

\begin{rem}
In the case of free groups, Sela introduced \emph{geometric} limit groups as quotients of finitely generated free groups by the  kernel of an action induced by a family of homomorphisms on a real tree and proved that the class of geometric limit groups is precisely the class of finitely generated fully residually free groups, see \cite{Sela1}. The above theorem gives a characterisation of the class of geometric limit groups over partially commutative groups as the class of finitely generated fully residually partially commutative groups.
\end{rem}

\section{Graph Towers}\label{sec:9}

Limit groups have been studied in various branches of group theory and logic from different viewpoints: algebraic, model-theoretic, geometric, in universal algebra and in algebraic geometry over groups. They can be described as finitely generated fully residually free groups, groups that satisfy the same set of existential formulas as free groups, as coordinate groups of irreducible algebraic sets, as subgroups of the ultrapower of the free group, as groups that realise atomic types over a free group etc., see \cite{DMR}. Limit groups form a subclass of groups acting freely on $\mathbb{R}^n$-trees (and $\mathbb{Z}^n$-trees), \cite{Guir, Rem, KhMNull, MRS}, they are toral relatively hyperbolic, \cite{Dahm, Aliba} and $\CAT(0)$, \cite{Alib}.

The first structural characterisation of limit groups is due to Kharlampovich and Miasnikov. In \cite{KhMNull}, the authors showed that limit groups are subgroups of the, so-called, NTQ groups (in Sela's terminology, see \cite{Sela1}, these groups are called $\omega$-residually free towers). Among other things, Kharlampovich and Miasnikov showed that limit groups are subgroups of an iterated sequence of extensions of centralisers. One of the main corollaries of this characterisation is the fact that limit groups are finitely presented and, more generally, coherent.

Similar results were proven independently by Sela in \cite{Sela1, Sela2}. Sela introduced and studied limit groups via their action on real trees and obtained a description of the hierarchical structure of these groups as well as their finite presentability. Following Sela's approach, Champetier and Guirardel, see \cite{CG}, gave another proof of the fact that limit groups are subgroups of $\omega$-residually free towers and also characetrised the class of limit groups (without requiring to pass to a subgroup) as the smallest class containing finitely generated free groups, and stable under free products and under taking generalized double over a group in the class. In \cite{BFnotes}, Bestvina and Feighn characterised them as constructible limit groups - a hierarchical description in terms of their generalised abelian decomposition and existence of a nice retraction onto a lower level, where free groups, are the groups on the lowest level of the hierarchy.

Partially commutative groups are known to have a highly complex family of subgroups: they contain surface groups (with only 3 exceptions) as well as graph braid groups  as subgroups, see \cite{CrispW}; furthermore, some of their subgroups are main examples of groups with remarkable finiteness properties (see \cite{Stal, Bieri, BB}). In particular, partially commutative groups are not coherent. It is clear, that the structure of limit groups over partially commutative groups is at least as complex as the structure of subgroups of partially commutative groups, hence, one can not expect to completely describe the structure of limit groups over partially commutative groups (without passing to a subgroup). Our goal is to characterise limit groups over partially commutative groups as subgroups of graph towers - a class of groups that generalises the class of $\omega$-residually free towers and NTQ groups.

Recall that a finitely generated group is an $\omega$-residually free tower if it belongs to the smallest class of groups containing all finitely generated free groups and (non-exceptional) surface groups, which is stable under taking free products, free extensions of centralisers, and attaching retracting surfaces along maximal cyclic subgroups.

The class of $\omega$-residually free towers and NTQ groups, not only characterises the class of limit groups, but also plays a crucial role in the classification of groups elementarily equivalent to a free group, \cite{KhMTar, SelaTar}. We expect that the class of graph towers will play a similar role in the classification of groups elementarily equivalent to a given partially commutative group.

We now turn our attention to defining graph towers.  Free extensions of centralisers are, in some sense, the basic operations for constructing limit groups, since, as we  have already mentioned, limit groups are subgroups of iterated sequences of free extensions of centralisers. In order to gain intuition in the case of partially commutative groups, let us review the behaviour of extensions of centralisers of elements of a given partially commutative group $\GG$. On the one hand, we showed in \cite{CK1, CKpc} that any free extension of a centraliser of an irreducible element $b\in \GG$ is discriminated by $\GG$. On the other hand, as the following example shows, in general, free extensions of centralisers of non-irreducible elements are not discriminated by $\GG$. 

\begin{expl} \label{expl:bld}
Let $\GG=F(a,b)\times F(c,d)$, and let $w=(ab,cd)\in \GG$. Consider the extension of centraliser of $w$, $H=\langle \GG, t\mid [t,C(w)]=1\rangle$. Take the element $v=[a^t, c]$ in $H$. For any homomorphism $\phi$ from $H$ to $\GG$ which fixes $\GG$, we have that $\phi(t) \in C(w)=\langle ab\rangle\times \langle cd\rangle$. It follows that $\phi(v)=1$. Hence, $H$ is not discriminated (even separated) by $\GG$.
\end{expl}

In fact, one can give a precise characterisation of the set of elements of a partially commutative group $\GG$ whose extension of centraliser is discriminated by $\GG$ in terms of their block decompositions.

The above discussion already demonstrates that, on the one hand, one can not expect that decompositions over abelian groups (or small groups, or slender groups) determine the structure of limit groups over partially commutative groups, since in the free extension of a centraliser $H= \GG \ast_{C(b)} \langle C(b),t \mid [t,C(b)]\rangle$ the amalgamation is taken over a centraliser of an irreducible element, which in general, is a partially commutative group. On the other hand, it shows that there are some constraints on the type of elements whose centralisers need to  be extended.

The requirements for attaching a (non-exceptional) surface with boundary to $\GG$ are similar: the boundary components of the surface are glued along maximal cyclic subgroups generated by irreducible elements $b_1, \dots, b_k$. In the case of free groups, the surface retracts onto the free group and the discriminating family is obtained from the retraction by pre-composing it with modular automorphisms. If one would like to follow this pattern in the case of partially commutative groups, firstly, one needs to ensure the existence of ``modular'' automorphisms. For this, a sufficient condition is that $\BA(b_i)=\BA(b_j)$, $i,j=1,\dots,k$. Secondly, in order for the surface group to be discriminated into the subgroup generated by the irreducible elements $b_i$, $i=1,\dots, k$, it is necessary to impose the condition that the surface group commute with $\BA(b_i)$. 

It is not hard to see that these conditions are sufficient for the corresponding group to be discriminated by $\GG$. In fact, we shall show that, basically, these conditions are necessary.

Roughly speaking, graph towers are built hierarchically from the partially commutative group $\GG$ by gluing retracting abelian groups and surface groups with the conditions that 
\begin{itemize}
\item the abelian and surface groups are amalgamated along cyclic subgroups generated by ``irreducible'' elements and
\item these abelian and surface groups commute with the centraliser of the subgroup onto which they retract.
\end{itemize}

As we discussed in the introduction, graph towers can be defined as an iterated sequence built inductively using, basically, amalgamated products over certain centralisers, see Lemma \ref{lem:prtower}. Although the idea is clear, the main technical problem is that, a priori, we do not have control on the structure of centralisers of elements of graph towers. In particular, it is not clear what an ``irreducible'' element is. Informally, given a limit group over a partially commutative group, one could define an irreducible element to be an element for which there exists a discriminating family that maps it into irreducible elements. But we can not define it this way either, since when constructing graph towers, we do not yet know that they are discriminated by $\GG$ and in fact, it is only a posteriori that we are able to conclude that the homomorphisms induced by solutions that factor through the fundamental branch of the Makanin-Razborov diagram discriminate the graph tower into $\GG$ and so to deduce that our ``irreducible'' elements have this property.

To deal with this difficulty, graph towers are defined as a triple: the group $\Ts$ (the graph tower), a partially commutative group $\HH$ and an epimorphism from $\HH$ to $\Ts$. We will show in the following section that in fact, both $\Ts$ and $\HH$ are fully residually $\GG$ groups.

The partially commutative group $\HH$ is defined via its commutation graph $\Gamma$. We subdivide the set of edges of the graph $\Gamma$ into two disjoint sets: $E_c(\Gamma)$ and $E_d(\Gamma)$. These two sets capture the different nature of commutation for the images of the elements of $\HH$ in $\GG$, namely two vertices $x$ and $y$ are joined by an edge from $E_c(\Gamma)$ if and only if for a discriminating family, $x$ and $y$ are sent to the same cyclic subgroup in $\GG$; and two vertices $x$ and $y$ are joined by an edge from $E_d(\Gamma)$ if and only if for a discriminating family, the images of $x$ and $y$ disjointly commute. (We also prove the existence of a discriminating family with this property.)

Now we can define irreducible elements in $\Ts$ as the images of irreducible elements in $\HH$. More generally, we define \cool subgroups $\KK$ as closed subgroups of $\HH$ (i.e. $\KK^{\perp\perp}=\KK$) so that $\KK^\perp$ is $E_d(\Gamma)$-directly indecomposable. The motivation behind this definition is as follows. Notice that since in a given finitely generated partially commutative group $\GG$ there are finitely many different canonical parabolic subgroups, so for a canonical parabolic subgroup $\KK$ of $\HH$ (where $\HH$ is discriminated by $\GG$) there exists a subgroup $\GG_\KK$ of $\GG$ so that for a discriminating family $\{\varphi_i\}$ we have that $\varphi_i(\KK)<\GG_\KK$ and for no proper subgroup $\GG_\KK'$ of $\GG_\KK$ there exists such a discriminating family. We show that given a \cool subgroup $\KK$ of $\HH$, there exists a discriminating family so that $\GG_\KK$ is \cool in $\GG$. In particular, $\GG_{\KK^\perp}$ is a directly indecomposable subgroup of $\GG$ and if $\GG_{\KK^\perp}$ is not cyclic, then $C_\GG(\GG_{\KK^\perp})=\GG_\KK$. For intuition, it is helpful to think that the image of the subgroup $\KK^\perp$ in $\Ts$ is the subgroup onto which the corresponding abelian (or surface) group retracts and so $\GG_{\KK^\perp}$ is the directly indecomposable canonical parabolic subgroup where the abelian (or surface) subgroup is mapped by the discriminating family.

\begin{defn}\label{defn:tower}
To any graph tower $\Ts$, we associate a partially commutative group $\HH$ and an epimorphism $\pi:\HH \to \Ts$.  The partially commutative  group $\HH$ is defined via its commutation graph $\Gamma$. The set of edges $E(\Gamma)$ of the graph $\Gamma$ is subdivided into two disjoints sets, the set of $d$-edges and the set of $c$-edges, i.e. $E(\Gamma)=E_d(\Gamma)\cup E_c(\Gamma)$, $E_d(\Gamma)\cap E_c(\Gamma)=\emptyset$. Furthermore, the set of edges $E_c(\Gamma)$ has the following property:
\begin{equation} \label{eq:ecprop}
\hbox{ if } (x,y),(y,z)\in E_c(\Gamma), \hbox{ then } (x,z)\in E_c(\Gamma).
\end{equation}
For every subgroup $\KK$ of $\HH$, abusing the notation, we denote the image $\pi(\KK)$ of $\KK$ in $\Ts$ by $\KK$.

We define a $\GG$-graph tower as an iterated sequence. We denote graph towers by $\Ts$ and write $\Ts^l$, to indicate that the graph tower $\Ts^l$ is of height $l$.

A $\GG$-graph tower $\Ts^0$ of height $0$ is  our fixed partially commutative group $\GG$. In this case, the  partially commutative group associated to $\Ts^0$ is also $\GG=\GG(\Gamma_0)$, all edges of $\Gamma_0$ are $d$-edges and the epimorphism $\pi_0$ is the identity. 

Assume that $\Ts^{l-1}$ is a $\GG$-graph tower of height $l-1$,  $\GG(\Gamma_{l-1})$ is its associated partially commutative group and  $\Ts^{l-1}= \factor{\GG(\Gamma_{l-1})}{\ncl\langle S_{l-1}\rangle}$.
A $\GG$-graph tower $\Ts^l$ of height $l$, the associated partially commutative group $\GG(\Gamma_{l})$ and the epimorphism $\pi_l:\GG(\Gamma_{l}) \to \Ts^l$ are constructed using one of the following alternatives.

\subsubsection*{Basic Floor}
The graph $\Gamma_l$ is defined as follows
\begin{itemize}
\item $V({\Gamma_l})=V({\Gamma_{l-1}}) \cup \{ x_1^l,\dots, x_{m_l}^l \}$; 
\item $E_d(\Gamma_l)= E_d(\Gamma_{l-1}) \cup \{ (x_i^l, x^{l-1}), i=1,\dots, m_l, x^{l-1} = \az(\KK)\}$, where  $\KK$ is an $E_d(\Gamma_{l-1})$-\cool subgroup of $\GG(\Gamma_{l-1})$; 
\item $E_c(\Gamma_l)=E_c(\Gamma_{l-1})$ if the subgroup $\KK^\perp$ is directly indecomposable (as a subgroup of $\GG(\Gamma_{l-1})$) and 
\item $E_c(\Gamma_l)=E_c(\Gamma_{l-1}) \cup \{(x_i^l,x_j^l)\mid 1\le i<j \le m_l\} \cup \{(x_i^l,x^{l-1})\mid 1\le i\le m, x^{l-1}= \az(\KK^\perp)\}$ if the subgroup $\KK^\perp$ is directly decomposable (we will show in Lemma \ref{lem:typesKperp}, that, in this case, $\KK^\perp$ is $E_c(\Gamma_{l-1})$-abelian).
\end{itemize}
We set $\GG(\Gamma_l)$ to be the associated partially commutative group. It follows from the definition that $\GG(\Gamma_{l-1})$ is a retraction of $\GG(\Gamma_l)$ and so is $\Ts_{l-1}$ of $\factor{\GG(\Gamma_{l})}{\ncl\langle S_{l-1}\rangle}$.

The group $\Ts^l$ is a quotient of  $\factor{\GG(\Gamma_{l})}{\ncl\langle S_{l-1}\rangle}$ by $\ncl\langle S(x^l,\GG(\Gamma_{l-1}))\rangle$, i.e. $\Ts^l= \factor{\GG(\Gamma_{l})}{\ncl \langle S_{l-1},S \rangle}$, where the set of relations $S$ is:
\begin{itemize}
\item the set of \emph{basic relations} $[C_{\Ts^{l-1}}(\KK^\perp),x_i^{l}] = 1$, $1\le i\le m_l$ .
\end{itemize}

\subsubsection*{Abelian Floor} 
The graph $\Gamma_l$ is defined as follows:
\begin{itemize}
\item $V({\Gamma_l})=V({\Gamma_{l-1}}) \cup \{ x_1^l,\dots, x_{m_l}^l \}$; 
\item $E_d(\Gamma_l)= E_d(\Gamma_{l-1}) \cup \{ (x_i^l, x^{l-1}), i=1,\dots, m_l, x^{l-1} =\az(\KK)\}$, where  $\KK$ is an $E_d(\Gamma_{l-1})$-\cool subgroup of $\GG(\Gamma_{l-1})$;
\item $E_c(\Gamma_l)=E_c(\Gamma_{l-1}) \cup \{(x_i^l,x_j^l)\mid 1\le i<j \le m_l\} $ if the subgroup $\KK^\perp$ is directly indecomposable (as a subgroup of $\GG(\Gamma_{l-1})$) and 
\item $E_c(\Gamma_l)=E_c(\Gamma_{l-1}) \cup \{(x_i^l,x_j^l)\mid 1\le i<j \le m_l\} \cup \{(x_i^l,x^{l-1})\mid 1\le i\le m, x^{l-1}= \az( \KK^\perp)\}$  if the subgroup $\KK^\perp$ is  directly decomposable (we will show in Lemma \ref{lem:typesKperp}, that, in this case, $\KK^\perp$ is $E_c(\Gamma_{l-1})$-abelian).
\end{itemize}

We set $\GG(\Gamma_l)$ to be the associated partially commutative group. It follows from the definition that $\GG(\Gamma_{l-1})$ is a  retraction of $\GG(\Gamma_l)$ and so is $\Ts^{l-1}$ of $\factor{\GG(\Gamma_{l})}{\ncl\langle S_{l-1}\rangle}$.

The group $\Ts^l$ is a quotient of  $\factor{\GG(\Gamma_{l})}{\ncl\langle S_{l-1}\rangle}$ by $\ncl\langle S(x^l,\GG(\Gamma_{l-1}))\rangle$, i.e. $\Ts^l= \factor{\GG(\Gamma_{l})}{\ncl \langle S_{l-1},S \rangle}$, where the set of relations $S$ is one of the following types:
\begin{itemize}
\item the relations $[C_{\Ts^{l-1}}(u),x_i^{l}] = 1$, where $u\in \KK^{\perp}<\GG(\Gamma_{l-1})$ is a non-trivial cyclically reduced root block element; 
\item the set of basic relations.
\end{itemize}

\subsubsection*{Quadratic Floor}
The graph $\Gamma_l$ is defined as follows:
\begin{itemize}
\item $V({\Gamma_l})=V({\Gamma_{l-1}}) \cup \{ x_1^l,\dots, x_{m_l}^l \}$; 
\item $E_d(\Gamma_l)= E_d(\Gamma_{l-1}) \cup \{ (x_i^l, x^{l-1}), i=1,\dots, m_l, x^{l-1} = \az(\KK)\}$, where  $\KK$ is a $E_d(\Gamma_{l-1})$-\cool subgroup of $\GG(\Gamma_{l-1})$;
\item $E_c(\Gamma_l)=E_c(\Gamma_{l-1})$ if the subgroup $\KK^\perp$ is directly indecomposable (as a subgroup of $\GG(\Gamma_{l-1})$) and 
\item $E_c(\Gamma_l)=E_c(\Gamma_{l-1}) \cup \{(x_i^l,x_j^l)\mid 1\le i<j \le m_l\} \cup \{(x_i^l,x^{l-1})\mid 1\le i\le m, x^{l-1}= \az(\KK^\perp)\}$ if the subgroup $\KK^\perp$ is  directly decomposable (we will see in Lemma \ref{lem:typesKperp}, that, in this case, $\KK^\perp$ is $E_c(\Gamma_{l-1})$-abelian).
\end{itemize}

We set $\GG(\Gamma_l)$ to be the associated partially commutative group. It follows from the definition that $\GG(\Gamma_{l-1})$ is a retraction of $\GG(\Gamma_l)$ and so is $\Ts^{l-1}$ of $\factor{\GG(\Gamma_{l})}{\ncl\langle S_{l-1}\rangle}$.

The group $\Ts^l$ is a quotient of  $\factor{\GG(\Gamma_{l})}{\ncl \langle S_{l-1}\rangle}$ by $\ncl \langle S(x^l,\GG(\Gamma_{l-1} \rangle))$, i.e. $\Ts^l= \factor{\GG(\Gamma_{l})}{\ncl \langle S_{l-1},S \rangle}$, where the set of relations $S$ consists of the set of basic relations 
$\{[C_{\Ts^{l-1}}(\KK^\perp),x_i^{l}] = 1\mid 1\le i\le m_l\}$ and a relation $W$ of one of the two following forms: 
(orientable)
\begin{equation}\label{eq:orient}
\begin{array}{l}
\phantom{a}[x_1,x_2]\cdots[x_{2g-1},x_{2g}]{u_{2g+1}}^{x_{2g+1}}\cdots {u_m}^{x_{m}}=\\
\phantom{a}[v_1,v_2]\cdots[v_{2g-1},v_{2g}]{u_{2g+1}}^{w_{2g+1}}\cdots {u_m}^{w_{m}}
\end{array}
\end{equation}
or (non-orientable)
\begin{equation}\label{eq:nonorient}
x_1^2\cdots x_{2g}^2 {u_{2g+1}}^{x_{2g+1}}\cdots {u_m}^{x_{m}}=v_1^2 \cdots v_{2g}^2 {u_{2g+1}}^{w_{2g+1}}\cdots {u_m}^{w_{m}},
\end{equation}
where $u_i, v_j, w_k \in \KK^\perp$ and 
\begin{itemize}
\item[$\circledast$] either the Euler characteristic of $W$ is at most -2, or $W$ corresponds to a punctured torus and the subgroup $\langle u_i, v_j, w_k \mid i,k=2g+1,\dots, m, j=1,\dots, 2g, \rangle$ of $\Ts^{l-1}$ is non-abelian, i.e. the retraction of the (punctured) surface onto $\Ts^{l-1}$ is non-abelian;
\item [$\circledast\circledast$]  or $g+m \ge 2$ and
\begin{itemize}
\item  the subgroup $\langle [v_1,v_2], \dots, [v_{2g-1},v_{2g}], {u_{2g+1}}^{w_{2g+1}},\dots, {u_m}^{w_{m}} \rangle$  is non-abelian, where $W$ is orientable or,
\item the subgroup $\langle v_1^2, \dots, v_{2g}^2, {u_{2g+1}}^{w_{2g+1}},\dots, {u_m}^{w_{m}}\rangle$ is non-abelian, where $W$ is non-orientable,
\end{itemize} 
i.e. the solution is not atom-commutative (see \cite[Definition 11]{KhMIrc}).
\end{itemize}
\end{defn}

\begin{lem} \label{lem:prtower}
The graph tower $\Ts^l$ admits one of the following decomposition as amalgamated product:
\begin{itemize}
\item [a1)] $\Ts^{l-1} \ast_{C_{\Ts^{l-1}}(\KK^\perp)} \left(C_{\Ts^{l-1}}(\KK^\perp)\times\langle x_1^l, \dots, x_{m_l}^l \rangle\right)$ \\
{\rm(}if $S$ is empty and $\KK^\perp$ is non-abelian{\rm)};

\item [a2)] $\Ts^{l-1} \ast_{C_{\Ts^{l-1}}(\KK^\perp)} \left(C_{\Ts^{l-1}}(\KK^\perp) \times \langle x_1^l, \dots, x_{m_l}^l \mid [x_i^l,x_j^l]=1, 1\le i,j \le m_l, i\ne j \rangle\right)$ \\ {\rm(}if $S$ is empty and $\KK^\perp$ is abelian{\rm)};

\item [b1)] $\Ts^{l-1} \ast_{C_{\Ts^{l-1}}(u)} \left(C_{\Ts^{l-1}}(u) \times \langle x_1^l, \dots, x_{m_l}^l \mid [x_i^l,x_j^l]=1, 1\le i,j \le m_l, i\ne j \rangle\right)$ \\ {\rm(}if $S$ is abelian and $u$ is non-trivial{\rm)};

\item [b2)] $\Ts^{l-1} \ast_{C_{\Ts^{l-1}}(\KK^\perp)} \left(C_{\Ts^{l-1}}(\KK^\perp) \times \langle x_1^l, \dots, x_{m_l}^l \mid [x_i^l,x_j^l]=1, 1\le i,j\le m_l, i\ne j \rangle\right)$ \\ {\rm(}if $S$ is abelian and $\KK^\perp$ is non-abelian{\rm)};

\item [c)] $\Ts^{l-1} \ast_{C_{\Ts^{l-1}}(\KK^\perp)\times \langle u_{2g+1},\dots, u_m\rangle} \left( \langle u_{2g+1}, \dots, u_m, x_1^l, \dots, x_{m_l}^l \mid W \rangle \times C_{\Ts^{l-1}}(\KK^\perp)\right)$ \\ {\rm(}if $W$ satisfies one of the properties $\circledast$ and $\circledast\circledast$ from {\rm Definition \ref{defn:tower})};
\end{itemize}
\end{lem}
\begin{proof}
The only decomposition that does not follow immediately from the definition of a graph tower is c). From the definition, we have that
$$
\Ts^{l-1} \ast_{ \langle C_{\Ts^{l-1}}(\KK^\perp),u_{2g+1},\dots, u_m\rangle} \left( \langle C_{\Ts^{l-1}}(\KK^\perp),  u_{2g+1}, \dots, u_m, x^l \mid W=1, [ C_{\Ts^{l-1}}(\KK^\perp), x^l]=1\rangle\right),
$$
where $x^l= \{ x_1^l,\dots,x_{m_l}^l\}$.
We shall prove in Section \ref{sec:11}, that the subgroup $\langle C_{\Ts^{l-1}}(\KK^\perp),u_{2g+1},\dots, u_m\rangle$ is the direct product of  $C_{\Ts^{l-1}}(\KK^\perp)$ and $\langle u_{2g+1},\dots, u_m\rangle$ and hence decomposition c).
\end{proof}

\begin{rem}\
\begin{itemize}
\item Since graph towers can be described as iterated sequences of amalgamated products, each graph tower $\Ts^i$ naturally embeds into the graph tower $\Ts^{i+1}$. In fact, for each floor, there is a natural retraction of $\Ts^{i+1}$ onto $\Ts^{i}$.

\item If the the graph tower is constructed using only basic floors, then $\Ts$ and $\HH$ coincide, i.e. the graph tower is itself a partially commutative group.

\item Notice that graph towers are a natural generalization of the notions of $\omega$-residually free towers and NTQ-groups. Indeed, let $\GG$ be a free group and assume by induction that $\Ts^{l-1}$ is an NTQ-group.

By assumption, in cases a1), b2) and c) the group $\KK^\perp$ is non-abelian. In particular, since limit groups have the CSA property, it follows that the centraliser $C_{\Ts^{l-1}}(\KK^\perp)$ is trivial. Therefore, case a1) corresponds to the free product of the group $\Ts^{l-1}$ and a free group; case b2) corresponds to the free product of $\Ts^{l-1}$ and a free abelian group; and case c) corresponds to an amalgamated product of $\Ts^{l-1}$ with a surface group (so that the natural retraction of the surface to $\Ts^{l-1}$ is non-abelian).

In cases a2) and b1), the subgroup $\KK^\perp$ is abelian. Hence, these cases correspond to extension of centralisers of maximal abelian subgroups, or equivalently, to an amalgamated product of $\Ts^{l-1}$ and a free abelian group (amalgamated by a maximal (in $\Ts^{l-1}$) abelian subgroup).

We therefore obtain an NTQ-group. Notice that the difference with an $\omega$-residually free tower is that our construction is not canonical in the following sense. In the construction of an $\omega$-residually free tower, at a given floor one attaches all the pieces corresponding to the abelian JSJ-decomposition. In our case, pieces are attached one by one, hence the same graph tower can be constructed in several different ways (for example, at the same level one can first attach  an abelian group and then a surface group or vice versa).

\item Notice that, in general, the centraliser of a set of elements in a partially commutative group  is neither  commutative transitive, nor malnormal, nor small (see \cite{DKRpar}). Hence, the splittings we find do not correspond to the (abelian) JSJ decomposition, or, in terms of \cite{BFnotes}, the  splitting we obtain is not a GAD. Nevertheless, the decomposition we find plays the role analogous to the one the JSJ plays for limit groups: one can define modular automorphisms and show that under some restrictions on the decomposition (similar to the ones for constructible limit groups), the groups we consider are discriminated using the retraction and appropriate modular automorphisms. 

\item  As mentioned above, edge groups that appear in the decomposition fail the necessary conditions for the existing JSJ theories: they are not abelian, slender or small etc. On the other hand, the most general theory developed by Guirardel and Levitt, see \cite{GuiLe}, requires the group to be finitely presented, while our groups do not need to be finitely presented. 

Nevertheless, it is conceivable that under appropriate definitions, the decomposition we obtain could correspond to a JSJ-decomposition where (in the terminology of \cite{GuiLe}) flexible vertex stabilisers are extensions of 2-orbifold groups with boundary (by centralisers), attached to the rest of the group in a particular way. 

As we mentioned above, the construction of the graph tower is not canonical. In the case of free groups, Kharlampovich and Miasnikov, \cite{EJSJ}, described an elimination process based on the Makanin-Razborov process that gives an effective construction of the (cyclic)  JSJ decomposition of a fully residually free group. If an appropriate JSJ theory is developed for the class of groups we deal in this paper, then one could modify the process developed in \cite{CKpc} to obtain a canonical construction of the graph tower. 
\end{itemize}
\end{rem}

\begin{figure}[!h]
  \centering
   \includegraphics[keepaspectratio,width=4in]{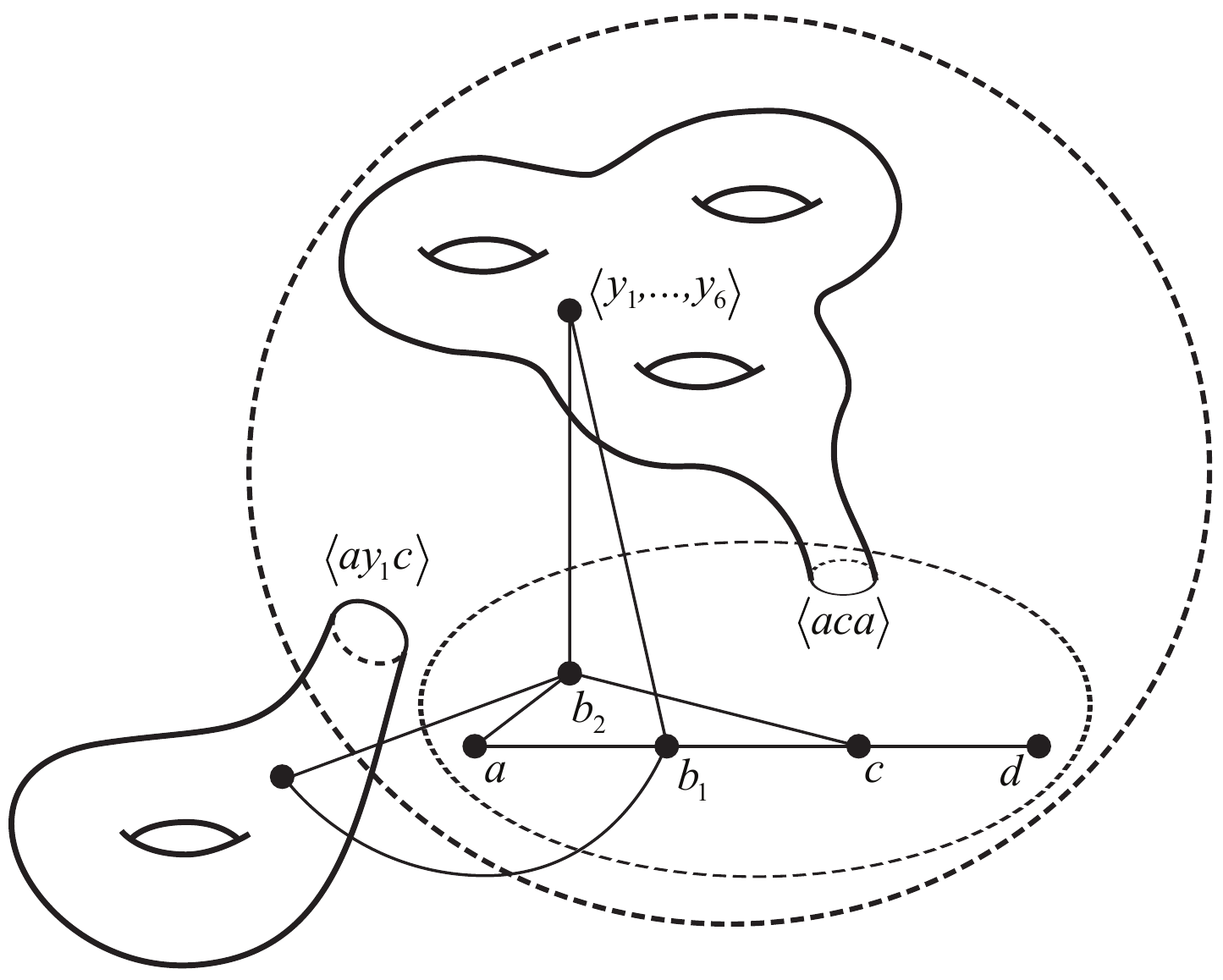}
\caption{Graph tower of height 2. Edges of the graph correspond to commutation of the generators.} \label{pic:tower}
\end{figure}

\section{Graph Tower Associated to the Fundamental Branch}\label{sec:10}

The goal of this and the next sections is to show that, for a given limit group $G$ over a partially commutative group $\GG$, one can effectively construct an embedding of $G$ into a graph tower. In this section, we construct a graph tower $\Ts$ associated to $G$. In the next section, following the approach from \cite{KhMNull}, we show that the graph tower $\Ts$ is discriminated by $\GG$ and use this fact to prove that $G$ embeds into $\Ts$.

As we have already mentioned, the construction of the  graph tower strongly relies on the process described in \cite{CKpc}, which given a finitely generated group $G$ and a partially commutative group $\GG$, produces a diagram that describes the set of all homomorphisms from $G$ to $\GG$. We now informally recall some of the notions and results from \cite{CKpc}.

Given a finitely generated group $G$ and a partially commutative group $\GG$, one can effectively construct finitely many constrained generalised equations  $\Omega_0, \dots, \Omega_m$ (see definition below) and homomorphisms $\pi_i$ from $G$ to the coordinate group of the generalised equation $\Omega_i$ so that any homomorphism from $G$ to $\GG$ factors through the coordinate group of a generalised equation $\Omega_i$ for some $i=0,\dots,m$, i.e. for any homomorphism $\varphi:G\to \GG$, there exist $i$ and a homomorphism $\varphi': G_{R(\Omega_i)}\to \GG$ so that $\varphi=\pi_i \varphi'$.

Informally, a generalised equation $\Upsilon$ is a combinatorial object which encodes a system of equations over a free monoid. 

\begin{figure}[!h]
  \centering
   \includegraphics[keepaspectratio,width=3.2in]{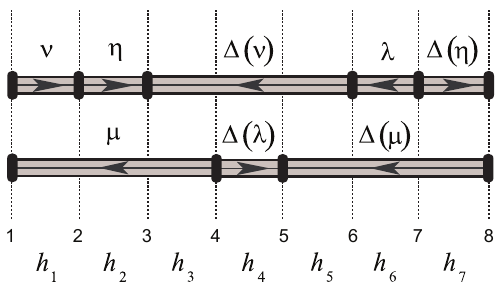}
\caption{Simple quadratic generalised equation} \label{pic:ge}
\end{figure}

Items $h_i$ correspond to variables and each pair of dual bases $\mu, \Delta(\mu)$ defines the equation $(h_{\alpha(\mu)} \cdots h_{\beta(\mu)-1})^{\varepsilon(\mu)}=(h_{\alpha(\Delta(\mu))} \cdots h_{\beta(\Delta(\mu))-1})^{\varepsilon(\Delta(\mu))}$. A tuple $H=(H_1,\dots, H_{\rho_\Upsilon})$ is a solution of the generalised equation $\Upsilon$ if $H_{\alpha(\mu)} \cdots H_{\beta(\mu)-1}$ and $H_{\alpha(\Delta(\mu))} \cdots H_{\beta(\Delta(\mu))-1}$ are freely reduced words so that $(H_{\alpha(\mu)} \cdots H_{\beta(\mu)-1})^{\varepsilon(\mu)}$ is graphically equal to $(H_{\alpha(\Delta(\mu))} \cdots H_{\beta(\Delta(\mu))-1})^{\varepsilon(\Delta(\mu))}$, for every pair of dual bases $\mu, \Delta(\mu)$, see \cite[Chapter 3]{CKpc} for details. 

A constrained generalised equation $\Omega=\langle \Upsilon, \Re_\Upsilon\rangle$ is a generalised equation $\Upsilon$ together with a set of commutation constraints $\Re_\Upsilon\subseteq h\times h$ associated to the variables $h$ of $\Upsilon$. A tuple $H=(H_1,\dots, H_{\rho_\Omega})$ is a solution of the generalised equation $\Omega$ if $H_{\alpha(\mu)} \cdots H_{\beta(\mu)-1}$ and $H_{\alpha(\Delta(\mu))} \cdots H_{\beta(\Delta(\mu))-1}$ are geodesic words in $\GG$ so that $(H_{\alpha(\mu)} \cdots H_{\beta(\mu)-1})^{\varepsilon(\mu)}$ is graphically equal to $(H_{\alpha(\Delta(\mu))} \cdots H_{\beta(\Delta(\mu))-1})^{\varepsilon(\Delta(\mu))}$, for every pair of dual bases $\mu, \Delta(\mu)$, and $H$ satisfies the commutation constraints, i.e. $H_i\lra H_j$ for all commutation constraints $\Re_\Upsilon(h_i,h_j)$. 

The coordinate group $G_{R(\Upsilon)}$ of $\Upsilon$ is the group $\factor{\GG[h]}{R(\Upsilon)}$, where, abusing the notation, we denote by $\Upsilon$ the system of equations associated to the generalised equation $\Upsilon$.
The coordinate group $G_{R(\Omega)}$ of $\Omega$ is the group $\factor{\GG[h]}{R(\Omega)}$, where, abusing the notation, we denote by $\Omega$ the system of equations associated to the generalised equation $\Upsilon$ and the set of commutation relations $[h_i,h_j]$ for all $h_i$, $h_j$ so that $\Re_\Upsilon(h_i,h_j)$. We denote by $G_\Upsilon$ (by $G_\Omega$) the group $\factor{\GG[h]}{\ncl\langle\Upsilon\rangle}$ (the group $\factor{\GG[h]}{\ncl\langle\Omega\rangle}$). 

In this section we primarily work with the groups $G_{\Omega}$. However, the notions of canonical homomorphisms, fundamental branch etc. are defined using coordinate groups of generalised equations. All these notions carry over from coordinate groups to the groups $G_{\Omega}$ and we use them in this context. For instance, by construction, the canonical homomorphism $\pi(v_0,v_1):G_{R(\Omega_0)}\to G_{R(\Omega_1)}$ is induced by a homomorphism from $G_{\Omega_0}$ to $G_{\Omega_1}$ which, abusing the notation, we also denote by $\pi(v_0, v_1)$. Similarly, any solution $H$ of $\Omega$ is induced by a homomorphism from $G_\Omega$ to $\GG$, which we also denote by $H$.

The description of the set of homomorphisms $\Hom(G,\GG)$ reduces to the study of the sets of solutions of the constrained generalised equations $\Omega_i$, $i=0,\dots,m$.

Given a generalised equation $\Omega$, we described a process (analogous to the one described by Makanin, \cite{Mak82, Mak84} and Razborov, \cite{Razborov1, Razborov2}, in the case of free groups) that produces a (possibly) infinite, locally finite, rooted tree $T(\Omega)$ which encodes the solution set of $\Omega$. Infinite branches of this tree are of three specific types: linear, quadratic or general type. The dynamics of these infinite branches determine specific families of automorphisms of the coordinate group associated to $\Omega$. One then shows that, using these families of automorphisms, one can produce a finite tree $T_{\sol}(\Omega)$ that encodes the set of solutions of $\Omega$, see \cite[Theorem 9.2]{CKpc}. 

If the group $G$ is a limit group over $\GG$, since the corresponding solution tree $T_{\sol}$ is finite, it follows by Lemma \ref{lem:discfam}, that there exists a branch of the solution tree $T_{\sol}$ so that the subfamily of homomorphisms that factors through this branch is a discriminating family of $G$ into $\GG$. We call this branch \emph{a fundamental branch of $G$}.  Any discriminating family of homomorphisms that factors through the fundamental branch is called a \emph{fundamental sequence}.

The goal of this section is, given a fundamental branch $\Omega_0\to\dots\to \Omega_q$ associated to the group $G$, to construct the graph tower $\Ts_0$ as well as a homomorphism $\tau_0$ from $G_{\Omega_0}$ to $\Ts_0$ so that for every homomorphism $\pi_0 H$ from $G$ to $\GG$ of a discriminating family there exists a homomorphism $H'$ from $\Ts_0$ to $\GG$ such that the following diagram is commutative:
\begin{equation} \label{diag:period}
\xymatrix@C3em{
G \ar[d]_{\pi_0}\\
G_{R(\Omega_0)}  \ar[dr]_{H} & \ar[l] G_{\Omega_0}\ar[d]_H \ar[r]^{\tau_0} &\Ts_0\ar[dl]^{H'} \\
&\GG 
}
\end{equation}

The graph tower is constructed from a fundamental branch of $G$ using induction on the height of the branch. One begins with the group associated to the leaf of the fundamental branch, which is a partially commutative group discriminated by $\GG$ and keeps building the graph tower according to the type of epimorphism associated to the edges (or more precisely, the type of automorphism group associated to the vertices) of the fundamental branch. The types of automorphisms are intrinsically related to the different dynamics that appear in the construction of the tree $T(\Omega)$: linear type (exotic, Levitt, or type 7-10), quadratic type (surface, or type 12) and general type (axial + surface, or type 15).

\label{pagen}
The induction hypothesis at step $q$ are as follows:
\begin{itemize}
\item [IH] Given a generalised equation $\Omega_q$, there exists a graph tower $(\Ts_q,\HH_q)$, where $\HH_q=\GG(\Gamma_q)$, and a homomorphism $\tau_q$ from $G_{\Omega_q}$ to $\Ts_q$ such that for every solution $H^{(q)}$ of the fundamental sequence, there exists a homomorphism ${H^{(q)}}':\Ts_q \to \GG$ that makes the diagram 
$$
\xymatrix@C3em{
G_{\Omega_q}\ar[dr]_{H^{(q)}} \ar[rr]^{\tau_q} & &\Ts_q\ar[dl]^{{H^{(q)}}'} \\
&\GG& 
}
$$
commutative; notice that if $\pi_q:\HH_q \to \Ts_q$, then $\pi_q {{H^{(q)}}'}$ is a homomorphism from $\HH_q$ to $\GG$ that, abusing the notation, we also denote by ${{H^{(q)}}'}$.

\item [IH1] for all generators $x_i,x_j \in \HH_q$, we have that $(x_i, x_j) \in E_d(\Gamma_q)$ if and only if ${H^{(q)}}'(x_i) \lra {H^{(q)}}'(x_j)$ for all homomorphisms ${H^{(q)}}'$ induced by a solution ${H}^{(q)}$ from the fundamental sequence; furthermore, we have that if $(x_i,x_j) \in E_c(\Gamma_q)$ then ${H^{(q)}}'(x_i)$, ${H^{(q)}}'(x_j)$ belong to a cyclic subgroup for all homomorphisms ${H^{(q)}}'$ induced by a solution ${H}^{(q)}$ from the fundamental sequence;

\item [IH2] for any item $h_i^{(q)}$, if $\tau(h_i^{(q)})=y_{i1}\dots y_{ik}$, then $\BA({H^{(q)}}'(y_{ij})) \supset \BA(H_i^{(q)})$ and $\bigcap \limits_{j=1}^k \BA({H^{(q)}}'(y_{ij}))= \BA(H_i^{(q)})$ for all solutions $H^{(q)}$ of the fundamental sequence and homomorphisms ${H^{(q)}}'$ induced by $H^{(q)}$.
\end{itemize}

The induction hypothesis IH is a natural assumption which is needed to prove that $G$ embeds into $\Ts_0$. The induction hypothesis IH1 and IH2 are essential to keep control on the ``type'' of subgroups whose centralisers we extend. Furthermore, they will be crucial in the next section in order to show that graph towers are discriminated by homomorphisms induced by solutions of the fundamental sequence.

\begin{rem}\label{rem:recall}
We would like to mention the following facts from \cite{CKpc} that will be used in the course of the proof.
\begin{enumerate}
\item In the construction of the process in \cite{CKpc}, we define several different trees. Of special interest to us are the tree $T(\Omega)$ (with infinite branches) and the finite tree $T_{\sol}(\Omega)$.
\item Groups corresponding to the leaves of $T_{\sol}(\Omega)$ are explicitly described (see \cite[Proposition 9.1]{CKpc}).
\item Without loss of generality, one can assume that the word $H(\sigma)$, where $\sigma$ is a closed section of $\Omega$ and $H$ is a solution of the constrained generalised equation $\Omega$ is a subword of a word in the DM-normal form (see \cite[Section 3.3.1]{CKpc}).
\item Automorphisms associated to the vertices of the tree $T_{\sol}(\Omega)$ are completely induced (see \cite[Section 7.1]{CKpc}). Roughly speaking, this means that solutions (in the partially commutative group $\GG=\langle \cA \rangle$) of a constrained generalised equation are induced by a solution (in the free group $F(\cA)$) of the corresponding generalised equation (over the free group).
\end{enumerate}
\end{rem}

Before turning our attention to the construction of the graph tower, we further discuss two key notions for the construction.

We first clarify the structure  of  orthogonal complements $\KK^\perp$ of \cool subgroups and  justify the somewhat non-intuitive definition. The  structure of $\KK^\perp$ allows us, a posteriori, to establish that a discriminating family of the graph tower maps a \cool subgroup $\KK$ into a closed subgroup of $\GG$ and the subgroup $\KK^\perp$ into a directly indecomposable subgroup (which we treat differently depending on if it is cyclic or non-abelian).

Secondly, we introduce the notion of a tribe. This notion plays a crucial role in the study of the dynamics of infinite branches. It is intrinsically related to the fact that one can define some ``finite hierarchy'' (finite partial order) on the type of constraints associated to the variables. We will see that one can interpret this partial order as a measure of complexity on the constraints. For this notion to be well-defined, it is essential that the partially commutative group $\GG$ be finitely generated.

\subsection*{Co-irreducible subgroups}

The definition of \cool subgroup might seem somewhat artificial since we ask that the subgroup $\KK^\perp$ be $E_d(\Gamma)$-directly indecomposable. In the following lemma we clarify the structure of the subgroup $\KK^\perp$ in the group $\HH$.

\begin{lem}\label{lem:typesKperp}
Let $(\Ts,\HH)$ be a graph tower that satisfies the induction hypothesis {\rm IH} and {\rm IH1}. Then, if $\KK^\perp$ is an $E_d(\Gamma)$-directly indecomposable canonical parabolic subgroup of $\HH$, then $\KK^\perp$ is either a directly indecomposable subgroup of $\HH$ or $E_c(\Gamma)$-abelian.
\end{lem}
\begin{proof}
Let $\HH=\GG(\Gamma)$. If follows from  the induction hypothesis IH1 that if $(x,y)\in E_c(\Gamma)$, then $(x,z)\in E_d(\Gamma)$ if and only if $(y,z)\in E_d(\Gamma)$.

Assume that $\KK^\perp$ is directly decomposable, but  $E_d(\Gamma)$-directly indecomposable in $\HH$. Without loss of generality, assume that $\KK^\perp=\KK_1\times \KK_2$. We use induction on the rank of $\KK_2$ to prove that then $\KK^\perp$ is $E_c(\Gamma)$-abelian.

Suppose that $\KK_2$ has rank 1, i.e. $\KK_2=\langle x \rangle$. Let $\az(\KK_1)=\az(\KK_{1,c}) \cup \az(\KK_{1,d})$ where $\az(\KK_{1,c})=\{y \in \az(\KK_1) \mid (x,y)\in E_c(\Gamma)\}$ and $\az(\KK_{1,d})=\{y \in \az(\KK_1) \mid (x,y)\in E_d(\Gamma)\}$. Since $\KK^\perp$ is $E_d(\Gamma)$-directly indecomposable, we have that $\KK_1\ne \KK_{1,d}$. If $\KK_1=\KK_{1,c}$, then since the set $E_c(\Gamma)$  satisfies condition (\ref{eq:ecprop}), it follows that $\KK^\perp$ is $E_c(\Gamma)$-abelian. Otherwise, we have that $(x,k_d)\in E_d(\Gamma)$, for all $k_d\in\KK_{1,d}$ and $(x,k_c)\in E_c(\Gamma)$ for all $k_c\in \KK_{1,c}$. It follows from the above observations that $(k_d,k_c)\in E_d(\Gamma)$ for all $k_d\in\KK_{1,d}$ and for all $k_c\in \KK_{1,c}$ and hence $\KK^\perp=\KK_{1,d} \times \langle \KK_{1,c},x\rangle$ is $E_d(\Gamma)$-directly decomposable - a contradiction.

Assume that $\KK_2$ has rank $r$. Let $x\in \KK_1$ and let $\az(\KK_2)=\az(\KK_{2,c}) \cup \az(\KK_{2,d})$, where $\az(\KK_{2,c})=\{y \in \az(\KK_2) \mid (x,y)\in E_c(\Gamma)\}$ and $\az(\KK_{2,d})=\{y \in \az(\KK_2) \mid (x,y)\in E_d(\Gamma)\}$. Notice that since by assumption $\KK^\perp$ is $E_d(\Gamma)$-directly indecomposable, we have that the set $\az(\KK_2,c)$ is non-empty. As above, it follows that $\KK^\perp=\langle \KK_1\cup \KK_{2,c}\rangle \times \KK_{2,d}$ and the rank of $\KK_{2,d}$ is strictly less than the rank of $\KK_2$. By induction we conclude that $\KK^\perp$ is $E_c(\Gamma)$-abelian.
\end{proof}

\subsection*{Tribes}

Let $\{\BA(S_1), \dots, \BA(S_\m)\}$ be the finite set of distinct canonical parabolic subgroups of $\GG=\langle \cA \rangle$, where $S_i\subset \cA$, $i=1,\dots,\m$.

Let $\Omega=\langle \Upsilon, \Re_\Upsilon\rangle$ be a generalised equation. Since there are only finitely many canonical parabolic subgroups of $\GG$, by Lemma \ref{lem:discfam}, there exists a fundamental sequence so that for every item $h_i$ (base $\mu$, section $\sigma$) of $\Omega$, there exists a canonical parabolic subgroup $\GG_{h_i}$ ($\GG_\mu$ and $\GG_\sigma$, correspondingly) such that $\langle\az(H_i)\rangle=\GG_{h_i}$ ($\langle\az(H(\mu))\rangle=\GG_\mu$ and $\langle\az(H(\sigma))\rangle=\GG_\sigma$, correspondingly) for every solution $H$ of the discriminating family.

A \emph{tribe} of $\Omega$ is a set of items, bases and sections of $\Omega$. Given a fundamental sequence, we say that an item $h_j$ (or a base, or a section) of $\Omega$ \emph{belongs to the tribe} $t_i$ if and only if for the fundamental sequence one has that $\BA(\GG_{h_j})=\BA(S_i)$ (or $\BA(\GG_\mu)=\BA(S_i)$, or $\BA(\GG_\sigma)=\BA(S_i)$, correspondingly), for some $i=1,\dots, \m$. Notice that, in general, two items $h_j$, $h_k$ might belong to the same tribe although $\GG_{h_j} \ne \GG_{h_k}$. Furthermore, if $\Re_\Upsilon(h_j,h_k)$, then $h_j$ and $h_k$ belong to different tribes since for every solution $H$ of the fundamental sequence one has that $H_j \lra H_k$. Abusing the notation, for a tribe $t_i$ we denote by $\BA(t_i)$ the canonical parabolic subgroup $\BA(S_i)$, i.e. the set of letters $a_k\in \cA$ so that for every $h_j\in t_i$ we have $H_j\lra  a_k$. For an item $h_i$ (base $\mu$, section $\sigma$) of $\Omega$ we set $t(h_i)$ ($t(\mu)$ and $\sigma(\mu)$) to be the tribe to which $h_i$ ($\mu$ or $\sigma$) belongs.

We say that a tribe $t_j$ \emph{dominates} the tribe $t_i$ and write $t_j\succ t_i$ if and only if $\BA(t_j) \ge \BA(t_i)$. The relation $\succ$ is a partial order on the set of all tribes. A tribe is \emph{minimal with respect to} $\succ$ if and only if it does not strictly dominate any other tribe.

Let $\Omega=\langle \Upsilon, \Re_\Upsilon \rangle$ be a generalised equation and assume that $\Re_\Upsilon(h_i,h_k)$. Then, for any solution $H$ of $\Omega$, we have that $H_i \lra H_k$. If $h_j$ belongs to a tribe that dominates the tribe of $h_i$, it follows that $H_j \lra H_k$ for all solutions $H$ of $\Omega$. Hence, every solution of $\Omega$ is a solution of $\Omega' =\langle\Upsilon, \Re_\Upsilon\cup \{(h_j,h_k)\}\rangle$. 

\begin{defn}
We say that the set $\Re_\Upsilon$ is \emph{completed} if  for all $h_j$ that belong to a tribe that dominates the tribe of $h_i$, if $\Re_\Upsilon(h_i,h_k)$ then $\Re_\Upsilon(h_j,h_k)$. In particular, if two items $h_i$ and $h_j$ belong to the same tribe, then $\Re_\Upsilon(h_i,h_k)$ if and only if $\Re_\Upsilon(h_j,h_k)$. Moreover, if $h_j$ belongs to a tribe that dominates the tribe of $h_i$, then $\Re_\Upsilon(h_j) \supseteq \Re_\Upsilon(h_i)$.
\end{defn}
\begin{rem} \label{rem:completed}
Without loss of generality, we further assume that for all generalised equations $\Omega=\langle \Upsilon, \Re_\Upsilon\rangle$ the sets $\Re_\Upsilon$ are completed.
\end{rem}

\begin{lem} \label{lem:tribes}
Let $\Omega'=\langle \Upsilon', \Re_{\Upsilon'}\rangle$ be obtained from $\Omega=\langle \Upsilon, \Re_\Upsilon\rangle$ by an elementary transformation $\ET 1-\ET 5$ or a derived transformation $\D 1, \D 2, \D 3, \D 5, \D 6$ and let $\pi$ be the corresponding epimorphism, $\pi:G_{R(\Omega)}\to G_{R(\Omega')}$. Let $w_i(h')=h_{i1}'^{\epsilon_{i1}}\cdots h_{ik_i}'^{\epsilon_{ik_i}}$, $\epsilon_{ij}=\pm 1$, $j=1,\dots,k_i$, $i=1, \dots, \rho_\Omega-1$, be the image of $h_i$ under $\pi$.  Then, for every $j$ the tribe $t(h_{ij}')$ dominates the tribe $t(h_i)$. Furthermore, every minimal tribe of $\Omega'$ dominates a minimal tribe of $\Omega$.
\end{lem}
\begin{proof}
The statement is clear for all the transformations but $\ET 4$ and $\ET 5$. From the description of $\ET 5$, it follows that if $\rho_{\Omega'}=\rho_\Omega$, then $\pi(h_i)=h_{i_j}'$ for all $i$. By definition of a solution of generalised equation, for every solution $H$ of $\Omega$, we have $H_i\doteq H'_{i_j}$ and hence the statement. 

Suppose now that $\rho_{\Omega'}>\rho_{\Omega}$. Then, $\pi(h_i)=h_{i_j}'$ for all $i\ne q$ and $\pi(h_q)=h_{{q'}-1}'h_{q'}'$ (here we use the notation from the description of $\ET 5$, see \cite{CKpc}). Since for every solution $H$ of $\Omega$ we have that $H_q\doteq H'_{q'-1}H'_{q'}$, the first statement follows. Furthermore, if for some $j$ the tribe $t(h_{i_j}')$ (or $t(h_{q'-1}')$ or $t(h_{q'}')$) is minimal in $\Omega'$, then it dominates the tribe $t(h_i)$ (or $t(h_q)$), which is also minimal.

The proof for $\ET 4$ is analogous and left to the reader.
\end{proof}

By construction, every boundary of the generalised equation in $\Omega_{i+1}$ either corresponds to a boundary of $\Omega_i$ or it has been introduced between $\beta(\mu_i)-1$ and $\rho_{\Omega_i}$ in $\Omega_i$. It follows that if the boundaries $j$ and $j+1$ of $\Omega_{i+1}$ correspond to boundaries between $k$ and $k+1$ in $\Omega_{i}$, then the tribe of an item $h_j$ in the generalised equation $\Omega_{i+1}$ dominates the tribe of $h_k$ in $\Omega_i$. Applying this argument recursively, we conclude that if $\Omega_l$, $l>i$ has been obtained from $\Omega_i$ by a sequence of entire transformations and the boundaries $j$ and $j+1$ of $\Omega_{l}$ correspond to boundaries between $k$ and $k+1$ in $\Omega_{i}$, then the tribe of an item $h_j$ in the generalised equation $\Omega_{l}$ dominates the tribe of $h_k$ in $\Omega_i$. 

Let $t(\mu)$ be the tribe of a base $\mu$ in $\Omega$. Then, the tribe of every item covered by $\mu$ dominates $t(\mu)$. We note that the tribe of $\mu$ in $\Omega_i$ is dominated by the tribe of $\mu$ in $\Omega_{i+1}$ (unless $\mu$ is completely removed when passing from $\Omega_i$ to $\Omega_{i+1}$).

\subsection*{Base of induction} If the height $q$ of the fundamental branch equals $0$, then the group $G_{\Omega_0}$ is the group associated to the leaf of the tree $T_{\sol}(G)$. By the description of the leaves of the tree $T_{\sol}$ (see \cite[Proposition 9.1]{CKpc} and its proof), it follows that $G_{\Omega_0}$  is a partially commutative group, which is fully residually $\GG$ and it is built from $\GG$ by a sequence of extensions of centralisers of directly indecomposable canonical parabolic subgroups. Hence, in our terms, this partially commutative group is a graph tower over $\GG$, where each floor is basic, and at each floor $\Ts_k=\HH_k$ (and $\tau_k$ is the identity).

If the variables $x_1,\dots,x_{l_k}$ extend a centraliser of a cyclic canonical parabolic subgroup $\langle a_i \rangle$ in $\GG(\Gamma_{k-1})$, then the subgroup $\langle a_i, x_1,\dots,x_{l_k}\rangle$ in $\GG(\Gamma_k)$ is free abelian and, by definition, the edges $(a_i,x_i), (x_i,x_j)$ belong to $E_c(\Gamma_k)$. All the other (new) edges belong to $E_d(\Gamma_k)$.

From the construction of the leaves of the tree $T_{\sol}(G)$, see \cite[Proposition 9.1]{CKpc}, we obtain that the graph tower satisfies the induction hypothesis:
\begin{itemize}
\item  the graph tower $(\Ts=\HH_q, \HH_q)$ trivially makes Diagram (\ref{diag:period}) commutative;
\item  by definition of a solution of a generalised equation and since $H'=H$, for all $x_i,x_j \in \HH_q$, we have that $(x_i^{(1)},x_j^{(1)}) \in E_d(\Gamma_q)$ if and only if $H'(x_i) \lra H'(x_j)$ for every homomorphism $H'$ induced by a solution $H$ from the fundamental sequence; furthermore, if $(x_i,x_j) \in E_c(\Gamma_q)$, then $H'(x_i)$, $H'(x_j)$ belong to the same cyclic subgroup for every homomorphism $H'$ induced by a solution $H$ from the fundamental sequence;
\item since $\tau_q:G_{\Omega_q} \to \HH_q$ is the identity, it obviously satisfies the hypothesis IH2: if $\tau_q(h_i)=y_{i1}\dots y_{ik}$, then $\BA(H'(y_{ij})) \supset \BA(H_i)$ and $\bigcap \limits_{j=1}^k \BA(H'(y_{ij}))= \BA(H_i)$ for all solutions $H$ of the fundamental sequence and homomorphisms $H'$ induced by $H$.
\end{itemize}

\subsection*{Step of Induction}

Suppose that the graph tower $\Ts_1$ exists and the induction hypothesis IH, IH1 and IH2 are satisfied for all groups $G_{\Omega_1}$ so that the fundamental sequence for $G_{R(\Omega_1)}$ has length less than or equal to $q-1$. We show how to construct the graph tower $\Ts_0$ for groups $G_{\Omega_0}$ whose fundamental sequence has length $q$ and prove that the graph tower satisfies all the induction hypothesis.

We construct the group $\Ts_0$ starting from the graph tower $\Ts_1$. The construction of $\Ts_0$ depends on the type of the branch through which the fundamental sequence factors through. It follows by construction of the tree $T(\Omega)$, see \cite{CKpc}, that solutions of the fundamental sequence 
\begin{itemize}
\item either factor through a finite branch of the tree $T(\Omega_0)$; in this case, every solution of the fundamental sequence factors through the group $G_{R(\Omega_v)}$ associated  to the leaf $v$ of $T(\Omega_0)$ (and so the automorphism group needed for the description of  solutions of the fundamental sequence is trivial), where the length of the fundamental branch for $\Omega_v$ equals $q-1$.
\item or they factor through an infinite branch of the tree $T(\Omega_0)$. By \cite[Lemma 4.19]{CKpc}, there are three different types of infinite branches:
\begin{enumerate}
\item infinite branch of type 7-10; in this case, every solution of the fundamental sequence is the composition of an automorphism from the automorphism group associated to the vertex of type $7-10$ (i.e. automorphisms invariant with respect to the kernel), the epimorphism $\pi(v_0,v_1)$ and a homomorphism from $G_{R(\Omega_1)}$ to $\GG$.
\item infinite branch of type 12; in this case, every solution of the fundamental sequence is the composition of an automorphism from the automorphism group associated to the vertex of type $12$ (i.e. automorphisms invariant with respect to the non-quadratic part), the epimorphism $\pi(v_0,v_1)$ and a homomorphism from $G_{R(\Omega_1)}$ to $\GG$.
\item infinite branch of type 15; in this case, we have again two alternatives: 
\begin{itemize}
\item either the fundamental sequence goes through a vertex of type $2$, i.e. the fundamental sequence factors through a generalised equation singular (or strongly singular) with respect to a periodic structure and so every solution is a composition of an automorphism associated to a vertex of type $2$, the epimorphism $\pi(v_0,v_1)$ and a homomorphism from $G_{R(\Omega_1)}$ to $\GG$; 
\item or else, every solution of the fundamental sequence is the composition of an automorphism from the group of automorphisms generated by automorphisms invariant with respect to the non-quadratic part and automorphisms associated to regular periodic structures, the epimorphism $\pi(v_0,v_1)$ and a homomorphism  from $G_{R(\Omega_1)}$ to $\GG$.
\end{itemize}
\end{enumerate}

\end{itemize}

The rest of this section essentially relies on the analogue of the Makanin-Razborov process described in \cite{CKpc}. Although our intention is to make the exposition as self-contained as possible, the reader is referred to \cite{CKpc} for notions and notation not explicitly defined here.

\subsection{Trivial group of automorphisms}

If the group of automorphisms associated to a vertex of the tree $T_{\sol}$ is trivial, we set $\Ts_0=\Ts_1$, $\HH_0=\HH_1$ and $\tau_0$ to be $\pi(v_0,v_1) \tau_1$, where $\tau_1$ is the homomorphism from $G_{\Omega_1}$ to $\Ts_1$ that makes Diagram (\ref{diag:period}) commutative. If the group of automorphisms is trivial, any homomorphism of the fundamental sequence factors through $G_{R(\Omega_1)}$ and so $\tau_0$ and $\Ts_0$ make Diagram (\ref{diag:period}) commutative.

\subsection{Quadratic Case}

In this section we study the quadratic part of a generalised equation in three steps. Firstly, we use the dynamics of an infinite branch of type 12 (surface type) in order to  understand the set of constraints  $\Re_\Upsilon $ of a generalised equation $\Omega=\langle \Upsilon, \Re_\Upsilon\rangle$ that repeats infinitely many times in the infinite branch. More precisely, we show that all items of the quadratic section that are not fixed by the group of automorphisms assigned to the vertex belong to the same minimal tribe and so, in particular, all of them commute with the same set of items. Furthermore, items of the quadratic section that are fixed by the automorphisms belong to tribes that dominate the minimal tribe (the one defined by any of the items which is not fixed by the automorphisms). 

Once we have established these relations between the tribes of items of the quadratic part, we use them to prove that there exists an automorphism of the group $G_\Omega$ that fixes the non-quadratic part and presents the quadratic part as a quotient of a partially commutative group by a surface relation.

Finally, we construct the graph tower and show that it satisfies the corresponding induction hypothesis.

\subsubsection*{Analysing Constraints}
Let us begin with the analysis of an infinite branch of type $12$. Let $\Omega=\langle\Upsilon, \Re_\Upsilon\rangle$ be a quadratic constrained generalised equation of type $12$. Throughout this section we will be mostly concerned with the active bases, boundaries, items etc., i.e. bases, items and boundaries of the (active) quadratic part. When no confusion arises, we will omit saying that the bases (items, boundaries etc.) are active.

Let us consider an infinite branch of type $12$ in the tree $T(\Omega)$, 
$$
\Omega_0 \to \Omega_1\to \dots
$$

Then, by \cite[Lemma 7.12]{CKpc}, there exists a generalised equation $\Omega_{r_0}$ that repeats infinitely many times in the infinite branch: $\Omega_{r_0}=\Omega_{r_1}=\dots$ Without loss of generality, we assume that $r_0=0$. Furthermore, the canonical epimorphism $\pi(r_i,r_j)$ from $G_{\Omega_{r_i}}$ to $G_{\Omega_{r_j}}$ is an automorphism of $G_{\Omega_{r_i}}$. In fact, since there are finitely many different generalised equations appearing in the infinite branch,  without loss of generality, we may assume that every generalised equation of the infinite branch appears in it infinitely many times.

Recall that, since the branch is infinite and $\Omega_0$ repeats infinitely many times, we have that the complexity, the number of active bases and the number of items is constant throughout the branch, i.e. $\comp(\Omega_i)=\comp(\Omega_{i+1})$, $n_A(\Omega_i)=n_A(\Omega_{i+1})$ and $\rho_{\Omega_i}=\rho_{\Omega_{i+1}}$ for all $i$. Furthermore, if $\mu_i$ is the carrier of $\Omega_i$, then every $\mu_i$-tying introduces a new boundary. There is a natural one-to-one correspondence between the bases of $\Omega_i$ and $\Omega_{i+1}$. Namely, the base $\mu_i$ in $\Omega_{i+1}$ is the only leading base which is not the carrier and all the other bases are naturally preserved. Hence, we may assume that the sets of bases of $\Omega_i$ and $\Omega_{i+1}$ coincide.

\begin{lem} \label{lem:1clsec}
Let $\Omega$ be a generalised equation of type $12$ that repeats infinitely many times in the infinite branch. Then, the active part of the generalised equation $\Omega$ has only one closed section.
\end{lem}
\begin{proof}
Suppose, on the contrary, that $\Omega$ contains more than one (active) closed section: $[1,i]$ and $[i,j]$.

For every $n$ we partition the set of bases $\BS(\Omega_n)$ of $\Omega_n$ into two. Let $L(\Omega_n)$ be the set of bases $\lambda$ of $\Omega_n$ so that $\alpha(\lambda)$ and $\beta(\lambda)$ lie to the left of the boundary $i$ of $\Omega$. Set $R(\Omega_n)$ to be the set of bases $\lambda$ of $\Omega_n$ so that $\alpha(\lambda)$ and $\beta(\lambda)$ lie to the right of the boundary $i$ of $\Omega$. Since the boundary $i$ is closed in $\Omega$, it follows that $\BS(\Omega_n)=L\sqcup R$.

Notice that if $\lambda\in R(\Omega_n)$, then $\lambda\in R(\Omega_{n+1})$.

We first observe that the carrier base  of $\Omega_n$ belongs to $L(\Omega_n)$ for all $n$.  Indeed, if $\Omega_{k+1}$ is the generalised equation so that its carrier $\mu_{k+1}$ belongs to $R(\Omega_{k+1})$ and for all $n\le k$ the carrier $\mu_n$ of $\Omega_n$ belongs to $L(\Omega_n)$, then it follows that $n_A(\Omega_k)>n_A(\Omega_{k+1})$ - a contradiction. 

We now show that if $\nu\in R(\Omega)$ and $\Delta(\nu)\in L(\Omega)$, then $\Delta(\nu)$ is never carrier.  Assume the contrary, i.e. $\Delta(\nu)=\mu_l$ is the carrier of $\Omega_l$ for some $l$. Let $\lambda$ be a transfer base. In the infinite branch, the base $\lambda$ is transferred infinitely many times. Let $\Omega_m$ be the next time $\lambda$ is transferred. It is clear that the carrier $\mu_m$ belongs to $R(\Omega_m)$, deriving a contradiction with the above observation.  

Since, the generalised equation $\Omega$ repeats, it follows that no boundaries are introduced between $\alpha(\Delta(\nu))$ and $\beta(\Delta(\nu))$. This derives a contradiction with the assumption that the section $[i,j]$ of $\Omega$ is active.
\end{proof}

Let $\lambda$ be a base of $\Omega$. We call the number $l(\lambda)=\beta(\lambda)-\alpha(\lambda)$ the \emph{length} of $\lambda$ (in $\Omega$). We call the base \emph{short}  if $l(\lambda)=1$ and \emph{long} otherwise. 

Let $\lambda_1$ and $\lambda_2$ be two long bases of $\Omega$. Notice that since there are no free boundaries in $\Omega$ and since, by Lemma \ref{lem:1clsec}, the generalised equation $\Omega$ has only one active closed section, then every boundary touches precisely two bases. Therefore, the bases $\lambda_1$ and $\lambda_2$ either do not overlap or if they do, then $\alpha(\lambda_2)=\beta(\lambda_1)- 1$ or $\alpha(\lambda_1)=\beta(\lambda_2)- 1$.

For every generalised equation $\Omega_r$ from the infinite branch and any $h_i^{(r)} \in \Omega_r$, the set of words $\{ \pi(v_r,v_s)(h_i^{(r)}) \mid  s>r, \Omega_r=\Omega_s \}$ is either finite or infinite, i.e. the orbit of an item under the automorphisms associated to $\Omega_r$ is either a finite or an infinite set. Define the set $\mathcal F(\Omega_r)$ to be the set of items of $\Omega_r$ for which this set is finite and  let $\Hc(\Omega_r)$ be its complement.

Notice that by the description of automorphisms associated to a generalised equation of type $12$, see \cite[Lemma 7.7]{CKpc}, any item that belongs to a quadratic-coefficient base of $\Omega_r$ or does not belong to the quadratic part is fixed by the automorphisms, and so belongs to the set $\mathcal F(\Omega_r)$.

Let $h_i^{(0)}\in \mathcal F(\Omega_0)$ and let 
$$
\{\varphi_s(h_i^{(0)})\mid \varphi_s=\pi(v_0,v_s)(h_i^{(0)}),\Omega_0=\Omega_s\}=\{w_{i1},\dots,w_{ik}\}
$$
be the orbit of $h_i^{(0)}$ by the family of automorphisms $\{\varphi_s\}$ associated to the vertex $v_0$.
Observe that the set of solutions
$$
F=\{\varphi_s \pi H^{(1)}\}
$$
that factors through the fundamental branch is a discriminating family, where $\varphi_s$ is an automorphism associated to the vertex corresponding to $\Omega_0$, $\Omega_1$ belongs to $T_{\sol}(\Omega_{0})$, the epimorphism $\pi: G_{\Omega_{0}} \to G_{\Omega_1}$ is proper, and $H^{(1)}$  is a solution of $\Omega_{1}$. By Lemma \ref{lem:discfam}, it follows that one of the finitely many subfamilies 
$$
F_j=\{\varphi_{s_j} \pi H^{(1)}\mid H^{(1)} \hbox{ is a solution of } \Omega_{1}\},
$$
where $\varphi_{s_j}(h_i^{(0)})=w_{ij}$, $j=1,\dots,k$, is discriminating. Hence, without loss of generality, we may assume that if $h_i^{(0)}\in \mathcal F(\Omega_0)$, then $\varphi_s(h_i^{(0)})=w$.

We therefore arrive at the following
\begin{rem}\label{rem:discrfam}
If $h_i^{(0)}\in \mathcal F(\Omega_0)$, then $\varphi\pi(v_0,v_1)(h_i^{(0)})=w(h^{(1)})$, where $w(h^{(1)})$ is a word  in items from $\Omega_1$, such that for a discriminating family of solutions $H^{(0)}=\varphi\pi(v_0,v_1)H^{(1)}$ of the fundamental sequence, we have that $H_i^{(0)}=w(H^{(1)})$.
We will show that, in fact, one can think of these items as quadratic-coefficient.
\end{rem}

Let $\Omega_s=\Omega_0$ be so that for all $h_i^{(0)}\in \mathcal F(\Omega)$ we have that $\pi(v_0,v_1)(h_i^{(0)})=w(h^{(s)})=h_{i_1}^{(1)}\cdots h_{i_n}^{(1)}$. For all $s'>s$ such that $\Omega_{s'}=\Omega_0$, we have that 
$$
\pi(v_0,v_{s'})(h_i^{(0)})=\pi(v_s,v_{s'})(\pi(v_0,v_s)(h_i^{(0)}))=\pi(v_s,v_{s'})(h_{i_1}^{(1)}\cdots h_{i_n}^{(1)})=h_{i_1}^{(1)}\cdots h_{i_n}^{(1)}. 
$$
From the definition of $\pi(v_s,v_{s'})$, we conclude that $\pi(v_s,v_{s'})(h_{i_j}^{(s)})=h_{i_j}^{(s')}$ and so $h_{i_j}^{(s)}\in \mathcal F(\Omega_s)$ and $\pi(v_s,v_t)(h_{i_j}^{(s)})=h_m^{(t)}$, for all $s<t<s'$, $\Omega_s=\Omega_{s'}$.

\begin{rem} \label{rem:K}
Without loss of generality, we may assume that for a fundamental sequence, we have that $h_i^{(0)}\in \mathcal F(\Omega_0)$ if and only if $\varphi_s(h_i^{(0)})=h_i^{(0)}$ for all $s$ so that $\Omega_s=\Omega_0$. Therefore, for the discriminating family, one has that $H_i^{(0)}=H^{(1)}(\pi(h_i))$, where $\pi: G_{R(\Omega_{0})} \to G_{R(\Omega_1)}$ is the proper quotient in the tree $T_{\sol}(\Omega_{0})$ and $H^{(1)}$ is a solution of $\Omega_{1}$.

Furthermore, we have that if $h_i^{(0)}\in \mathcal F(\Omega_0)$ then $\pi(v_0,v_t)(h_i^{(0)})=h_m^{(t)} \in \mathcal F(\Omega_t)$ for all $t>0$. Hence $|\mathcal F(\Omega_0)| \le \dots \le |\mathcal F(\Omega_t)| \le \dots$ Then, since the number $|\mathcal F(\Omega_0)|$ is bounded above by $\rho_{\Omega_0}=\rho_{\Omega_t}$, without loss of generality, we may assume that the number $|\mathcal F(\Omega_0)|$ is constant throughout the branch, i.e. $|\mathcal F(\Omega_0)|=|\mathcal F(\Omega_t)|$ for all $t$.
\end{rem}

Given a generalised equation $\Omega$, we define $\mm(\Omega)$ to be the set of minimal tribes of the generalised equation:
$$
\mm(\Omega) = \{ t(h_j) \mid h_j \in \Hc(\Omega) \hbox{ and } t(h_j) \hbox{ does not strictly dominate } t(h_k), \hbox{ for all }  h_k \in \Hc(\Omega) \},
$$ 
and $h(\mmm_i(\Omega))$ to be the set of items that belong to a given minimal tribe $\mmm_i$ of $\Omega$:
$$
h(\mmm_i(\Omega))=\{ h_j \in \Hc(\Omega)\} \mid t(h_j)\in \mmm_i, \mmm_i \in \mm(\Omega) \}.
$$ 
We draw the reader's attention to the fact that the minimal tribes are defined using the items from $\Hc(\Omega)$, which, by definition, are active.

\begin{lem} \label{lem:mintribe}
In an infinite branch of type $12$, there exists $i\in \BN$ so that $\mm(\Omega_{i})=\mm(\Omega_j)$ for all $j>i$. Furthermore, the cardinality of the set $h(\mmm_k(\Omega_{i}))$ is constant, i.e. $|h(\mmm_k(\Omega_{i}))|=|h(\mmm_k(\Omega_j))|$ for every minimal tribe $\mmm_k\in \mm(\Omega_{i})=\mm(\Omega_j)$.
\end{lem}
\begin{proof}
Since the generalised equation repeats infinitely many times, $\Omega_{r_0}=\Omega_{r_1}=\dots$ and since there are finitely many different tribes, for sufficiently large $i$  and passing to a subsequence $\Omega_{r_{i_0}}=\Omega_{r_{i_1}}=\dots$, we may assume that the tribe of every item $h_k^{(r_{i_j})}$ is the same, i.e. $t(h_k^{(r_{i_0})})=t(h_k^{(r_{i_j})})$ for all $j\in \BN$. Hence $\mm(\Omega_{r_{i_0}})=\mm(\Omega_{r_{i_j}})$ and $|h(\mmm_k(\Omega_{r_{i_0}}))|=|h(\mmm_k(\Omega_{r_{i_j}}))|$ for all $\mmm_k\in \mm(\Omega_{r_{i_0}})= \mm(\Omega_{r_{i_j}})$. To simplify the notation, set $i_0=0$. 

It follows from the above equalities that in order to prove that $\mm(\Omega_{i})=\mm(\Omega_j)$ and that $|h(\mmm_k(\Omega_i))|=|h(\mmm_k(\Omega_j))|$ for all $i,j\in \BN$ and $\mmm_k\in \mm(\Omega_{i})$, it suffices to show that $\mm(\Omega_i)\supseteq \mm(\Omega_{i+1})$ and that $|h(\mmm_k(\Omega_i))|\ge |h(\mmm_k(\Omega_{i+1}))|$.

Let $\mu$ be the carrier of $\Omega_i$, let $j,j+1$ be the boundaries that touch a transfer base (that is $j+1\le \beta(\mu)-1$) and let $(j, \mu, \b(j))$ be the corresponding boundary connection introduced in the entire transformation. Suppose that $\b(j)$ is introduced between $\cali(\b(j))$ and $\cali(\b(j))+1$. By Lemma \ref{lem:tribes}, it follows that $t(h_{\cali(\b(j))}^{(i+1)})$ dominates $t(h_j^{(i)})$ and  $t(h_{\b(j)}^{(i+1)})$ dominates $t(h_{\cali(\b(j))}^{(i)})$. Therefore, we have that $\mm(\Omega_i)\supseteq \mm(\Omega_{i+1})$ and $|h(\mmm_k(\Omega_i))|\ge |h(\mmm_k(\Omega_{i+1}))|$ for all minimal tribes $\mmm_k$.
\end{proof}

We call the items of $\Omega_i$ that belong to a minimal tribe \emph{minimal}.

Recall that the carrier base $\mu_i$ of $\Omega_i$ is a long base. Since there are no free boundaries, there are exactly $\beta(\mu_i)-1$ transfer bases and all of them are short. Moreover, $\mu_i$ is a short transfer base of $\Omega_{i+1}$, see Figure \ref{pic:ge}.

\begin{defn} \label{defn:88} 
Let $\mu$ be the carrier base of $\Omega$. We take the (uniquely defined) transfer base $\lambda$ so that $\alpha(\lambda)=1$ and we $\mu$-tie the boundary $2=\beta(\lambda)$. We then transfer $\lambda$ from $\mu$ onto $\Delta (\mu)$. 

Now, we cut $\mu$ in the boundary $2$ and delete the section $[1,2]$.

It follows from the definition that the (usual) entire transformation $\D 5$ is a composition of the transformations just introduced (namely of $\beta(\mu)-2$ such transformations). From now on, we abuse the terminology and refer to the introduced transformation as to the \emph{entire transformation} and refer to the usual entire transformation as to $\D 5$ or \emph{complete entire transformation}. 
\end{defn}

\begin{lem} \label{lem:AB}
In an infinite branch $\Omega_0\to\Omega_1\to\dots$ of type $12$, there exists a generalised equation $\Omega_{p}$ such that 
\begin{description}
\item[Case A] either the item $h_1^{(p)}$ belongs to a minimal tribe of $\Omega_p$,
\item[Case B] or the entire transformation introduces a new boundary in an item that belongs to a minimal tribe of $\Omega_p$.
\end{description}
\end{lem}
\begin{proof}
Let $h_j^{(r_0)}$ be a minimal item of $\Omega_{r_0}$. Since $h_j^{(r_0)} \in \Hc(\Omega_{r_0})$, by definition of the set $\Hc(\Omega_{r_0})$, the orbit of $h_j^{(r_0)}$ under the automorphisms associated to $\Omega_{r_0}$ is infinite. Therefore, there exists $p\in \BN$ so that 
$$
|\pi(r_0,i)(h_j^{(r_0)})|=1, \hbox{ for all } i\le p \hbox{ and } |\pi(r_0,p+1)(h_j^{(r_0)})|\ge 2.
$$
Hence,
\begin{itemize}
\item either a base covering the item $\pi(r_0,p)(h_j^{(r_0)})=h_{j_p}^{(p)}$ is transferred (Case A of the lemma) and so, by Remark \ref{rem:K}, one can assume that $j_p=1$,
\item or, in the entire transformation of the generalised equation $\Omega_{p}$, a new boundary  is introduced between the boundaries ${j_p}$ and $j_p+1$ (Case B of the lemma).
\end{itemize}
Furthermore, since $\pi(r_0,i)(h_j^{(r_0)})=h_{j_i}^{(i)}$ for all $i \le p$, it follows that the tribe of $h_{j_i}^{(i)}$ is the same for all $i\le p$, i.e. $t(h_{j_i}^{(i)})=t(h_{j}^{(r_0)})$. Hence if $h_j^{(r_0)}$ belongs to a minimal tribe $\mmm_k$ in $\Omega_0$, then $h_{j_i}^{(i)}$ also belongs to the tribe $\mmm_k$ and, by Lemma \ref{lem:mintribe}, $\mmm_k$ is a minimal tribe in $\Omega_{p}$.
\end{proof}

\begin{lem}
In an infinite branch of type $12$,  
$$
\Omega_0\to \Omega_1\to \dots,
$$
where $\Omega_i$ is obtained from $\Omega_{i+1}$ by an entire transformation, there exists a generalised equation $\Omega_{k}$ of the infinite branch such that the item $h_1^{(k)}$ is minimal {\rm(Case A} of {\rm Lemma \ref{lem:AB}} holds{\rm)}.
\end{lem}
\begin{proof}
By Lemma \ref{lem:AB}, it suffices to show that if in the infinite branch there exists a generalised equation that satisfies Case B of Lemma \ref{lem:AB}, then there exists a generalised equation that satisfies Case A.

Without loss of generality, we may assume that the generalised equation $\Omega_{0}$ satisfies Case B of Lemma \ref{lem:AB}, i.e.  the entire transformation introduces a boundary connection $(2,\mu, \b(2))$ and the boundary $\b(2)$ is between the boundaries $\cali(\b(2))$ and $\cali(\b(2))+1$, where $\mu$ is the carrier base of $\Omega_0$ and $h_{\cali(\b(2))}^{(0)}$ is a minimal item that belongs to some minimal tribe $\mmm$.

Since, by Lemma \ref{lem:mintribe}, the number of minimal items of a fixed minimal tribe $\mmm$ in the infinite branch is constant, it follows that either $h_{\cali(\b(2))}^{(1)}$ or $h_{\b(2)}^{(1)}$ is minimal and belongs to $\mmm$. If  $h_{\b(2)}^{(1)} \in h(\mmm)$, then, by Lemma \ref{lem:tribes}, $h_1^{(0)}\in h(\mmm)$ and hence $\Omega_0$ satisfies Case A of Lemma \ref{lem:AB}. 

Assume now that $h_{\cali(\b(2))}^{(1)}$ is minimal and belongs to $\mmm$. Let $k$ be so that for all the boundary connections $(i, \mu,\b(i))$, $i=2,\dots, k-1$, one has $\cali(\b(i))=\cali(\b(2))$ and 
$\cali(\b(k))>\cali(\b(2))$. 

Suppose that $\Omega_i$, $i=1, \dots, k-1$, does not satisfy  Case A of Lemma \ref{lem:AB}. Then, since the number of minimal items of the minimal tribe $\mmm$ is constant in the infinite branch, it follows that $h_{\b(k-1)}^{(k)}$ is minimal and belongs to $h(\mmm)$ (note that here $k-1$ is the boundary of $\Omega_0$). 
By definition of $k$ and Lemma \ref{lem:tribes}, we have that the tribe of $h_{\b(k-1)}^{(k)}$ dominates the tribe of $h_{k-1}^{(0)}$ and hence, $h_{k-1}^{(0)}$ is minimal and belongs to $h(\mmm)$. We conclude that the generalised equation $\Omega_{k}$ falls under Case A of Lemma \ref{lem:AB}.
\end{proof}

\begin{lem}\label{lem:mincar}
Let $\Omega_0$ be a generalised equation of type $12$ that repeats infinitely many times in the infinite branch. Let  $\mu$ be the carrier of $\Omega_0$ and let $h_1^{(0)}$ be an item from a minimal tribe $\mmm$.  Then, the item $h_{\beta(\mu)-1}^{(0)}$ is also minimal and belongs to $h(\mmm)$. In particular, the item $h_1^{(1)}$ of $\Omega_1$, where $\Omega_1$ is obtained from $\Omega_0$ by a complete entire transformation is minimal and belongs to $h(\mmm)$.

The tribe of every item covered by $\mu$ dominates the minimal tribe $\mmm$ of $h_1^{(0)}$. Furthermore, if an item $h_i^{(0)}$ covered by $\mu$ strictly dominates the minimal tribe, then $\pi(v_0,v_1)(h_i^{(0)})=h_j^{(1)}$ and, in particular, $h_i^{(0)} \in \mathcal{F} (\Omega_0)$.
\end{lem}
\begin{proof}
Recall that every $\mu$-tying introduces a new boundary. We only consider the case when $\varepsilon(\Delta(\mu))=\varepsilon(\mu)=1$. The remaining cases are analogous. We $\mu$-tie all the boundaries $2, \dots, \beta(\mu)-1$ thus introducing boundaries $n_2,\dots, n_{\beta(\mu)-1}$ between the boundaries $\cali(n_2)$ and $\cali(n_2)+1$, \dots, $\cali(n_{\beta(\mu)-1})$ and $\cali(n_{\beta(\mu)-1})+1$. We will denote the corresponding items of $\Omega_1$ by $h_{\cali(j)}^{(1)}$, $h_{\cali(j)+1}^{(1)}$ and $h_{n_j}^{(1)}$. 

Since the number of minimal items of a fixed minimal tribe in the infinite branch is constant, it follows that at least one of the items $h_{\cali(2)}^{(1)}$ and $h_{n_2}^{(1)}$ belongs to $h(\mmm)$. Hence, the item $h_{\cali(2)}^{(0)}$ of $\Omega_0$ is minimal and belongs to $h(\mmm)$. Therefore, since the number $|h(\mmm)|$ is constant both $h_{\cali(2)}^{(1)}$ and $h_{n_2}^{(1)}$ are minimal and belong to $h(\mmm)$. 

If $\cali(3)>\cali(2)$, we conclude that the item $h_2^{(0)}$ of $\Omega_0$ is minimal and belongs to $h(\mmm)$. Either $\cali(2)=\cali(3)$ or, otherwise, we can repeat the above argument for $h_2$. Let $m$ be maximal so that $\cali(2)=\cali(m)$. Then, the items $h_{n_i}^{(1)}$ of $\Omega_1$ belong to the same tribe as the item $h_i^{(0)}$ of $\Omega_0$, $i=2,\dots, m-1$, and since the number $|h(\mmm)|$ in the infinite branch is constant, then $h_{n_m}^{(1)}$ is minimal and belongs to $h(\mmm)$. It follows that the item $h_m^{(0)}$ is minimal and belongs to $h(\mmm)$, hence we can repeat the argument for $h_m^{(0)}$.

We conclude that the item $h_{\beta(\mu)-1}^{(0)}$ of $\Omega_0$ is minimal and belongs to $h(\mmm)$ as so does the item $h_{n_{\beta(\mu)-1}}^{(1)}$ of $\Omega_1$.

Notice that we have shown that for any item $h_i^{(0)}$ covered by $\mu$, if $\cali(i) < \cali(i+1)$, then the tribe of $h_i^{(0)}$ is minimal and belongs to $h(\mmm)$. Furthermore, all the boundaries are introduced in items whose tribe is $\mmm$. Hence for any $h_i^{(0)}$ covered by $\mu$ such that $\cali(i) = \cali(i+1)$, we have that its tribe dominates the minimal tribe $\mmm$ of $h_1^{(0)}$.
\end{proof}

\begin{rem} \label{rem:manymin}
From the proof of Lemma \ref{lem:mincar}, it follows that, under the assumptions of the lemma, the image of the minimal item $h_{\cali(2)}^{(0)}$ of $\Omega_0$ under $\pi(v_0, v_1)$ is $h_{\cali(2)}^{(1)}h_{n_2}^{(1)}$ and both of the items $h_{\cali(2)}^{(1)}$, $h_{n_2}^{(1)}$ are minimal.
\end{rem}

\begin{lem} \label{lem:alltransfer}
Let $\Omega_0$ be a generalised equation of type $12$ that repeats infinitely many times in the infinite branch. Then, in the infinite branch, all bases of $\Omega_0$, but, perhaps, bases $\nu$ contained in the long base $\lambda$ so that $\beta(\lambda)=\rho_A$, are transferred infinitely many times.
\end{lem}
\begin{proof}
Notice that, without loss of generality, we can assume that if a base is transferred, then it is transferred infinitely many times in an infinite branch. 

Consider the long base $\lambda$ of $\Omega_0$ so that $\beta(\lambda)=\rho_A$. Let $C(\lambda)$ be the set of bases contained in $\lambda$.

If either $\lambda$ or any of the bases in $C(\lambda)$ is the carrier base or is transferred in some generalised equation of the infinite branch, say $\Omega_k$, then  $h_1^{(k)}$ is an item with boundaries inside one of the items $h_{\alpha(\lambda)}^{(0)},\dots,h_{\beta(\lambda)}^{(0)}$. Hence, the section $[1,\alpha(\lambda)]$ of $\Omega_0$ (and so all the bases in this section) has been transferred and the statement follows. Notice that since the base $\lambda$ is long, for a base contained in $\lambda$ to be a carrier base of a generalised equation $\Omega_k$, it is necessary that $\Delta(\lambda)$ or $\lambda$ be the carrier base of a generalised equation $\Omega_l$, where $l<k$.
 
Assume further that $\lambda$ and the bases from $C(\lambda)$ are neither carrier bases nor transfer bases.

Without loss of generality, we may assume that there exists a base $\eta \in C(\lambda)\cup \{\lambda\}$ so that  $\Delta(\eta)$ is the carrier of some $\Omega_k$. Indeed, suppose that $\Delta(\eta)$ is neither carrier nor transfer base for all $\eta \in C(\lambda)\cup \{\lambda\}$. In this case, the section $[\alpha(\lambda), \beta(\lambda)]$ can be declared non-active and we can set $\rho_A$ to be $\alpha(\lambda)$. Similarly, if some base $\Delta(\eta)$ is transferred but for all generalised equations we have $\beta(\Delta(\eta))-\alpha(\Delta(\eta))=1$, then the section $[\alpha(\lambda), \beta(\lambda)]$ can be declared non-active and we can set $\rho_A$ to be $\alpha(\lambda)$. Finally, assume that a base $\Delta(\eta)$ is transferred and for some generalised equation we have that $\beta(\Delta(\eta))-\alpha(\Delta(\eta))>1$. Since $\Delta(\eta)$ is transferred infinitely many times and transfer bases are always short, it follows that before being transferred, the base $\Delta(\eta)$ must be a carrier base (since this is the only way a long base can become short). Therefore, there exists a base $\eta \in C(\lambda)\cup \{\lambda\}$ so that  $\Delta(\eta)$ is the carrier of some $\Omega_k$.

Note that since the branch is infinite, every $\mu$-tying introduces a new boundary. Since there exists a base $\eta\in C(\lambda)$ so that $\Delta(\eta)$ is the carrier of some generalised equation, say $\Omega_k$, at this step the new boundaries are introduced in the section $[\alpha(\lambda), \beta(\lambda)]$. Hence, for the generalised equation $\Omega_0$ to repeat, there must exist $l$ so that $\Delta(\lambda)$ is the carrier of $\Omega_l$.

Notice that, in particular, we have shown that in the assumptions of the lemma, either $\lambda$ or its dual $\Delta(\lambda)$ is a carrier base.

Suppose that $\varepsilon(\lambda)\cdot \varepsilon(\Delta(\lambda))=-1$. Then, in the generalised equation $\Omega_{l+1}$ we have $\beta(\lambda)< \rho_A$. Hence,  for the generalised equation $\Omega_0$ to repeat, the base $\lambda$ has to be transferred. In this case, we conclude that all the bases of $\Omega_0$, but, perhaps, bases $\nu$ contained in a base $\lambda$ so that $\beta(\lambda)=\rho_A$, are transferred infinitely many times.

Suppose now that $\varepsilon(\lambda)\cdot \varepsilon(\Delta(\lambda))=1$. Let $\nu$ be the (uniquely defined) transfer base of $\Omega_l$ so that $\alpha(\nu)=1$. Since the base $\nu$ is a transfer base of $\Omega_l$, it is transferred infinitely many times in an infinite branch and hence for some generalised equation, say $\Omega_k$, the item $h_1^{(k)}$ is an item with boundaries inside one of the items $h_{\alpha(\lambda)}^{(0)},\dots,h_{\beta(\lambda)}^{(0)}$. Hence, the section $[1,\alpha(\lambda)]$ of $\Omega_0$ (and so all the bases in this section) has been transferred and the statement follows.
\end{proof}

\begin{lem} \label{lem:1stlast}
Let $\Omega$ be a generalised equation of type $12$ that repeats infinitely many times in the infinite branch and let $h_1$ be minimal. Then, for every long base $\lambda$, the items $h_{\alpha(\lambda)}$ and $h_{\beta(\lambda)-1}\ne h_{\rho_A-1}$ are minimal.

In other words, all items that are not covered by a short base are minimal and belong to the same minimal tribe  $\mmm$ and all the other items belong to tribes that dominate $\mmm$.
\end{lem}
\begin{proof}
By Lemma \ref{lem:mincar}, if $\Omega_1$ is obtained from $\Omega$ by a (complete) entire transformation $\D 5$, then the item $h_1^{(1)}$ of $\Omega_1$ is also minimal. Hence, in the infinite branch, the items $h_1^{(i)}$ and $h_{\beta(\mu_i)-1}^{(i)}$ are always minimal and the items $h_j$, $1< j<  \beta(\mu_i)-1$ belong to tribes that dominate the tribe of $h_1^{(i)}$, where $\mu_i$ is the carrier of $\Omega_i$. 

Let $\lambda$ be the (uniquely defined) long base so that $\beta(\lambda)=\rho_A$. We now consider an arbitrary long base $\kappa$ of the generalised equation $\Omega$ distinct from $\lambda$. To prove the lemma, it suffices to show that $h_{\beta(\kappa)-1}$ is minimal.

For every generalised equation $\Omega_i$ we have that
\begin{description}
\item[Case I] either the boundary $1$ of $\Omega_i$ corresponds to a boundary $j_i$ of $\Omega$, $j_i\ge \beta(\kappa)$ in $\Omega$,
\item[Case II]  or there exists a base $\nu_i$ of $\Omega_i$ so that $\beta(\nu_i)$ corresponds to the boundary $\beta(\kappa)$ of $\Omega$.
\end{description}

Furthermore, since $\kappa\ne \lambda$, by Lemma \ref{lem:alltransfer}, there exists $N$ so that for all $n\ge N$ the generalised equation  $\Omega_n$ satisfies Case I, and for all $m<N$ the generalised equation $\Omega_m$ satisfies Case II. 

Notice that the boundary $\beta(\mu_{N-1})$ of $\Omega_{N-1}$ corresponds to a boundary $j_{N}$ of $\Omega$, so that $j_{N} \ge  \beta(\kappa)$. Suppose that $j_N=\beta(\kappa)$ (i.e. the boundary $\beta(\nu_{N-1})=\beta(\mu_{N-1})$ corresponds to the boundary $\beta(\kappa)$ of $\Omega$). Then, the boundary $\beta(\mu_{N-1})-1$ of $\Omega_{N-1}$ corresponds to a boundary between $\beta(\kappa)-1$ and $\beta(\kappa)$ in $\Omega$. Since $h_{\beta(\mu_{N-1})-1}^{(N-1)}$ is minimal, so is $h_{\beta(\kappa)-1}$.

If $j_N>\beta(\kappa)$ in $\Omega$ (i.e. $\nu_{N-1}$ is a transfer base of $\Omega_{N-1}$), then, it is not hard to see that the boundary $\alpha(\mu_{N-1})$ of $\Omega_{N-1}$ corresponds to a boundary between $\beta(\kappa)-1$ and $\beta(\kappa)$ in $\Omega$. We conclude that $h_{\beta(\kappa)-1}$ is minimal.
\end{proof}

\begin{cor} \label{cor:prelasttransfer}
Let $\Omega$ be a generalised equation of type $12$ that repeats infinitely many times in the infinite branch and let the item $h_1$ of $\Omega$ be a minimal item that belongs to the minimal tribe $\mmm=t(h_1)$ of $\Omega$.  Then, the tribe $t(h_i)$ of any item  $h_i$ of $\Omega$ dominates the tribe $t(h_1)$, $i=2,\dots, \rho_A-1$.
\end{cor}
\begin{proof}
By Lemma \ref{lem:alltransfer}, all bases of $\Omega$ but, perhaps, bases $\nu$ contained in a base $\lambda$ so that $\beta(\lambda)=\rho_A$, are transferred infinitely many times. Hence, by Lemma \ref{lem:1stlast}, it now follows that the tribes of all the items $h_i$, $i=1,\dots, \alpha(\lambda)-1$ of $\Omega$ dominate the tribe $t(h_1)$, where $\lambda$ is the (uniquely defined) long base so that $\beta(\lambda)=\rho_A$. 

Since the base $\Delta(\lambda)$ is not contained in $\lambda$, so the tribe of every item that belongs to $\Delta(\lambda)$ dominates the tribe $t(h_1)$. It follows that the base $\Delta(\lambda)$, and so $\lambda$, belong to a tribe that dominates $t(h_1)$. Therefore, the tribe of every item that belongs to $\lambda$ dominates $t(h_1)$.
\end{proof}

\begin{defn}
It follows from Corollary \ref{cor:prelasttransfer} that if $\Omega$ is a generalised equation of type $12$ that repeats infinitely many times in the infinite branch, then there is only one minimal tribe in $\Omega$, i.e. $|\mm(\Omega)|=1$. We further denote this minimal tribe by $\mmm$.
\end{defn}

\subsubsection*{Structure of the Quadratic Part}

Let $\Omega=\langle \Upsilon, \Re_\Upsilon\rangle$ be a generalised equation of type $12$ that repeats infinitely many times in the infinite branch and assume that the tribe of $h_1$ is minimal. Denote by $\BS_A(\Omega)$ the set of all active bases of the generalised equation $\Omega$. Our next goal is to show that the set  
$$
\{ h(\eta)\mid \eta \in \BS_A(\Omega)\}
$$
is a generating set of the subgroup of $G_\Omega$ generated by the active items of $\Omega$ (i.e. the subgroup generated by the items from the quadratic section of $\Omega$). In this new generating set, we prove that this subgroup is Tietze-equivalent to a one-relator quotient of a partially commutative group, where the relation is a quadratic word.

To simplify the notation, below we sometimes write $\eta$ instead of $h(\eta)$. Firstly, we use induction to define the map $\psi$ from the subgroup generated by the items of the quadratic part of $\Upsilon$ to the subgroup $\langle  \BS_A(\Omega) \rangle$ of $G_\Upsilon$ as follows.

Recall that if $\mu_1$ is the long leading base, then every item $h_i$, $i=1,\dots, \beta(\mu_1)-2$ is covered by a short base $\nu_i$. Then, for every item $h_i$, $i=1,\dots, \beta(\mu_1)-2$, we set $\psi(h_i)=\nu_i$. 

Notice that 
$$
h_{\beta(\mu_1)-1}=h_{\beta(\mu_1)-2}^{-1} \dots h_2^{-1}h_1^{-1}\mu_1. 
$$
We set 
$$
\psi(h_{\beta(\mu_1)-1})=\psi(h_{\beta(\mu_1)-2}^{-1} \dots h_2^{-1}h_1^{-1})\mu_1. 
$$
Note that the image of the item $h_{\beta(\mu_1)-1}$ is a word $w_{h_{\beta(\mu_1)-1}}$ in the bases $\nu_i$ so that $\alpha(\nu_i)<\beta(\mu_1)-1$ and every such base occurs in $w_{h_{\beta(\mu_1)-1}}$ precisely once.

Suppose that the map $\psi$ is defined for all the items $h_i$, $i<k$, and that if an item $h_i$ is covered only by long bases, then the image $\psi(h_i)$ is a word $w_{h_i}$ in the bases $\nu$ so that $\alpha(\nu)<i$ and every such base occurs in $w_{h_i}$ precisely once. 

We define $\psi(h_k)$ recursively. If $h_k$ is covered by a short base $\nu_k$, then we set $\psi(h_k)=\nu_k$. Otherwise, $h_k$ is only covered by (two) long bases $\kappa_1$ and $\kappa_2$, where $\alpha(\kappa_2)=\beta(\kappa_1)- 1=k$. Then, we define 
$$
\psi(h_k)=\psi(h_{k-1}^{-1} \dots h_{1}^{-1} h_1 \dots h_{\alpha(\kappa_1-1)})\kappa_1=\psi((h_{k-1}^{-1}  \dots h_{\alpha(\kappa_1)}^{-1})\kappa_1=\nu_{k-1}^{-1}\dots \nu_{\alpha(\kappa_1)+1}^{-1} \psi(h_{\alpha(\kappa_1)}^{-1})\kappa_1.
$$ 
Since $\alpha(\kappa_1)<k$ and the item $h_{\alpha(\kappa_1)}$ is covered by two long bases, it follows by induction assumption that $\psi(h_{\alpha(\kappa_1)})$ is a word in the bases $\nu$ so that $\alpha(\nu)<\alpha(\kappa_1)$, and every such base occurs in this word precisely once. It follows that $\psi(h_k)$ is a word in the bases $\nu$ so that $\alpha(\nu)<k$, and every such base occurs in this word precisely once. 

\begin{expl}
Let $\Omega$ be the generalised equation given on Figure \ref{pic:ge}. Then,
$$
\begin{array}{llll}
\psi(h_1)=\nu, & \psi(h_2)=\eta, & \psi(h_3)=\eta^{-1}\nu^{-1}\mu, & \psi(h_4)=\Delta(\lambda),\\ \psi(h_5)=\Delta(\lambda)^{-1}\mu^{-1}\nu\eta\Delta(\nu),&  \psi(h_6)=\lambda, & \psi(h_7)=\Delta(\eta).
\end{array}
$$
\end{expl}

We now determine the presentation of the subgroup generated by active items of the quadratic part of $\Upsilon$ in the generators $\BS_A(\Omega)$.

Note that, by Lemma \ref{lem:1clsec}, every boundary of $\Omega$ touches precisely two bases, i.e. for every boundary $i$, $1 <i< \rho$, there exist a base $\nu_i$ such that $\alpha(\nu_i)=i$ and a base $\lambda_i$ such that $\beta(\lambda_i)=i$. Therefore we can define the words:
$$
W_1= \mu_1 \nu_{\beta(\mu_1)} \nu_{\beta(\nu_{\beta(\mu_1)})}\cdots\nu_m \nu_{\beta(\nu_m)},
$$
where $\beta(\nu_{\beta(\nu_m)})=\rho_{\Omega}$ and 
$$
W_2=\nu_1\nu_{\beta(\nu_1)} \nu_{\beta(\nu_{\beta(\nu_1)})}\cdots \nu_n \nu_{\beta(\nu_n)},
$$
where $\beta(\nu_{\beta(\nu_n)})=\rho_{\Omega}$.
We set 
\begin{equation} \label{eq:W}
W=W_1W_2^{-1}.
\end{equation}

For each item $h_i$, $i=\rho_A,\dots, \rho_\Omega-1$, i.e. for items that does not belong to the active quadratic section, we set $\psi(h_i)=h_i$.

\begin{expl}
Let $\Omega$ be the generalised equation given on  Figure \ref{pic:ge}. Then, 
$$
W_1=\nu\eta\Delta(\nu)\lambda\Delta(\eta), \quad W_2=\mu\Delta(\lambda)\Delta(\mu).
$$
\end{expl}

\begin{lem} \label{lem:9.19}
In the above notation, the map $\psi$ induces an isomorphism from $G_\Upsilon$ to the group
$$
\caK_\Upsilon=\langle \BS_A(\Omega), h_{\rho_A},\dots, h_{\rho_\Omega-1} \mid W=1, \eta^{\varepsilon(\eta)}\Delta(\eta) ^{-\varepsilon(\Delta(\eta))}=1, \nu^{\varepsilon(\nu)} h(\Delta(\nu))^{-\varepsilon(\Delta(\nu))},R'\rangle
$$
where $\eta$ runs over the set of quadratic bases, $\nu$ runs over the set of quadratic-coefficient bases and  $R'$ is the set of relations of the non-active part, i.e. $R'$ is the set of relations $h(\mu)^{\varepsilon(\mu)} h(\Delta(\mu) )^{-\varepsilon(\Delta(\mu))}$, for all pairs of non-active bases $\mu$, $\Delta(\mu)$.
\end{lem}
\begin{proof}
Straightforward computation shows that for every base $\eta$ that belongs to the quadratic part such that either $\eta$ is short or $\beta(\eta)\ne \rho_A$, we have $\psi(h(\eta))=\eta$. Indeed, if $\eta$ is short, then the statement is obvious. Let $\eta$ be a long base. Then, the item $h_{\beta(\eta)-1}$ is covered by two long bases, $\eta$ and $\kappa$. We have
$$
\begin{array}{l}
\psi(h(\eta))=\psi(h_{\alpha(\eta)})\dots \psi(h_{\beta(\eta)-2})\psi(h_{\beta(\eta)-1})=\\
\psi(h_{\alpha(\eta)})\dots \psi(h_{\beta(\eta)-2})\cdot \nu_{\beta(\eta)- 2}^{-1}\dots \nu_{\alpha(\eta)+1}^{-1} \psi(h_{\alpha(\eta)}^{-1})\eta=\\
\psi(h_{\alpha(\eta)})\cdot \nu_{\alpha(\eta)+1}\cdots \nu_{\beta(\eta)-2}\cdot \nu_{\beta(\eta)- 2}^{-1}\dots \nu_{\alpha(\eta)+1}^{-1} \psi(h_{\alpha(\eta)}^{-1})\eta=\eta
\end{array}
$$
We conclude that $\psi$ maps $h(\eta)^{\varepsilon(\eta)}h(\Delta(\eta))^{-\varepsilon(\Delta(\eta))}$ to the identity in $\caK_\Upsilon$ for all pairs of dual bases except for the pair $\lambda, \Delta(\lambda)$, where $\lambda$ is the (uniquely defined) long base so that $\beta(\lambda)=\rho_A$. For the pair $\lambda, \Delta(\lambda)$, we have
$$
\psi(h(\lambda))=\psi(h_{\alpha(\lambda)})\dots\psi(h_{\beta(\lambda)-1})=
\psi(h_{\alpha(\lambda)})\nu_{\alpha(\lambda)+1}\dots \nu_{\rho_A-1}.
$$
Thus, $\psi(h(\lambda))$ is the product of all bases $\nu$ of $\Omega$ so that $\nu\ne \lambda$. It follows that $\psi(h(\lambda)^{\varepsilon(\lambda)}h(\Delta(\lambda))^{-\varepsilon(\Delta(\lambda))})$ is trivial in $\caK_\Upsilon$ since $W=1$ and $\lambda^{\varepsilon(\lambda)}\Delta(\lambda) ^{-\varepsilon(\Delta(\lambda))}=1$ are relations of $\caK_\Upsilon$. We conclude that $\psi$ is a (surjective) homomorphism.

We define the map $\varrho:\caK_\Upsilon\to G_{\Upsilon}$ by $\varrho(\eta)=h(\eta)$. The map $\varrho$ extends to a homomorphism. Straightforward verification shows that every relation of $\caK_\Upsilon$ maps to the identity in $G_{\Omega}$ and that $\varrho\psi=\id$ and $\psi\varrho=\id$ and we conclude that $\psi$ is an isomorphism.
\end{proof}

Let us now show how to extend the isomorphism $\psi$ to an isomorphism of $G_\Omega$.

Let $\Xi$ be the set of relations defined as follows. 
\begin{itemize}
\item
For every pair of items $h_i,h_j$ of $\Omega$ so that $\Re_\Upsilon(h_i,h_j)$ and $h_i$, $h_j$ are covered by short bases $\nu_i$, $\nu_j$, correspondingly, we set $[\nu_i,\nu_j]=1\in \Xi$.
\item
For every pair of items $h_i,h_j$ of $\Omega$ so that $\Re_\Upsilon(h_i,h_j)$ and $h_i$ is covered by a short base $\nu_i$ and $h_j$ is a non-active item, we set $[\nu_i, h_j]=1\in \Xi$.
\item
For every pair of items $h_i,h_j$ of $\Omega$ so that $\Re_\Upsilon(h_i,h_j)$ and $h_i$ is covered only by long bases and $h_j$ is a non-active item, we set $[\eta,h_j]=1\in \Xi$, for all bases $\eta$ in the quadratic section.
\item
For every pair of items $h_i,h_j$ of $\Omega$ so that $\Re_\Upsilon(h_i,h_j)$ and $h_i$ and $h_j$ belong to the non-active part, we set $[h_i,h_j]=1 \in \Xi$.
\end{itemize}

\begin{lem}\label{lem:isoquad}
Let $\Omega=\langle \Upsilon, \Re_\Upsilon\rangle$ be a generalised equation  of type $12$ that repeats infinitely many times in the infinite branch  and let $h_1$ be minimal. Then, in the above notation, the map $\psi$ induces an isomorphism from $G_\Omega$ to the group
\begin{equation} \label{eq:cak}
\caK=\langle \BS_A(\Omega), h_{\rho_A},\dots, h_{\rho_\Omega-1} \mid W=1, \eta^{\varepsilon(\eta)}\Delta(\eta) ^{-\varepsilon(\Delta(\eta))}=1, \nu^{\varepsilon(\nu)} h( \Delta(\nu))^{-\varepsilon(\Delta(\nu))},R', \Xi \rangle,
\end{equation}
where $\eta$ runs over the set of quadratic bases, $\nu$ runs over the set of quadratic-coefficient bases and  $R'$ is the set of relations of the non-active part, i.e. $R'$ is the set of relations $h(\mu)^{\varepsilon(\mu)} h(\Delta(\mu) )^{-\varepsilon(\Delta(\mu))}$, for all pairs of non-active bases $\mu$, $\Delta(\mu)$.
\end{lem}
\begin{proof}
By Lemma \ref{lem:9.19}, the map $\psi$ induces an isomorphism from $G_\Upsilon$ to $\caK_\Upsilon$. Observe that if an item $h_i$ belongs to two long bases, then, by Lemma \ref{lem:mintribe}, it belongs to the minimal tribe. Since any other item $h_j$ from the quadratic part belongs to a tribe that dominates the tribe of $h_i$, it follows that $h_j\notin \Re_\Upsilon(h_i)$, i.e. if $\Re_\Upsilon(h_i,h_k)$, then $h_k$ is non-active. Now a straightforward verification shows that $\psi$ is an epimorphism.

Furthermore, since, by Remark \ref{rem:completed}, the set $\Re_\Upsilon$ is completed, it follows that $\psi$ is an isomorphism. Indeed, if an item $h_i$ belongs to two long bases, then by Lemma \ref{lem:1stlast}, it belongs to the minimal tribe. Since items of the quadratic part belong to tribes that dominate the minimal tribe and, by assumption, $\Re_\Upsilon$ is completed, it follows that if  $\Re_\Upsilon(h_i,h_j)$, then $\Re_\Upsilon(h_i,h_k)$ for all items $h_k$ from the quadratic part.
\end{proof}

\bigskip

Our next goal is to show that under the above conditions, one can take the quadratic equation $W$ to the standard form. In order to do so, we replace the generalised equation $\Omega$ by another generalised equation, in which if an item belongs to a tribe which strictly dominates the minimal tribe, then this item is covered by a short quadratic-coefficient base.

By Lemma \ref{lem:mincar}, items that belong to tribes which strictly dominate the minimal tribe belong to $\mathcal{F}(\Omega_0)$. By Remark \ref{rem:discrfam}, there exists a  fundamental sequence so that for any item $h_i^{(0)}$ that belongs to a tribe that strictly dominates $t(h_1^{(0)})$, there exists a word $w_i(h^{(1)})\in G_{\Omega_{1}}$ so that for all solutions $H^{(0)}=\varphi\pi(v_0,v_1)H^{(1)}$ of the discriminating family, we have that $\varphi \pi(v_0,v_1)(h_i)=w_i(h^{(1)})$ and $H_i^{(0)}=H^{(1)}(w_i)$.

We replace $\Omega_{0}=\langle \Upsilon_0, \Re_{\Upsilon_0} \rangle$ by a new generalised equation $\Omega_{0}'=\langle \Upsilon_0', \Re_{\Upsilon_0'} \rangle$ constructed as follows. Replace the non-active part of $\Omega_{0}$ by the non-active part of $\Omega_{1}=\langle \Upsilon_1,\Re_{\Upsilon_1}\rangle$. For any item $h_i^{(0)}$ in the active part that does not belong to the minimal tribe $t(h_1^{(0)})$ let $w_i(h^{(1)})={h_{i,1}^{(1)}}^{\epsilon_{i,1}}\dots {h_{i,{k_i}}^{(1)}}^{\epsilon_{i,k_i}}$, $\epsilon_{i,j} \in \{1, -1\}$. 

Introduce new boundaries in the item $h_i^{(0)}$ so that $h_i^{(0)}={{h_{i,1}^{(0)}}'}^{\epsilon_{i,1}} \cdots {{h_{i,{k_i}}^{(0)}}'}^{\epsilon_{i,k_i}}$, $\epsilon_{i,j} \in \{1, -1\}$, $j=1,\dots, k_i$. Erase the short base $\lambda$ of $\Omega_{0}$ that covers $h_i^{(0)}$. We now introduce new bases $\lambda_{i,1},\dots, \lambda_{i,{k_i}}$ along with the corresponding duals in such a way that $\lambda_{i,j}$ covers the item ${h_{i,j}^{(0)}}'$, $j=1,\dots,k_i$ and the dual $\Delta(\lambda_{i,j})$ covers the item $h_{i,j}^{(1)}$ and $\varepsilon(\lambda_{i,j})=\epsilon_{i,j}$,  $\varepsilon(\Delta(\lambda_{i,j}))=1$. 

Note that, by Lemma \ref{lem:1stlast}, it follows that if a short base covers an item that belongs to a tribe which strictly dominates $t(h_1^{(0)})$, then either it is a quadratic-coefficient base or it is quadratic and its dual is also a short base. Hence, the base $\Delta(\lambda)$ is a short base that covers an item that dominates $t(h_1^{(0)})$.

The set $\Re_{\Upsilon_0'}$ of $\Omega_0'$ is defined naturally: $\Re_{\Upsilon_0'}({h_{i,j}^{(0)}}', {h_{i',j'}^{(0)}}')$ if and only if $\Re_{\Upsilon_1}(h_{i,j}^{(1)}, h_{i',j'}^{(1)})$ or $\Re_{\Upsilon_0}(h_i^{(0)}, h_{i'}^{(0)})$; $\Re_{\Upsilon_0'}({h_{i,j}^{(0)}}', h_{i'}^{(0)})$ if and only if $\Re_{\Upsilon_0}(h_{i}^{(0)}, h_{i'}^{(0)})$; $\Re_{\Upsilon_0'}({h_{i,j}^{(0)}}', h_{i}^{(1)})$ if and only if $\Re_{\Upsilon_1}(h_{i,j}^{(1)}, h_{i}^{(1)})$; the other relations in $\Re_{\Upsilon_0'}$ are naturally induced by $\Re_{\Upsilon_0}$ and $\Re_{\Upsilon_1}$. We assume that $\Re_{\Upsilon_0'}$ is completed (see Remark \ref{rem:completed}). We conclude that:
\begin{itemize}
\item $\Omega_{0}'$ is a generalised equation of type $12$;
\item if $\varsigma: G_{\Omega_{0}} \to G_{\Omega_{0}'}$ is a natural homomorphism, then any solution $H$ of the generalised equation $\Omega_0$ that belongs to a fundamental sequence induces a solution $H'$ of the generalised equation $\Omega_0'$. In other words, for any homomorphism from the fundamental sequence there exists a homomorphism from $G_{\Omega_{0}'}$  so that the following diagram is commutative
$$
\xymatrix{
  G_{\Omega_{0}} \ar[rr]^{\varsigma} \ar[dr]_{H} &  &    G_{\Omega_{0}'} \ar[dl]^{H'}    \\
                & G_{\Omega_1}\ar[d]_{H^{(1)}}\\
 & \GG                               }
$$
\item $G_{\Omega_1}$ is a retraction of $G_{\Omega_0'}$.
\end{itemize}

Note that if $\nu$ is an active base of $\Omega_{0}'$, then either the tribe of every item ${h_i^{(0)}}'$, $i=\alpha(\nu),\dots, \beta(\nu)-1$, covered by $\nu$ dominates $t(h_1^{(0)})$ and at least one of the items belongs to the tribe $t(h_1^{(0)})$ or else every item covered by $\nu$ belongs to a tribe that strictly dominates $t(h_1^{(0)})$ and then the base $\nu$ is a short quadratic-coefficient base. In other words, in $\Omega'_0$ the tribe of every quadratic base is the same minimal tribe and the tribes of the quadratic-coefficient bases dominate this minimal tribe.

Let 
$$
\calL=\langle\{ \eta\mid \eta\in \BS_A(\Omega_{v_0}')\}, h_{\rho_A}, \dots, h_{\rho_\Omega-1}\mid 
\mu^{\varepsilon(\mu)}\Delta(\mu) ^{-\varepsilon(\Delta(\mu))}=1, \nu^{\varepsilon(\nu)}h(\Delta(\nu)) ^{-\varepsilon(\Delta(\nu))}=1\rangle
$$
for all quadratic bases $\mu$ and all quadratic coefficient bases $\nu$. It is clear that $\calL$ is a free group.

Let $\caM$ be quotient of $\calL$ by the set of commutators $\Xi'$ defined analogously to $\Xi$, see Equation (\ref{eq:cak}). By definition,  $\caM$ is a free partially commutative group.

The word $W$ defined in Equation (\ref{eq:W}) is a quadratic word in the free group $\calL$. It can be taken to the surface relation form (\ref{eq:orient}) or (\ref{eq:nonorient}) by an automorphism of the free group $\calL$, see \cite{ComEd}.

\begin{lem} \label{lem:qnorm}
Let $\phi$ be the automorphism of the free group $\calL$ that takes the quadratic equation $W=1$ to the normal form. Then, $\phi$ induces an automorphism $\phi'$ of the partially commutative group $\caM$ that takes the quadratic equation $W=1$ to the normal form and fixes all the quadratic-coefficient bases and items $h_i$, $i=\rho_A,\dots, \rho_\Omega-1$.
\end{lem}
\begin{proof}
Denote by $\var(W)$ the set of variables of $W=1$ that occur in $W$ exactly twice. By definition of $\calL$ and $\caM$ and by construction of $\Omega_0'$, we get that if $x\in \var(W)$, then $x$ is a quadratic base of $\Omega_0'$ and the tribe of $x$ is the minimal tribe $t(h_1^{(0)})$.

Suppose that $W = AxBxC$. Let $\phi_x$ be the automorphism of $\calL$ induced by the map  $x\to A^{-1}xA B^{-1}$. The automorphism $\phi_x$ transforms the word $W$ into 
$$
\phi_x(W) = A A^{-1}x A B^{-1}BA^{-1} xAB^{-1}C = x^2AB^{-1}C. 
$$
Every letter in $A$ and $B$ is either a variable of $W=1$ (then it corresponds to a quadratic base which belongs to $t(h_1^{(0)})$) or it is a coefficient of $W=1$ (then it corresponds to a quadratic-coefficient base which belongs to a tribe that strictly dominates $t(h_1^{(0)})$). Hence, by \cite{pcauto}, it follows that $\phi_x$ induces an automorphism of $\caM$.

Observe that $AB^{-1}C$ is a quadratic word in fewer variables than $W$. The statement now follows by induction.

Suppose now that every variable $x$ of $W$ occurs in it as $x$ and as $x^{-1}$. Let $W = Ax^{-1}BxC$, where the number of variables $|\var(B)|$ in $B$ is minimal among all such decompositions of $W$. In particular, it follows that $B$ is linear.

If $\var(B) = \emptyset$, then we consider the automorphism $\phi_x$ of $\calL$ defined by the map $x \mapsto xC^{-1}$. Then, $\phi_x(W)=ACx^{-1}Bx$.  As above, the automorphism $\phi_x$ induces an automorphism of $\caM$. Note that the number of variables in $AC$ is strictly lower than that of $W$ and the statement, in this case, follows by induction.

Let $\var(B)\ne \emptyset$. Then, $B = B_1y^{\delta}B_2$, where $\delta=\pm 1$, and neither $B_1$ nor $B_2$ contains $y^{\pm 1}$. Applying the automorphism $y \mapsto y^{-1}$, if necessary, we may assume that $\delta=1$.
 
Consider the automorphism $\phi_y$ of $\calL$ defined by $y \mapsto B_1^{-1}yB_2^{-1}$. Then, $\phi_y(W) = \phi_y(A)x^{-1}yx\phi_y(C)$. The variable $y^{-1}$ occurs either in $A$ or in $C$. We assume that $y^{-1}$ occurs in $C$ (the other case is similar), i.e. $C = C_1y^{-1}C_2$. Then, 
$$
\phi_y(A) = A, \ \phi_y(C)= C_1B_2y^{-1}B_1C_2 \hbox{ and }\phi_y(W) = Ax^{-1}yxC_1B_2y^{-1}B_1C_2.
$$

Applying the automorphism $\phi_x$ defined by the map $x \mapsto x(C_1B_2)^{-1}$, we get that 
$$
\phi_x\phi_y(W)= AC_1B_2x^{-1}yxy^{-1}B_1C_2.
$$
Let $\phi_1$ be the automorphism that conjugates $x$ and $y$ by $ACB_2$. Then
$$
\phi_1\phi_x\phi_y(W)= x^{-1}yxy^{-1}B_1C_2AC_1B_2. 
$$
As above, the automorphisms $\phi_1$, $\phi_x$ and $\phi_y$ induce automorphisms of $\caM$. We now observe that $B_1C_2AC_1B_2$ is a quadratic word in fewer variables than $W$ and the statement now follows by induction.
\end{proof}

\begin{cor}\label{cor:quadNF}
In the notation of {\rm Lemma \ref{lem:qnorm}}, the automorphism $\phi'$ induces an automorphism of the group $\mathcal K$ that takes the quadratic relation to the normal form. Therefore, if $\Omega=\langle \Upsilon, \Re_\Upsilon\rangle$ is a generalised equation of type $12$ that repeats infinitely many times in the infinite branch, $h_1$ is minimal and $\Re_\Upsilon$ is completed, then, $G_\Omega$ is isomorphic to the group
$$
\caK=\langle \BS_A(\Omega), h_{\rho_A},\dots, h_{\rho_\Omega-1} \mid W=1, \eta^{\varepsilon(\eta)}\Delta(\eta) ^{-\varepsilon(\Delta(\eta))}=1, \nu^{\varepsilon(\nu)}  \Delta(\nu)^{-\varepsilon(\Delta(\nu))}=1,R', \Xi' \rangle,
$$
where $\eta$ runs over the set of quadratic bases, $\nu$ runs over the set of quadratic-coefficient bases and $W$ is a quadratic word in the normal form 
$$
\begin{array}{l}
\phantom{a}[\eta_1,\eta_2]\cdots[\eta_{2g-1},\eta_{2g}]{\lambda_{2g+1}}^{\eta_{2g+1}}\cdots {\lambda_m}^{\eta_{m}}=\\
\phantom{a}\pi(v_0,v_1)([\eta_1,\eta_2]\cdots[\eta_{2g-1},\eta_{2g}]{\lambda_{2g+1}}^{\eta_{2g+1}}\cdots {\lambda_m}^{\eta_{m}})
\end{array}
$$
or 
$$
\eta_1^2\cdots \eta_{2g}^2 {\lambda_{2g+1}}^{\eta_{2g+1}}\cdots {\lambda_m}^{\eta_{m}}=\pi(v_0,v_1)(\eta_1^2\cdots \eta_{2g}^2 {\lambda_{2g+1}}^{\eta_{2g+1}}\cdots {\lambda_m}^{\eta_{m}})
$$
and where $R'$ is the set of relations of the non-active part, i.e. $R'$ is the set of relations $\mu^{\varepsilon(\mu)} \Delta(\mu)^{-\varepsilon(\Delta(\mu))}$, where $\mu$, $\Delta(\mu)$ runs over the set of pairs of non-active bases.
\end{cor}

\subsubsection*{Construction of a Homomorphism from $G_{\Omega_0}$ to a Graph Tower}

We finally address the construction of the graph tower. Our goal is to construct a graph tower $(\Ts_0, \HH_0)$ and a homomorphism $\tau_0$ from $G_{\Omega_0}$ to $\Ts_0$ such that for every homomorphism $H$ from $G_{R(\Omega_0)}$ to $\GG$ from the fundamental sequence there exists a homomorphism $H'$ such that Diagram (\ref{diag:period}) commutes. Recall that we use induction on the length of the fundamental sequence $G_{R(\Omega_0)}\to G_{R(\Omega_1)}\to \dots \to G_{R(\Omega_q)}$ and our induction hypothesis are IH, IH1 and IH2 defined on page  \pageref{pagen}.

We further assume that $\Omega_0$ satisfies the properties of $\Omega_0'$, i.e. the short bases that belong to a tribe that strictly dominate the minimal tribe are quadratic-coefficient bases.

We aim to prove that there exists the following commutative diagram
$$
\xymatrix{
G_{\Omega_0}\ar[d]_{\pi(v_0,v_1)} \ar[rr]^{\tau_0}&  &\Ts_0 \ar@{->>}@<-1ex>[d]& & \ar@{->>}[ll] \HH_0 \ar@{->>}@<-1ex>[d]\\
G_{\Omega_1}  \ar[rr]_{\tau_1}	\ar[dr]_H				 &  & \ar[dl]^{H'} \Ts_1\ar@{_{(}->}@<-1ex>[u]  & & \ar@{->>}[ll] \HH_1\ar@{_{(}->}@<-1ex>[u]\\
 & \GG&
}
$$
and that the graph tower $(\Ts_0, \HH_0)$ satisfies all the induction hypothesis. We fix the following notation for the generators of the groups we consider: $G_{\Omega_i} =\langle h^{(i)}\rangle$, $i = 0,1$, $\HH_1 = \langle x^{(1)}\rangle$, and $\Ts_1 = \langle y^{(1)}\rangle$. Denote by $\varphi_0$ the composition of $\pi(v_0,v_1)$ and $\tau_1$. 

Let $\caN$ be the subset of the set of quadratic bases of $\Omega_0$ so that, for each pair of dual bases $\mu, \Delta(\mu)$, the set $\caN$ contains exactly one of them. Let $n$ be the cardinality of the set $\caN$. Notice that, by definition, the set $\caN$ contains long bases and short bases that belong to the minimal tribe.

For every $\eta \in \caN$, set $\varphi_0(h^{(0)}(\eta))=w_\eta(y^{(1)})$. We define the canonical parabolic subgroup $\KK$ of $\HH_1$ to be $\BA_{E_d(\Gamma_1)}(\{w_\eta(x^{(1)}), \eta \in \caN\})$.

\begin{lem}
The subgroup $\KK$ of $\HH_1$ is $E_d(\Gamma_1)$-\cool.
\end{lem}
\begin{proof}
Let us begin with an observation. By Remark \ref{rem:discfam}, there exists a canonical parabolic subgroup $\GG_\KK< \GG$ such that for a fundamental sequence of solutions, we have that $\langle\az({H^{(0)}}'(\KK))\rangle=\GG_\KK$. By definition of $\KK$ and by the induction hypothesis IH1, we have that ${H^{(1)}}'(\varphi_0( h(\eta))) \lra \GG_\KK$ for a fundamental sequence of $G_{\Omega_1}$. Furthermore, if $a_i \lra \{{H^{(1)}}'(\varphi_0( h(\eta)))\mid \eta \in \mathcal N\}$, where $a_i\in\cA$, $\GG=\GG(\cA)$, again by the induction hypothesis IH1, we have that $a_i\in \KK$ and so $a_i\in \GG_\KK$, hence we have that
$$
\langle {\az(\{{H^{(1)}}'(\varphi_0( h(\eta)))\mid \eta \in \mathcal N\} )}^\perp\rangle=\GG_\KK.
$$

Since every solution $H^{(0)}$ from the fundamental sequence is obtained from a solution $H^{(1)}$ of $\Omega_1$ from the fundamental sequence by precomposing it with a canonical automorphism (and $\pi(v_0,v_1)$), it follows that $a_i \lra  \{H^{(0)}(\eta), \eta \in \caN\}$ for all solutions $H^{(0)}$ of the fundamental sequence if and only if $a_i \lra \{H^{(1)}(\pi(v_0,v_1)(h(\eta))), \eta \in \caN\}$ for all $H^{(1)}$ from the fundamental sequence, where $a_i\in \GG$. Therefore, $\langle \az (\{H^{(0)}(\eta), \eta \in \caN\})^\perp\rangle= \GG_\KK$ and since all the bases $\eta$ belong to the same minimal tribe, it follows that $\langle \az(H^{(0)}(\eta))^\perp\rangle=\GG_\KK$, for all $\eta \in \mathcal N$. Furthermore, since any minimal item $h_i^{(0)}$ belongs to the same minimal tribe, we have that $\langle {\az(H_i^{(0)})}^\perp\rangle=\GG_\KK$. Hence, if $h_j^{(0)}$ is so that the tribe of $h_j^{(0)}$ dominates the minimal tribe, then we have that $H_j^{(0)}\in \GG_{\KK}^\perp$.

Let us now address the statement of the lemma. By definition, the group $\KK$ is closed, i.e. $\KK=\KK^{\perp\perp}$. We only need to show that $\KK^\perp$ is $E_d(\Gamma_1)$-directly indecomposable.

Assume the contrary, then $\KK^\perp=\KK_1\times \dots \times \KK_r$, where $r>1$ (and this decomposition is with respect to the edges from $E_d(\Gamma_1)$). Without loss of generality, we shall assume that $r=2$. By Remark \ref{rem:discfam}, there exist canonical parabolic subgroups $\GG_{\KK_i}<\GG$, $i=1,2$ such that for all solutions $H^{(0)}$ from the fundamental sequence we have ${H^{(0)}}'(\KK_i)=\GG_{\KK_i}$, $i=1,2$. Furthermore, by the induction hypothesis IH1, $\GG_{\KK_1} \lra \GG_{\KK_2}$. As we have shown, if $h_i^{(0)}$ is an item whose tribe dominates the minimal tribe, then ${H_i^{(0)}}\in \GG_\KK^\perp< \GG_{\KK_1}\times \GG_{\KK_2}$. It follows by Lemma \ref{lem:1stlast}, that $H^{(0)}([1,\rho_A])\in \GG_{\KK_1}\times \GG_{\KK_2}$.

By Lemma \ref{lem:1clsec}, the section $[1, \rho_A]$ is a closed section of $\Omega$ and, by Remark \ref{rem:recall}, the word $H^{(0)}[1,\rho_A]$ is a subword of a word in the DM-normal form. It follows by Lemma \ref{lem:pc}, that the  word $H^{(0)}[1,\rho_A]$ contains only a bounded number of  $\GG_{\KK_1}\times \GG_{\KK_2}$-alternations.

When applying an entire transformation from $\Omega_0$ to $\Omega_1$, the only items that could be mapped to a word of length greater than one by  the epimorphism $\pi(v_0,v_1)$ are $h_1^{(0)}$ and $h_{\cali(2)}^{(0)}$, where the boundary connection $(2, \mu, n_2)$ is introduced and $n_2$ is introduced between the boundaries $\cali(2)$ and $\cali(2)+1$ of $\Omega_0$. The image of $h_1^{(0)}$ is a word that contains at least one minimal item, namely the item $h_{\cali(2)}^{(1)}$. By Remark \ref{rem:manymin}, the image of $h_{\cali(2)}^{(0)}$ is the word $h_{\cali(2)}^{(1)}h_{n_2}^{(1)}$ and both of the items $h_{\cali(2)}^{(1)}$, $h_{n_2}^{(1)}$ are minimal. Hence, the image of $h^{(0)}([1,\rho_A])$ is a word in variables $h^{(1)}$ that contains at least $|h(\mmm)|+1$ minimal items. 

Repeating this argument, we conclude that $\pi(v_0,v_k)(h^{(0)}([1,\rho_A]))$ is a word in $h^{(k)}$ that contains at least $|h(\mmm)|+k$ minimal items. By \cite[Lemma 7.12]{CKpc}, in an infinite branch of type 12, there exists an infinite sequence of generalised equations $\Omega_0= \Omega_{n_1}= \dots $, so that every solution $H^{(n_k)}$ of $\Omega_{n_k}$ induces a solution $H^{(0)}=\pi(v_0, v_{n_k}) H^{(n_k)}$ of $\Omega_0$. Therefore, on the one hand, $H^{(0)}[1,\rho_A] \in \GG_{\KK_1}\times \GG_{\KK_2}$. On the other hand, $H^{(0)}[1,\rho_A]$ is graphically equal to a word $w(H^{(1)})$ that contains $|h(\mmm)|+n_k$ minimal items. Every minimal item defines at least one $\GG_{\KK_1}\times \GG_{\KK_2}$-alternation. Hence, the word $H^{(0)}[1,\rho_A]$ contains at least $|h(\mmm)|+n_k$ many $\GG_{\KK_1}\times \GG_{\KK_2}$-alternations, deriving a contradiction. Therefore, we have that $s=r=1$, $H_i^{(0)}$ is a block and $\KK^\perp$ is directly indecomposable.
\end{proof}

\begin{rem}\label{rem:blocks}
It follows from the above proof that for all solutions $H$ from the  fundamental sequence, $H_i \in \GG_{\KK^\perp}$ is a block element for all minimal items $h_i$ from $\Omega$ and $\BA(H_i)=\BA(\GG_{\KK^\perp})$.
\end{rem}

\begin{lem} \label{lem:eired}
Let $\Omega$ be a generalised equation of type $12$ that repeats infinitely many times in the infinite branch  and let $h_1$ be minimal. Then, there exists an element $U\in G_{\Omega}$ so that $H(U)$ is irreducible for all $H$ from the fundamental sequence.
\end{lem}
\begin{proof}
Since the word $H_1$ is a block (see Remark \ref{rem:blocks}), since the tribe of $h_i$ dominates the minimal tribe (see Lemma \ref{cor:prelasttransfer}), and since the word $H[1,i]$ is reduced as written, it follows that the word $H[1,i]$ is a block element for all $1< i\le \rho_A$. Furthermore, if the word $H[1,i]$ is not irreducible, then it is not cyclically reduced.

If the word $H_1H[2,i] =w v w^{-1}$ is not cyclically reduced, then by \cite[Proposition 3.18]{EKR}, the word $w=d_1d_2$, where $d_1$ is a left divisor of $H_1$, $H_1=d_1 u_1$, $d_2$ is a left divisor of $H[2,i]$, $H[2,i]=d_2u_2$, and $d_2\lra u_1$. Furthermore, since $H_1$ is a block element and $\BA(d_2)>\BA(H_1)$, it follows that $d_2\not \lra H_1$, hence $d_1\ne 1$.

Let $\mu$ be the carrier base of $\Omega$. Suppose that $\varepsilon(\mu)=\varepsilon(\Delta(\mu))$. In this case, we show that $H[1,\alpha(\Delta(\mu))-1]$ is irreducible. Assume the contrary. Then, from the above discussion, it follows that $H_1H[2,\alpha(\Delta(\mu))-1]=d_1 u d_1^{-1}$, where $d_1 \ne 1$ is a left divisor of $H_1$. On the other hand, by Lemma \ref{lem:1clsec}, the word $H[1,\beta(\Delta(\mu))]= H[1,\alpha(\Delta(\mu))-1]H(\Delta(\mu))$ is reduced - a contradiction since $d_1^{-1}$ right-divides $H[1,\alpha(\Delta(\mu))-1]$ and $d_1$ left-divides $H(\Delta(\mu))$. 

Suppose now that $\varepsilon(\mu)=-\varepsilon(\Delta(\mu))$. Let $\nu$ be the (uniquely defined) long base, so that $\alpha(\nu)=\beta(\mu)-1$. Note that, since $\Omega$ is formally consistent, see \cite[Definition 3.9]{CKpc}, we have that $\nu\ne \Delta(\mu)$. Let $\Omega'$ be obtained from $\Omega$ by a complete entire transformation. If $\varepsilon(\nu)=\varepsilon(\Delta(\nu))$ in $\Omega'$, then, by Lemma \ref{lem:mincar}, the argument above applies to the generalised equation $\Omega'$ and the carrier base $\nu$. 

Suppose that $\varepsilon(\nu)=-\varepsilon(\Delta(\nu))$ in $\Omega'$. Let $\Omega''$ be obtained from $\Omega'$ by a complete entire transformation. Then, we have that $\varepsilon(\mu)=\varepsilon(\Delta(\mu))$ in $\Omega''$. Notice that since the generalised equation $\Omega$ repeats infinitely many times in the infinite branch, the tribe $t(\mu)$ of $\mu$ is minimal in $\Omega''$. Without loss of generality, let $\alpha(\mu)<\alpha(\Delta(\mu))$. From the above argument, it follows that the word $H''[\alpha(\mu), \alpha(\Delta(\mu))-1]$ is irreducible.

Therefore, since $G_{R(\Omega)}\simeq  G_{R(\Omega')} \simeq G_{R(\Omega'')}$, the statement follows.
\end{proof}

Recall that quadratic words of the type $[x,y], x^2, z^{-1}cz$, where $c$ is a constant are called \emph{atomic quadratic words} or simply \emph{atoms}. Let $W = 1$ be a quadratic equation over $G$ written in the form $r_1 r_2\cdots r_k=d$, where $r_i$ are atoms and $d \in G$. The number $k$ is called the \emph{rank} of the quadratic equation. Suppose that the rank  $k$ of $W=1$ is greater than or equal to $2$. A solution $H$ of $W$ is atom-commutative  if $[H(r_i),H(r_{i+1})]=1$ for all $i=1,\dots,k-1$.

Suppose that the quadratic equation satisfies one of the following two alternatives:
\begin{itemize}
\item the Euler characteristic of $W$ is at most $-2$, or $W$ corresponds to a punctured torus and the subgroup 
$$
\langle \varphi_0(\eta) , \varphi_0(h_i^{(0)}) \mid \eta \in \caN, h_i^{(0)} \hbox{ is covered by a quadratic-coefficient base} \rangle
$$
of $\Ts_1$ is non-abelian, i.e. the retraction of the (punctured) surface onto $\Ts_1$ is non-abelian;
\item the rank $k$ of $W$ is greater than or equal to $2$ and the minimal solution $\varphi_0$ is not atom-commutative.
\end{itemize}

Note that, in the case of free groups, the above conditions are sufficient to ensure that the radical of the quadratic equation $W=1$ coincides with the normal closure, \cite{KhMNull, Sela1}. We will see in the next section that the same result holds for arbitrary partially commutative groups (under the condition that the set of solutions of the quadratic equation factors through an infinite branch).

Notice that, by Lemma \ref{lem:typesKperp}, we have that $\KK^\perp$ is either $E_d(\Gamma)$-directly indecomposable  or $E_c(\Gamma)$-free abelian. If the equation $W$ satisfies one of the above alternatives, then $\KK^\perp$ cannot be free abelian and hence it is directly indecomposable.

Define the graph $\Gamma_0$ as follows: 
\begin{itemize}
\item $V(\Gamma_0)=V(\Gamma_1) \cup \{x_1^{(0)},\dots, x_n^{(0)}\}$; 
\item $E_c(\Gamma_0)=E_c(\Gamma_1)$ and 
\item $E_d(\Gamma_0)=E_d(\Gamma_1)\cup \{ (x_i^{(0)}, x_j^{(1)}) \mid \hbox{ for all } x_j^{(1)} \in \KK, i=1,\dots,n\}$, 
\end{itemize}
where $n=|\caN|$. Then, the group $\HH_0=\GG(\Gamma_0)$ is given by the presentation 
$$
\langle 
\HH_1, x_1^{(0)},\dots, x_n^{(0)} \mid [x_i^{(0)}, x_j^{(1)}]=1,  \hbox{ for all } x_j^{(1)} \in \KK, i=1,\dots,n
\rangle. 
$$

Define the map $\tau_0'$ as follows
$$
\left\{ 
\begin{array}{llll}
\eta_i             & \to & x_i^{(0)}, & \eta_i \in \caN \\
\Delta(\eta_i) & \to & {x_i^{(0)}}^\epsilon, & \eta_i \in  \caN, \eta_i=\Delta(\eta_i)^\epsilon, \epsilon\in \{\pm 1\}, \\
\nu                 &\to & \varphi_0(h(\nu)), & \hbox{for all quadratic-coefficient bases } \nu, \\
h_i^{(0)}          &\to  & \varphi_0(h_i^{(0)}), & \hbox{for all non-active items } h_i^{(0)}.
\end{array}
\right.
$$

If $\Ts_1=\langle y ^{(1)}\mid S_1\rangle$, then we set $\Ts_0$ to be the quotient of $\HH_0$ by the set of relations $S_1$ and the set of relations $S$ consisting of the set of basic relations $[C_{\Ts_1}(\KK^\perp),x_i^{(0)}] = 1$, $1\le i\le n$ and the relation $\tau_0'(W)$, where $W$ is a quadratic word in the normal form from Corollary \ref{cor:quadNF}. Notice that $\tau_0'(W)$ is a quadratic word in the normal form in variables $x_i^{(0)}$, $i=1,\dots,n$.

\begin{lem} \label{lem:qhom}
The map $\tau_0'$ extends via the isomorphism $\psi$ {\rm(}see {\rm Lemma \ref{lem:isoquad})} to a homomorphism $\tau_0$ from $G_{\Omega_0}$ to $\Ts_0$.
\end{lem} 
\begin{proof}
By Corollary \ref{cor:quadNF}, the group $G_{\Omega_0}$ is isomorphic to the group $\mathcal K$. We show that $\tau_0'$ induces a homomorphism from $\mathcal K$ to $\Ts_0$.

It is immediate to check that if $r$ is either the relation $W=1$ or a relation of the type $\nu^{\varepsilon(\nu)}=\Delta(\nu)^{\varepsilon(\Delta(\nu))}$ or a relation from $R'$, then $\tau_0'(r)=1$. Therefore, we are left to show that $\tau_0'$ is trivial on the set of commutators $\Xi$.

Assume that $([\nu_i, \nu_j]=1)\in \Xi$ where $\nu_i,\nu_j$ are short bases. Since bases from the quadratic part belong to tribes that dominate the minimal tribe, it follows that $\nu_i$ and $\nu_j$ belong to a tribe that strictly dominates the minimal tribe and hence, since we assume that $\Omega_0=\Omega_0'$, the bases $\nu_i$ and $\nu_j$ are quadratic-coefficient bases. Since $\tau_0'([\nu_i, \nu_j])=\varphi_0([\nu_i, \nu_j])$ and $\varphi_0$ is a homomorphism, we have that $\tau_0'([\nu_i, \nu_j])=1$.

Similarly, if $([\nu_i, h_j^{(0)}]=1)\in \Xi$, where $\nu_i$ is a quadratic-coefficient base and $h_j^{(0)}$ is a non-active item, then $\tau_0'([\nu_i, h_j^{(0)}])=\varphi_0([\nu_i, h_j^{(0)}])=1$.

If $([h_i^{(0)},h_j^{(0)}]=1) \in \Xi$, where $h_i^{(0)}$ and $h_j^{(0)}$ belong to the non-active part, then $\tau_0'([h_i^{(0)}, h_j^{(0)}])=\varphi_0([h_i^{(0)}, h_j^{(0)}])=1$.

Assume that $([\nu_i, h_j^{(0)}]=1)\in \Xi$, where $\nu_i$ is a quadratic base and $h_j^{(0)}$ is an item from the non-active part. By the definition of $\Xi$ and the fact that $\Re_{\Upsilon_0}$ is completed, one has that $\Re_{\Upsilon_0}(h_i^{(0)},h_j^{(0)})$ for all $h_i^{(0)}$ from the quadratic part of $\Omega_0$. Hence, from the description of the process, we conclude that $\Re_{\Upsilon_1}(h_i^{(1)},h_j^{(1)})$ for all $h_i^{(1)}$ from the word $\pi(v_0,v_1)(h_i^{(0)})$ and all $h_j^{(1)}$ from the word $\pi(v_0,v_1)(h_j^{(0)})$. By definition of a solution of a generalised equation, it follows that for any solution $H^{(1)}$ of $\Omega_1$, we have that $H_i^{(1)} \lra H_j^{(1)}$. By the induction hypothesis IH2 on $\tau_1$, for all $y_i$ from the word $\tau_1(h_i^{(1)})$ and all $y_j$ from the word $\tau_1(h_j^{(1)})$, we have that $H'(y_i) \lra H'(y_j)$. By the induction hypothesis IH1 on $\HH_1$, it follows that $(x_i,x_j)\in E_d(\Gamma_1)$. Therefore, for any $x_j$ from the word $\varphi_0(h_j^{(0)})$ and any $x_i$ from the word $\varphi_0(h_i^{(0)})$, where $h_i^{(0)}$ belongs to the quadratic part, we have that $(x_i,x_j)\in E_d(\Gamma_1)$ and so $x_j \in \BA(h^{(0)}(\eta))$ for every base $\eta$ from the quadratic part. Thus, $\varphi_0(h_j^{(0)}) \in \KK$. We conclude that $\tau_0([\nu_i, h_j^{(0)}])=1$. This shows that $\tau_0$ is a homomorphism.
\end{proof}

\begin{lem}\label{lem:qdiagcom}
The homomorphism $\tau_0:G_{\Omega_0} \to \Ts_0$ makes Diagram {\rm(\ref{diag:period})} commutative.
\end{lem}
\begin{proof}
Notice that any solution of the fundamental sequence is the composition of an automorphism associated to the vertex $v_0$, the epimorphism $\pi(v_0,v_1)$ and a solution of $\Omega_1$. Since by \cite[Lemma 7.12]{CKpc}, all the automorphisms fix the subgroup $NQ$ of $G_{\Omega_0}$ generated by the items covered by quadratic-coefficient bases and the non-active items, it follows that the restriction of any solution $H^{(0)}$ of the fundamental sequence to this subgroup coincides with a solution $H^{(1)}$ on the subgroup $\pi(v_0,v_1)(NQ)$ of $G_{\Omega_1}$, i.e. for all items $h_i^{(0)}\in NQ$ we have that $H^{(0)}_i=H^{(0)}(h_i^{(0)})=H^{(1)}(\pi(v_0,v_1)(h_i^{(0)}))$. By induction hypothesis, there exists a homomorphism ${H^{(1)}}'$ from $\Ts_1$ that makes Diagram  (\ref{diag:period}) commutative. We define ${H^{(0)}}'$ on the subgroup $\Ts_1$ of $\Ts_0$ to be ${H^{(1)}}'$ and set ${H^{(0)}}'(x_i^{(0)})=H^{(0)}(\eta_i)$, where $\eta_i \in \caN$. 

It suffices to show that ${H^{(0)}}'$ is a homomorphism. In this case, the commutativity of the diagram follows by construction.  Let us show that ${H^{(0)}}'([ x_i^{(0)}, C_{\Ts_1}(\KK^\perp)])=1$. From the definition of $\KK$ and the induction hypothesis on the tower $(\Ts_1,\HH_1)$, we have that 
$$
{H^{(1)}}'(\varphi_0(\eta))=H^{(1)}(\pi(v_0,v_1)(\eta)) \lra {H^{(1)}}'(\KK). 
$$
Since every solution $H^{(0)}$ from the fundamental sequence is obtained from a solution $H^{(1)}$ from the fundamental sequence by precomposing it with a canonical automorphism (and $\pi(v_0,v_1)$), we conclude that $a_i \lra  \{H^{(0)}(\eta), \eta \in \caN\}$ for all solutions $H^{(0)}$ of the fundamental sequence if and only if $a_i \lra \{H^{(1)}(\pi(v_0,v_1)(h(\eta))), \eta \in \caN\}$ for all $H^{(1)}$ from the fundamental sequence, where $a_i\in \GG$.  It follows that 
$$
{H^{(0)}}(\eta_i)={H^{(0)}}'(x_i^{(0)}) \lra {H^{(1)}}'(\KK)={H^{(0)}}'(\KK)\hbox{ and so } {H^{(0)}}'(x_i^{(0)}) \in \langle \az({H^{(0)}}'(\KK))^\perp\rangle.
$$
By induction hypothesis, $\langle\az({H^{(1)}}'(\KK))^\perp\rangle=\langle\az({H^{(1)}}'(\KK^\perp))\rangle$ and hence $[{H^{(0)}}'(x_i^{(0)}), C_\GG({H^{(0)}}'(\KK^\perp))]=1$.

Finally, direct computation shows that ${H^{(0)}}'(W)=H^{(0)}(W)=1$. This proves that ${H^{(0)}}'$ is a homomorphism.
\end{proof}

\begin{lem}\label{lem:qH0ind}
The group $\HH_0$ satisfies the induction hypothesis {\rm IH1}: for every $x_i,x_j \in \HH_0$, we have that $(x_i,x_j) \in E_d(\Gamma_0)$ if and only if ${H^{(0)}}'(x_i) \lra {H^{(0)}}'(x_j)$ for all homomorphisms ${H^{(0)}}'$ induced by a solution $H^{(0)}$ from the fundamental sequence. Furthermore, we have that if $(x_i,x_j) \in E_c(\Gamma_0)$, then ${H^{(0)}}'(x_i) ,{H^{(0)}}'(x_j)$ belong to a cyclic subgroup for all homomorphisms ${H^{(0)}}'$ induced by a solution $H^{(0)}$ from the fundamental sequence.
\end{lem}
\begin{proof}
If $x_i,x_j \in \HH_1<\HH_0$, then the statement follows by the induction hypothesis IH1 for $\HH_1$.

By definition $(x_i^{(0)},x_j^{(0)}) \notin E_d(\Gamma_0)$, for all $1\le i<j\le n$. Let us show that ${H^{(0)}}'(x_i^{(0)}) \not\lra {H^{(0)}}'(x_j^{(0)})$. Notice that by definition, ${H^{(0)}}'(x_i^{(0)})=H(\eta_i)$, where $\eta_i \in \caN$. Since all the bases from $\caN$ belong to the same tribe, it follows that $H^{(0)}(\eta_i) \not\lra H^{(0)}(\eta_j)$.

By definition $(x_i^{(0)}, x_j^{(1)})\in E_d(\Gamma_0)$, for all $x_j^{(1)}\in \KK$, $i=1,\dots,n$. From the definition of $\KK$ and the induction hypothesis IH1 for the graph tower $(\Ts_1,\HH_1)$, we have that ${H^{(1)}}'(\varphi_0(\eta))=H^{(1)}(\pi(v_0,v_1)(\eta)) \lra {H^{(1)}}'(\KK)$. As we have already seen in Lemma \ref{lem:qdiagcom}, we have that $a_i \lra  \{H^{(0)}(\eta), \eta \in \caN\}$ for all solutions $H^{(0)}$ of the fundamental sequence if and only if $a_i \lra \{H^{(1)}(\pi(v_0,v_1)(h(\eta))), \eta \in \caN\}$ for all solutions $H^{(1)}$ from the fundamental sequence, where $a_i\in \GG$.

Assume that for $x_j^{(1)}\in \HH_1$ we have that ${H^{(0)}}'(x_j^{(1)}) \lra {H^{(0)}}'(x_i^{(0)})$ for some $i=1,\dots,n$. Since by definition ${H^{(0)}}'(x_i^{(0)})=H^{(0)}(\eta_i)$ and all the quadratic bases belong to the same minimal tribe, it follows that ${H^{(0)}}'(x_j^{(1)}) \lra {H^{(0)}}'(x_i^{(0)})$ for all $i=1,\dots,n$. Since the sets $\{\pi(v_0,v_1)(h(\eta)), \eta \in \caN\}$ and $\{ h(\eta), \eta \in \caN\}$ define the same minimal tribe, we conclude that 
$$
{H^{(0)}}'(x_j^{(1)})={H^{(1)}}'(x_j^{(1)}) \lra H^{(0)}(\pi(v_0,v_1)(h(\eta_i)))={H^{(0)}}'(\varphi_0(h(\eta_i)))={H^{(1)}}'(\varphi(h(\eta_i))) 
$$
for all $i=1,\dots,n$. By the induction hypothesis IH1 on $\HH_1$, it follows that $x_j^{(1)} \in \BA(\varphi(h(\eta_i)))=\KK$. Therefore, $(x_j^{(1)},x_i^{(0)})\in E_d(\Gamma_0)$.

Finally, since all the edges $(x_i,x_j)$ from $E_c(\Gamma_0)$ belong to $E_c(\Gamma_1)$, by induction on $\HH_1$ and definition of ${H^{(0)}}'$, we conclude that for every solution $H^{(0)}$ from the fundamental sequence, ${H^{(0)}}'(x_i)={H^{(1)}}'(x_i)$ and ${H^{(0)}}'(x_j)={H^{(1)}}'(x_j)$ belong to the same cyclic subgroup.
\end{proof}

\begin{lem}\label{lem:qtau0ind}
The homomorphism $\tau_0$ satisfies the induction hypothesis {\rm IH2}: for all $h_i^{(0)}$, if $\tau_0(h_i^{(0)})=y_{i1}\dots y_{ik}$, then for a fundamental sequence of solutions we have that $\BA({H^{(0)}}'(y_{ij})) \supset \BA(H_i^{(0)})$ and $\bigcap \limits_{j=1}^k \BA({H^{(0)}}'(y_{ij}))= \BA(H_i^{(0)})$.
\end{lem}
\begin{proof}
If the item $h_i^{(0)}$ is covered by a quadratic-coefficient base or it is non-active,  then $\tau_0(h_i^{(0)})=\tau_1 \pi(v_0,v_1) (h_i^{(0)})$. By construction, the epimorphism $\pi(v_0,v_1)$ satisfies the statement of the lemma, and by induction hypothesis so does $\tau_1$. It follows that the statement of the lemma holds for $\tau_0(h_i^{(0)})$.

We are left to consider the case when  $h_i^{(0)}$ belongs to a quadratic base and is not covered by a quadratic-coefficient base. In this case, $h_i^{(0)}$ is covered by two long bases and, hence belongs to a minimal tribe. Then, by definition, $\tau_0(h_i^{(0)})=\tau_0'(\psi(h_i^{(0)}))$. If $\psi(h_i^{(0)})=w_i(\caN, C)$ where $C$ is the set of items covered by quadratic-coefficient bases, then $\tau_0(h_i^{(0)})=w_i(x_i^{(0)}, \varphi_0(C))=y_{i1}\dots y_{ik}$. Since ${H^{(0)}}'(\tau_0(h_i^{(0)}))=w_i(H^{(0)}(\eta_i),H^{(0)}(C))$ and all the bases and items of the quadratic part belong to tribes that dominate the minimal tribe, it follows that $\BA({H^{(0)}}'(y_{ij})) \supset \BA(H_i^{(0)})$. Furthermore, since the item $h_i^{(0)}$ belongs to a quadratic base, it follows that $y_{il}=x_i^{(0)}$ for some $l=1,\dots,k$ and some $i=1,\dots,n$. Since $x_i^{(0)}$ belongs to the minimal tribe, we have that $\bigcap \limits_{j=1}^k \BA({H^{(0)}}'(y_{ij}))= \BA(H_i^{(0)})$.
\end{proof}

\bigskip
Let us now deal with the exceptional cases.

\subsubsection*{Rank one} 
Assume that $W$ corresponds to a torus, i.e. $W=[\eta_1,\eta_2]=1$. Define the graph $\Gamma_0$ as follows. If $\KK^\perp$ is non-abelian, set:
\begin{itemize}
\item $V(\Gamma_0)=V(\Gamma_1) \cup \{x_1^{(0)}, x_2^{(0)} \}$,
\item $E_c(\Gamma_0)=E_c(\Gamma_1) \cup \{(x_1^{(0)}, x_2^{(0)})\}$ and 
\item $E_d(\Gamma_0)=E_d(\Gamma_1)\cup \{ (x_i^{(0)}, x_j^{(1)}) \mid \hbox{ for all } x_j^{(1)} \in \KK, i=1,2\}$.
\end{itemize}
If $\KK^\perp$ is abelian, then set:
\begin{itemize}
\item $V(\Gamma_0)=V(\Gamma_1) \cup \{x_1^{(0)}, x_2^{(0)} \}$,
\item $E_c(\Gamma_0)=E_c(\Gamma_1) \cup \{(x_1^{(0)}, x_2^{(0)})\} \cup \{(x_i^{(0)},x_j^{(1)}) \mid \hbox{ for all } x_j^{(1)} \in \KK^\perp, i=1,2\}$, and 
\item $E_d(\Gamma_0)=E_d(\Gamma_1)\cup \{ (x_i^{(0)}, x_j^{(1)}) \mid \hbox{ for all } x_j^{(1)} \in \KK, i=1,2\}$.
\end{itemize} 

If $\Ts_1=\langle y ^{(1)}\mid S_1\rangle$, then we set $\Ts_0$ to be the quotient of $\HH_0$ by the set of relations $S_1$ and the set of basic relations $S$: 
$$
[C_{\Ts_1}(\KK^\perp),x_i^{(0)}] = 1, \ i=1,2.
$$ 

\begin{lem}
The map  $\tau_0':\eta_i\mapsto x_i^{(0)}$, $i=1,2$ extends via the isomorphism $\psi$ {\rm(}see {\rm Lemma \ref{lem:isoquad})} to a homomorphism $\tau_0$ from $G_{\Omega_0}$ to $\Ts_0$ that makes Diagram {\rm(\ref{diag:period})} commutative. Furthermore, $\Ts_0$, $\HH_0$ and $\tau_0$ satisfy the induction hypothesis {\rm IH}, {\rm IH1} and {\rm IH2}.
\end{lem}
\begin{proof}
Proof is analogous to the proofs of Lemmas \ref{lem:qhom}, \ref{lem:qdiagcom}, \ref{lem:qH0ind}, \ref{lem:qtau0ind}.
\end{proof}

Notice that other quadratic equations of rank one do not define an infinite family of homomorphisms. Indeed, by Remark \ref{rem:recall}, any solution of a generalised equation is induced by a solution of the corresponding generalised equation over the free group. Now the claim follows from the description of the radical ideal for quadratic equations in the free group, see \cite{KhMIrc}.

\bigskip

\subsubsection*{Other cases} Finally, let us consider the remaining cases. Notice that if the image of the surface under the retraction is commutative, then the minimal solution is atom-commutative. Hence, without loss of generality, we assume that the minimal solution is atom-commutative.

Define the graph $\Gamma_0$ as follows. If $\KK^\perp$ is non-abelian, set:
\begin{itemize}
\item $V(\Gamma_0)=V(\Gamma_1) \cup \{x_1^{(0)},\dots,  x_n^{(0)} \}$; 
\item $E_c(\Gamma_0)=E_c(\Gamma_1) \cup \{(x_i^{(0)}, x_j^{(0)})\mid 1\le i<j\le n\}$ and 
\item $E_d(\Gamma_0)=E_d(\Gamma_1)\cup \{ (x_i^{(0)}, x_j^{(1)}) \mid \hbox{ for all } x_j^{(1)} \in \KK, i=1,\dots,n\}$.
\end{itemize}
If $\KK^\perp$ is abelian, then set
\begin{itemize}
\item $V(\Gamma_0)=V(\Gamma_1) \cup \{x_1^{(0)},\dots, x_n^{(0)}\}$; 
\item $E_c(\Gamma_0)=E_c(\Gamma_1) \cup \{(x_i^{(0)}, x_j^{(0)})\mid 1\le i<j\le n \} \cup \{(x_i^{(0)},x_j^{(1)}) \mid x_j^{(1)} \in \KK^\perp, i=1,\dots, n\}$;
\item $E_d(\Gamma_0)=E_d(\Gamma_1)\cup \{ (x_i^{(0)}, x_j^{(1)}) \mid \hbox{ for all } x_j^{(1)} \in \KK, i=1,\dots,n\}$.
\end{itemize}

If $\Ts_1=\langle y ^{(1)}\mid S_1\rangle$, then we set $\Ts_0$ to be the quotient of $\HH_0$ by the set of relations $S_1$ and the set of relations $S$ defined as follows:
\begin{itemize}
\item if one of the elements $\varphi_0(\eta_1\dots \eta_{2g}), \varphi_0(\lambda_{i})$, $i=2g+1,\dots,m$ is non-trivial, then the set of relations $S$ is defined as follows:
$$
S=\{x_1^{(0)}\dots x_{2g}^{(0)} = \varphi_0(\eta_1\dots \eta_{2g}),\ [x_i^{(0)},C]= 1, i=1,\dots,n\}, 
$$
where 
$$
C=C_{\Ts_1}(\varphi_0(\lambda_{2g+1})^{\varphi_0(\eta_{2g+1})}, \dots, \varphi_0(\lambda_m)^{\varphi_0(\eta_{m})},\varphi_0(\eta_1\dots \eta_{2g}) );
$$ 
\item otherwise, $S=\{x_1^{(0)}\dots x_{2g}^{(0)} =1\}\cup \{\hbox{the set of basic relations}\}$.
\end{itemize}

In these cases, define the map $\tau_0'$ as follows:
$$
\left\{ 
\begin{array}{llll}
\eta_j             & \to & x_j^{(0)}, & \eta_j \in \caN, j=1,\dots, 2g \\
\Delta(\eta_j) & \to & {x_j^{(0)}}^\epsilon, & \eta_j \in  \caN, \eta_j=\Delta(\eta_j)^\epsilon, \epsilon=\pm 1, j=1,\dots, 2g \\
\eta_j             & \to & x_j^{(0)}\varphi_0(\eta_j), & \eta_j \in \caN, j=2g+1,\dots,n \\
\Delta(\eta_j) & \to & {x_j^{(0)}\varphi_0(\eta_j)}^\varepsilon, & \eta_j \in  \caN, \eta_j=\Delta(\eta_j)^\epsilon, \epsilon=\pm 1, j=2g+1,\dots,n \\
\nu                 &\to & \varphi_0(h(\nu)), & \hbox{if $\nu$ is a quadratic-coefficient base,} \\
h_i^{(0)}          &\to  & \varphi_0(h_i^{(0)}), & \hbox{if $h_i^{(0)}$ is a non-active item}.
\end{array}
\right.
$$

\begin{lem} 
The map $\tau_0'$ extends via the isomorphism $\psi$ {\rm(}see {\rm Lemma \ref{lem:isoquad})} to a homomorphism $\tau_0$ from $G_{\Omega_0}$ to $\Ts_0$ that makes Diagram {\rm(\ref{diag:period})} commutative.  Furthermore, $\Ts_0$, $\HH_0$ and $\tau_0$ satisfy the induction hypothesis {\rm IH}, {\rm IH1} and {\rm IH2}.
\end{lem} 
\begin{proof}
By Corollary \ref{cor:quadNF}, the group $G_{\Omega_0}$ is isomorphic to the group $\mathcal K$. We show that $\tau_0'$ induces a homomorphism from $\mathcal K$ to $\Ts_0$.

It is immediate to check that $\tau_0'(r)=1$ if $r$ is either the relation $W$, a relation of the type $\nu^{\varepsilon(\nu)}=\Delta(\nu)^{\varepsilon(\Delta(\nu))}$ or a relation from $R'$.

Therefore, we are left to show that $\tau_0'$ is trivial on the set of commutators $\Xi$, see Equation (\ref{eq:cak}). The proof is analogous to the proof of Lemma \ref{lem:qhom}.

Let us now show that Diagram (\ref{diag:period}) is commutative. Any solution of the fundamental sequence is the composition of an automorphism associated to the vertex, the epimorphism $\pi(v_0,v_1)$ and a solution of $\Omega_1$ from the fundamental sequence. Since by \cite[Lemma 7.12]{CKpc} all the automorphisms fix the subgroup $NQ$ of $G_{R(\Omega_0)}$ generated by the items covered by quadratic-coefficient bases and the non-active items, it follows that the restriction of any solution $H^{(0)}$ of the fundamental sequence onto this subgroup coincides with a solution $H^{(1)}$ on the subgroup $\pi(v_0,v_1)(NQ)$ of $G_{\Omega_1}$, i.e. for all items $h_i^{(0)}\in NQ$ we have that $H^{(0)}(h_i^{(0)})=H^{(1)}(\pi(v_0,v_1)(h_i^{(0)}))$. By induction hypothesis, there exists a homomorphism ${H^{(1)}}'$ from $\Ts_1$ that makes Diagram (\ref{diag:period}) commutative. We define ${H^{(0)}}'$ on the subgroup $\Ts_1$ of $\Ts_0$ to be ${H^{(1)}}'$ and we define ${H^{(0)}}'(x_j^{(0)})=H^{(0)}(\eta_j)$, if $\eta_j \in \caN$ and $j=1,\dots, 2g$, and ${H^{(0)}}'(x_j^{(0)})=H^{(0)}(\eta_j)H^{(0)}(\varphi_0(\eta_j))^{-1}$, if $\eta_j \in \caN$ and $j=2g+1,\dots, n$. 

It suffices to show that ${H^{(0)}}'$ is a homomorphism. In this case, the commutation of Diagram (\ref{diag:period}) follows by construction. Since, by assumption, the minimal solution $\varphi_0(\eta_i)$ is atom-commutative and since any solution of the fundamental sequence is obtained from the minimal solution by precomposing it with an automorphism, it follows that any solution of the fundamental sequence is also atom-commutative. Furthermore, notice that by Remark \ref{rem:blocks}, the image of any item under the homomorphism $H^{(0)}$ is a block element such that $\BA(H^{(0)}(\eta_i))=\BA(\GG_{\KK^\perp})$, for all $i=1,\dots,n$. By Lemma \ref{lem:eired}, there exists an element $U$ so that $H^{(0)}(U)$ is irreducible. Hence, the images of all the atoms under $H^{(0)}$ belong to the same cyclic subgroup: 
$$
\langle \sqrt{H^{(0)}(\lambda_{2g+1}^{\eta_{2g+1}})}\rangle= \dots =\langle \sqrt{H^{(0)}(\lambda_n^{\eta_{n}})} \rangle=\langle \sqrt{H^{(0)}(\eta_1\dots \eta_{2g})}\rangle.
$$
We conclude that $H^{(0)}([x_i^{(0)},C])=1$. Therefore, Diagram (\ref{diag:period}) is commutative.

The proofs that the induction hypothesis IH1 and IH2 hold, are analogous to the proofs of Lemmas \ref{lem:qH0ind} and \ref{lem:qtau0ind}, correspondingly.
\end{proof}

Notice that since $H^{(0)}(x_1)$ is cyclically reduced, the argument given above shows that if the equation is non-orientable (there is an atom $x_1^2$), then $H^{(0)}(c^{x_i})$ is cyclically reduced. Hence, we conclude that in this case, the retraction of the surface is in fact abelian.

We summarize the results of this section in the proposition below.

\begin{prop} \label{prop:quad}
Let $\Omega=\langle \Upsilon, \Re_\Upsilon\rangle$ be a generalised equation of type $12$ that repeats infinitely many times in the infinite branch, let $h_1$ be minimal and $\Re_\Upsilon$ be completed. Then, there exists a graph tower $(\Ts, \HH)$ and a homomorphism $\tau$ from $G_{\Omega}$ to $\Ts$ such that for all solutions of the fundamental sequence, there exists a homomorphism from $\Ts$ to $\GG$ that makes Diagram {\rm(\ref{diag:period})} commutative.
\end{prop}

\subsection{Singular and strongly singular periodic structures}

In this and the following sections we use the notation and results from \cite[Chapter 6]{CKpc}, of which we now give a very brief and informal summary.

Informally, a periodic structure $\mathcal P$ is a set of items, bases and sections of a generalised equation $\Omega$ that imposes restrictions on the set of solutions of $\Omega$, namely that solutions are $P$-periodic and the image of items, bases and section from the periodic structure are ``long'' (they are subwords of $P^{n}$ of length greater than $2|P|$). 

Given a periodic structure on $\Omega$, we study the corresponding coordinate group by introducing a new set of generators and exhibiting the corresponding presentation.  In order to define the new set of generators, one constructs a graph $\Gamma=\Gamma(\mathcal P)$, whose edges are labelled by items that belong to a section of the periodic structure. Without loss of generality, one can assume that $\Gamma$ is connected. One then chooses a maximal subforest of $\Gamma$ so that the edges are labelled by items that belong to a section from $\P$, but the items themselves do not belong to $\P$. We then complete the forest to a maximal subtree $T$ of $\Gamma$. The free group generated by the items that label edges of $\Gamma$ is generated by the items that label edges of $T$ and cycles $\cc_e$, where $e\in \Gamma\smallsetminus T$.

Under the assumption that $\Omega$ is periodised with respect to $\P$ (i.e. words defined by cycles of $\Gamma$ based at a given point commute), one can  choose a basis $C^{(1)}\cup C^{(2)}$ of the subgroup generated by the cycles $\cc_e$.  For each edge $e_i\in T$ labelled by an item that belongs to $\P$, we consider two families of cycles $u_{i,e}$, $z_{i,e}$ defined by edges $e\notin T$ labelled by items that do not belong to $\P$. The cycles $u_{i,e}$, $z_{i,e}$ are based at the origin and the terminus of $e_i$, correspondingly. We prove that the set
\begin{itemize}
\item $\bar t$ of items that do not belong to sections from $\P$;
\item $\{h(e), e\in T, h(e)\notin \P\}$;
\item $\{h(e_1),\dots, h(e_m), e_i \in T, h(e_i)\in \mathcal P\}$;
\item $\{u_{ie}, z_{ie}, i=1,\dots, m, e\notin T, h(e)\notin \mathcal P\}$;
\item $C^{(1)}, C^{(2)}$;
\end{itemize}
is a generating set of the  group $G_{\Omega}$ and in this generating set, the system of equations $\Upsilon$ is equivalent to the union of the following two systems of equations:
            $$
                \left\{
                \begin{array}{ll}
                u_{ie}^{h(e_i)}=z_{ie},&\hbox{where } e\in T,\, e\in \Sh; 1\leq i\leq m \\
                \left[u_{ie_1},u_{ie_2}\right]=1,&\hbox{where } e_j\in T, \,e_j\in \Sh, \, j=1,2; 1\leq i\leq m \\
                \left[h(\cc_1),h(\cc_2)\right]=1, & \hbox{where } \cc_1,\cc_2\in C^{(1)}\cup C^{(2)}
                \end{array}
                \right.
            $$
        and a system:
            $$
            \Psi \left( \{h(e)\,\mid\, e\in T, e\in \Sh\}, h(C^{(1)}), \bar t,\bar u,\bar z,\cA\right)=1,
            $$
        such that neither $h(e_i)$, $1\le i\le m$, nor $h(C^{(2)})$ occurs in $\Psi$.

\begin{defn} \label{defn:singreg}
Let $\Omega$ be a generalised equation and let $\langle \P,R\rangle$ be a connected periodic structure on $\Omega$. We say that the generalised equation $\Omega$ is \emph{strongly singular with respect to the periodic structure $\langle \P,R\rangle$} if one of the following conditions holds.
\begin{itemize}
\item[(a)] The generalised equation $\Omega$ is not periodised with respect to the periodic structure $\langle \P,R\rangle$.
\item[(b)] The generalised equation $\Omega$ is periodised with respect to the periodic structure $\langle \P,R\rangle$ and there exists an automorphism $\varphi$ of the coordinate group $G_{R(\Upsilon)}$ of the form described in parts  \cite[Lemma 6.14(2) and (3)]{CKpc}, such that $\varphi$ does not induce an automorphism of $G_{R(\Omega)}$.
\end{itemize}

We say that the generalised equation $\Omega$ is \emph{singular with respect to the periodic structure $\langle \P,R\rangle$} if $\Omega$ is not strongly singular with respect to the periodic structure $\langle \P,R\rangle$ and one of the following conditions holds
\begin{itemize}
\item[(a)] The set $C^{(2)}$ has more than one element.
\item[(b)] The set $C^{(2)}$ has exactly one element, and (in the above notation) there exists a cycle $\cc_{e_0}\in \langle C^{(1)}\rangle$, $h(e_0)\notin \P$ such that $h(\cc_{e_0})\ne 1$ in $G_{R(\Omega)}$.
\end{itemize}

Otherwise, we say that $\Omega$ is \emph{regular with respect to the periodic structure $\langle \P,R\rangle$}. In particular, if $\Omega$ is singular or regular with respect to the periodic structure  $\langle \P,R\rangle$, then $\Omega$ is periodised.

When no confusion arises, instead of saying that $\Omega$ is (strongly) singular (or regular) with respect to the periodic structure $\langle \P,R\rangle$, we say that the periodic structure $\langle \P,R\rangle$ is (\emph{strongly})\emph{singular} (or \emph{regular}).
\end{defn}

Suppose that the fundamental sequence goes through a vertex $v_0$ of type $2$, i.e. solutions of the discriminating family are $P$-periodic and $\Omega_{v_0}$ is either singular or strongly singular with respect to the corresponding periodic structure $\langle\mathcal{P}, R\rangle$. 

To simplify the notation, set $\Omega_i=\Omega_{v_i}=\langle \Upsilon_i, \Re_{\Upsilon_i}\rangle$. Recall that $G_{\Omega_i}=\langle h^{(i)}\rangle$, $i=1,2$, $\HH_1=\langle x^{(1)}\rangle$, and $\Ts_1=\langle y^{(1)}\rangle$.
 
If the periodic structure $\langle\mathcal{P}, R\rangle$ is strongly singular, then we set $\Ts_0=\Ts_1$ and $\tau_0$ to be $\pi(v_0,v_1) \tau_1$, where $\tau_1$ is the homomorphism from $G_{\Omega_1}$ to $\Ts_1$ that makes Diagram (\ref{diag:period}) commutative. By \cite[Lemma 6.17]{CKpc}, any homomorphism of the fundamental sequence factors through $G_{R(\Omega_1)}$ and so $\tau_0$, $\Ts_0$ and $\HH_0$ make Diagram (\ref{diag:period}) commutative and satisfy the induction hypothesis.

Let us now assume that the  periodic structure $\langle\mathcal{P}, R\rangle$  is singular.

The set of elements
\begin{equation} \label{2.52}
\{h^{(0)}(e) \mid e \in T \} \cup \{h^{(0)}(\cc_e) \mid e \not \in T \}
\end{equation}
forms a basis of the free group $F(h^{(0)})$ generated by
$$
 \{h_k^{(0)} \mid h_k^{(0)}\in \sigma, \sigma\in {\P} \}.
$$

Suppose first that the periodic structure is singular of type (a), i.e. the set $C^{(2)}$ has more than one element. Denote by $\varphi_0$ the composition of $\pi(v_0,v_1)$ and $\tau_1$. 

Since the periodic structure is singular of type (a), the cardinality $n$ of the set $C^{(2)}$ is greater than or equal to $2$ and there exists an element $\cc\in C^{(2)}$ so that the image $\pi(v_0,v_1)(h^{(0)}(\cc))$ in $G_{R(\Omega_1)}$ is non-trivial. Let $w(y^{(1)})=w(\varphi_0(h^{(0)}(\cc)))$ be the image of $h^{(0)}(\cc)$ in $\Ts_1$, and consider the word $w(x^{(1)})$ in $\HH_1$. Set $\KK$ to be $\BA_{\HH_1}(w(x^{(1)}))$.

Let $x^{(0)}=\{x_1^{(0)}, \dots, x_n^{(0)} \}$, where $n=|C^{(2)}|$. Define the graph $\Gamma_0$ as follows. If the canonical parabolic subgroup $\KK^\perp$ is non-abelian, set:
\begin{itemize}
\item $V(\Gamma_0)=V(\Gamma_1) \cup \{x_1^{(0)},\dots, x_n^{(0)} \}$; 
\item $E_c(\Gamma_0)=E_c(\Gamma_1)\cup \{ (x_i^{(0)},x_j^{(0)})\mid 1\le i<j\le n)\}$ and 
\item $E_d(\Gamma_0)=E_d(\Gamma_1)\cup \{ (x_i^{(0)}, x_j^{(1)}) \mid \hbox{ for all } x_j^{(1)} \in \KK, i=1,\dots,n\}$.
\end{itemize} 
If the canonical parabolic subgroup $\KK^\perp$ is abelian, set:
\begin{itemize}
\item $V(\Gamma_0)=V(\Gamma_1) \cup \{x_1^{(0)},\dots, x_n^{(0)} \}$; 
\item $E_c(\Gamma_0)=E_c(\Gamma_1)\cup \{ (x_i^{(0)},x_j^{(0)})\mid 1\le i<j\le n\} \cup \{(x_i^{(0)},x_j^{(1)})\mid x_j^{(1)}\in \KK^\perp, 1\le i\le n\}$ and 
\item $E_d(\Gamma_0)=E_d(\Gamma_1)\cup \{ (x_i^{(0)}, x_j^{(1)}) \mid \hbox{ for all } x_j^{(1)} \in \KK, i=1,\dots,n\}$.
\end{itemize} 

Then, the group $\HH_0$ is defined to be $\GG(\Gamma_0)$. Let $S_0'$ be the following set of relations: 
$$
[x_i^{(0)}, C_{\Ts_1}(\varphi_0(h^{(0)}(\cc)))]=1, \hbox{ for all } 1\le i\le n
$$
and set $\Ts_0$ to be the quotient $\factor{\HH_0}{S_0}$, where $S_0=S_1\cup S_0'$.

Recall that, by definition of the generating set $\bar x$ associated to the periodic structure and since the generalised equation is periodised (the periodic structure being singular), the set $\{h(e) \mid e\in T\} \cup \{h^{(0)}(C^{(1)}), h^{(0)}(C^{(2)})\}$ is a basis of the free group $F(\{h_k^{(0)} \in \sigma, \sigma \in \mathcal P\})$. Furthermore, any $h_k^{(0)}\in \sigma, \sigma \in \mathcal P$ such that $h_k^{(0)}=h(e)$, $e\notin T$, $e:v\to v'$, is expressed in the generating set $\{h(e) \mid e\in T\} \cup \{h^{(0)}(C^{(1)}), h^{(0)}(C^{(2)})\}$ as 
$$
h_k^{(0)}=h^{(0)}(\p(v_0,v))^{-1} v_{k1}(C^{(1)}) v_{k2}(C^{(2)})h^{(0)}(\p(v_0, v')), 
$$
where 
$$
v_{k1}(C^{(1)})\in \langle h^{(0)}(C^{(1)})\rangle, \ v_{k2}(C^{(2)}) \in \langle h^{(0)}(C^{(2)})\rangle \hbox{ and } \p(v_0,v),\p(v_0, v')\in \langle h(e) \mid e\in T \rangle.
$$

\begin{lem} \label{lem:singatauhom}
The map $\tau_0$, 
$$
h_k^{(0)}\mapsto 
\left\{ 
\begin{array}{lll}
\varphi_0(h_k^{(0)}),& \hbox{ for all $h_k^{(0)} \in \sigma, \sigma \notin \mathcal P$,} \\
\varphi_0(h_k^{(0)}),& \hbox{ for all $h_k^{(0)}=h(e), e\in T$,} \\
\varphi_0(h^{(0)}(\p(v_0, v))^{-1}) \varphi_0(v_{k1}(C^{(1)})) v_{k2}(x^{(0)}) \varphi_0(h^{(0)}(\p(v_0,v'))), & \hbox{ for all $h_k^{(0)}=h(e), e \notin T$,}
\end{array}
\right.
$$
extends to a homomorphism from $G_{\Omega_0}$ to $\Ts_0$.
\end{lem} 
\begin{proof}
By \cite[Lemma 6.14]{CKpc}, the set 
\begin{equation}\label{eq:psgen}
\{h_k^{(0)}\mid h_k^{(0)} \in \sigma, \sigma \notin \mathcal P\} \cup \{h_k^{(0)} \mid h_k^{(0)} \in T \} \cup h^{(0)}(C^{(1)} \cup C^{(2)})
\end{equation}
is a generating set of $G_{\Omega_0}$.

Notice that, by definition, the map $\tau_0$ on the set 
$$
S=\{h_k^{(0)}\mid h_k^{(0)} \in \sigma, \sigma \notin \mathcal P\} \cup \{h_k^{(0)} \mid h_k^{(0)} \in T \} \cup h^{(0)}(C^{(1)}), 
$$
coincides with the homomorphism $\varphi_0$. It follows that $\tau_0$ extends to a homomorphism on the subgroup $\langle S \rangle$ of $G_{\Omega_0}$.

By \cite[Lemma 6.14]{CKpc}, since the generalised equation is periodised, i.e. $[h^{(0)}(\cc_1), h^{(0)}(\cc_2)]=1$ for all cycles $\cc_1, \cc_2 \in C^{(1)}\cup C^{(2)}$, it follows that all the relations  $h^{(0)}(\mu)^{-\varepsilon(\mu)} h^{(0)}(\Delta(\mu))^{\varepsilon(\Delta(\mu))}$ belong to the subgroup generated by $S$ and thus $\tau_0(h^{(0)}(\mu)^{-\varepsilon(\mu)} h^{(0)}(\Delta(\mu))^{\varepsilon(\Delta(\mu))})=1$. 

Therefore, in order to show that the map $\tau_0$ induces a homomorphism from $G_{\Upsilon_0}$ to $\Ts_0$ we need to show that $\tau_0([h^{(0)}(\cc_1),h^{(0)}(\cc_2)])=1$ in $\Ts_0$ for all $\cc_1,\cc_2 \in C^{(1)}\cup C^{(2)}$. If $\cc_1,\cc_2\in C^{(1)}$, then 
$$
\tau_0([h^{(0)}(\cc_1),h^{(0)}(\cc_2)])=\varphi_0([h^{(0)}(\cc_1),h^{(0)}(\cc_2)])=1. 
$$
If $\cc_1,\cc_2 \in C^{(2)}$, then, by definition of $\Gamma_0$, 
$$
\tau_0([h^{(0)}(\cc_1),h^{(0)}(\cc_2)])=[x_i^{(0)},x_j^{(0)}]=1.
$$
Finally, assume that $\cc_1\in C^{(1)}$ and $\cc_2\in C^{(2)}$. Since $[h^{(0)}(\cc_1),h^{(0)}(\cc)]=1$ in $G_{\Omega_0}$, for all $\cc_1\in C^{(1)}$, and the non-trivial cycle $\cc\in C^{(2)}$ that defines the subgroup $\KK$, we have that $[\varphi_0(h^{(0)}(\cc_1)), \varphi_0(h^{(0)}(\cc))]=1$ in $\Ts_1$ and so $\varphi_0(h^{(0)}(\cc_1)) \in C_{\Ts_1}(\varphi_0(h^{(0)}(\cc))$. We conclude from the definition of $S_0$ that $\tau_0([h^{(0)}(\cc_1),h^{(0)}(\cc_2)])=1$ in $\Ts_0$.

To prove that $\tau_0$ extends to a homomorphism from $G_{\Omega_0}$ to $\Ts_0$ it is left to show that $\tau_0([h_i^{(0)}, h_j^{(0)}])=1$, for all $h_i^{(0)}, h_j^{(0)}$ such that $\Re_{\Upsilon_0}(h_i^{(0)}, h_j^{(0)})$. In fact, since the map extends to a homomorphism on the subgroup generated by $S$, we only need to check that $\tau_0([h_i^{(0)}, h_j^{(0)}])=1$, for all $h_j^{(0)}\in \mathcal P$  so that $\Re_{\Upsilon_0}(h_i^{(0)}, h_j^{(0)})$. For every $P$-periodic solution $H^{(0)}$, we have that $\az(H_i^{(0)}) \subset \az(H_j^{(0)})$, for all $h_i^{(0)} \in  \sigma, \sigma \in \mathcal P$, $h_j^{(0)} \in \mathcal P$. Recall that if $\Re_{\Upsilon_0}(h_i^{(0)}, h_j^{(0)})$, then for all solutions $H^{(0)}$ of $\Omega_0$ we have that $H_i^{(0)} \lra H_j^{(0)}$. Therefore, if $\Re_{\Upsilon_0}(h_i^{(0)}, h_j^{(0)})$ and $h_j^{(0)} \in \mathcal P$, then $h_i^{(0)} \in \sigma, \sigma \notin \mathcal P$. 

Furthermore, since the periodic structure is not strongly singular and the set $\Re_{\Upsilon_0}$ is completed, it follows that for all $h_j^{(0)} \in \mathcal P$ and for every $h_i^{(0)}$ such that $\Re_{\Upsilon_0}(h_i^{(0)}, h_j^{(0)})$,  one has that $\Re_{\Upsilon_0}(h_i^{(0)}, h_k^{(0)})$, for all $h_k^{(0)} \in \sigma, \sigma \in \mathcal P$.

By properties of the elementary transformations, we have that $\Re_{\Upsilon_1}(h_r^{(1)}, h_s^{(1)})$, for all $h_r^{(1)}$ from $\pi(v_0,v_1)(h_i^{(0)})$ and all $h_s^{(1)} \in \pi(v_0,v_1)(h_k^{(0)})$, $h_k^{(0)}\in \sigma, \sigma \in \mathcal P$. By the induction hypothesis IH2 on $\tau_1$, for any $y_m^{(1)}$ in $\varphi_0(h_i^{(0)})$, $y_n^{(1)}$ in $\varphi_0(h_k^{(0)})$ and any solution $H^{(1)}$ of $\Omega_1$ from the fundamental sequence,  we have that $\BA({H^{(1)}}'(y_m^{(1)})) \supset \BA(H^{(0)}_i)$ and $\BA({H^{(1)}}'(y_n^{(1)})) \supset \BA(H_k^{(0)})$, hence ${H^{(1)}}'(y_m^{(1)}) \lra {H^{(1)}}'(y_n^{(1)})$. By the induction hypothesis IH1 on $\HH_1$, we have that $(x_n^{(1)},x_m^{(1)})\in E_d(\Gamma_1)$ and so $\varphi_0(h_i^{(0)}) \lra \varphi_0(h_k^{(0)})$ in $\HH_1$, for all $h_k^{(0)}\in \sigma,\sigma \in \mathcal P$ and, hence, in particular, $\varphi_0(h_i^{(0)})$ belongs to the subgroup $\KK$ of $\HH_1$. Therefore, if $\Re_{\Upsilon_0}(h_i^{(0)}, h_j^{(0)})$ and $h_j^{(0)}\in \mathcal P$, then $\tau_0([h_i^{(0)}, h_j^{(0)}])=1$.
\end{proof}

\begin{lem}\label{lem:singadiagcom}
The homomorphism $\tau_0:G_{\Omega_0} \to \Ts_0$ makes Diagram {\rm(\ref{diag:period})} commutative.
\end{lem}
\begin{proof}
By \cite[Lemma 6.18]{CKpc}, for any solution $H^{(0)}$ of the fundamental sequence, there exist a solution $H^{(1)}$ of $\Omega_1$ and an automorphism $\phi$ from the group $\AA(\Omega_0)$ (see \cite[Definition 6.16]{CKpc}) such that $H^{(0)}=\phi \pi(v_0,v_1) H^{(1)}$. Furthermore, if 
$$
S=\{h_k^{(0)}\mid h_k^{(0)} \in \sigma, \sigma \notin \mathcal P\} \cup \{h_k^{(0)} \mid h_k^{(0)} \in T \} \cup h^{(0)}(C^{(1)}), 
$$
then, we have that $H^{(0)}(h_i^{(0)})=H^{(1)}(\pi(v_0,v_1)(h_i^{(0)}))$, where $h_i^{(0)} \in \langle S \rangle$. By the induction hypothesis, there exists a homomorphism ${H^{(1)}}'$ from $\Ts_1$ to $\GG$ that makes Diagram (\ref{diag:period}) commutative. We define ${H^{(0)}}'$ on the subgroup $\Ts_1$ of $\Ts_0$ to be ${H^{(1)}}'$ and we define ${H^{(0)}}'(y_i^{(0)})=H^{(0)}(\cc_{2i})$, where $\cc_{2i} \in C^{(2)}$ is so that $\tau_0(h^{(0)}(\cc_{2i}))=y_i^{(0)}$.

If we show that ${H^{(0)}}'$ is a homomorphism, then the commutativity of Diagram follows by construction. Since solutions of the fundamental sequence are $P$-periodic, it follows that ${H^{(0)}}'(y_i^{(0)})=H^{(0)}(\cc_2)=P^n$ and ${H^{(0)}}'(\varphi_0(h^{(0)}(\cc)))=P^k$, $k\ne 0$. Therefore, ${H^{(0)}}'$ is a homomorphism and makes Diagram (\ref{diag:period}) commutative.
\end{proof}

\begin{lem}\label{lem:singaKcool}
The subgroup $\KK=\BA(w(x^{(1)}))$ of $\HH_1$ is $E_d(\Gamma_1)$-co-irreducible.
\end{lem}
\begin{proof}
By Remark \ref{rem:discfam}, there exists a canonical parabolic subgroup $\GG_\KK< \GG$ such that for a fundamental sequence of solutions we have that $\langle \az({H^{(0)}}'(\KK))\rangle=\GG_\KK$. By definition of $\KK$ and by the induction hypothesis IH1, we have that ${H^{(1)}}'(w(x^{(1)})) \lra \GG_\KK$. Furthermore, if $a_i \lra {H^{(1)}}'(w(x^{(1)}))$, where $a\in \cA$, $\GG=\GG(\cA)$, again, by the induction hypothesis IH1, we have that $a_i\in \KK$ and so $a_i\in \GG_\KK$, hence we have that $\langle \az({H^{(1)}}'(w(x^{(1)})))^\perp\rangle =\GG_\KK$.

Let us now address the statement of the lemma. By definition the group $\KK$ is closed, i.e. $\KK=\KK^{\perp\perp}$. We only need to show that $\KK^\perp$ is $E_d(\Gamma_1)$-directly indecomposable.

Assume the contrary, then $\KK^\perp=\KK_1\times \dots \times \KK_r$, where $r>1$ (with respect to the edges from $E_d(\Gamma_1)$). Without loss of generality, we may assume that $r=2$. By Remark \ref{rem:discfam}, there exist canonical parabolic subgroups $\GG_{\KK_i}<\GG$, $i=1,2$ such that for solutions $H^{(0)}$ from the fundamental sequence we have ${H^{(0)}}'(\KK_i)=\GG_{\KK_i}$, $i=1,2$. Furthermore, by the induction hypothesis IH1, one has that $\GG_{\KK_1} \lra \GG_{\KK_2}$. 

Since $\langle \az({H^{(1)}}'(w(x^{(1)})))^\perp\rangle =\GG_\KK$, it follows that ${H^{(1)}}'(w(x^{(1)}))\in \GG_\KK^\perp < \GG_{\KK_1}\times \GG_{\KK_2}$. Furthermore, it follows that ${H^{(1)}}'(w(x^{(1)}))$ contains at least one $\GG_{\KK_1}\times\GG_{\KK_2}$-alternation.

Since every solution of the fundamental sequence is $P$-periodic, it follows that ${H^{(1)}}'(w(x^{(1)}))=P^n$ and $P$ contains at least one $\GG_{\KK_1}\times \GG_{\KK_2}$-alternation. By the description of solutions of generalised equations periodised with respect to a periodic structure, see \cite[Chapter 6]{CKpc}, for every $m\in \BN$, there exists a solution $H^{(0)}$ of $\Omega_0$ so that $H^{(0)}(\cc_2)\doteq P^{m'}$, where $m'\ge m$, $\cc_2\in C^{(2)}$. Hence, $H^{(0)}(\cc_2)$ contains at least $m'$ many $\GG_{\KK_1}\times \GG_{\KK_2}$-alternations. This derives a contradiction, since, by Remark \ref{rem:recall}, $H^{(0)}(\cc_2)$ is a subword of a word in the normal form and, by Lemma \ref{lem:pc}, it contains only a bounded number of alternations.  Hence, we conclude that $P$ is a cyclically reduced irreducible root element and so $r=1$ and $\KK^\perp$ is $E_d(\Gamma_1)$-directly indecomposable.
\end{proof}

\begin{rem}\label{rem:blockperiod}
Solutions of generalised equations define graphical equalities, so, if solutions of the fundamental sequence are $P$-periodic, then the period $P$ is cyclically reduced. Furthermore, as we have shown above, the period $P$ is an irreducible root  element from $\GG_{\KK^\perp}=\GG_\KK^\perp$.
\end{rem}

\begin{lem}\label{lem:qH0ind1}
The group $\HH_0$ satisfies the induction hypothesis {\rm IH1}: for every $x_i,x_j \in \HH_0$, we have that $(x_i,x_j) \in E_d(\Gamma_0)$ if and only if $H'(x_i) \lra H'(x_j)$ for all homomorphisms $H'$ induced by solutions $H$ from the fundamental sequence; furthermore, we have that if $(x_i,x_j) \in E_c(\Gamma_0)$, then $H'(x_i)$ and $H'(x_j)$ belong to the same cyclic subgroup for all homomorphisms $H'$ induced by solutions $H$ from the fundamental sequence.
\end{lem}
\begin{proof}
If $x_i,x_j \in \HH_1<\HH_0$, then the statement follows by the induction hypothesis on $\HH_1$.

By definition, $(x_i^{(0)},x_j^{(0)}) \in E_c(\Gamma_0)$ for all $1\le i< j\le n$. Let us show that ${H^{(0)}}'(x_i^{(0)})$ and ${H^{(0)}}'(x_j^{(0)})$ belong to the same cyclic subgroup. Indeed, since solutions of the fundamental sequence are $P$-periodic, it follows that ${H^{(0)}}'(x_i^{(0)})=H^{(0)}(\cc_{2i})=P^k$ for all $i=1,\dots,n$.

Furthermore, if $\KK^\perp$ is abelian, then $(x_i^{(0)},x_j^{(1)})\in E_c(\Gamma_0)$ for all $x_j^{(1)}\in \KK^\perp$. By the induction hypothesis IH1, for all solutions the subgroup $\langle{H^{(0)}}'(\KK^\perp)\rangle=\langle{H^{(1)}}'(\KK^\perp)\rangle$ is a cyclic subgroup of $\GG$. Since, by Remark \ref{rem:blockperiod}, the period $P$ is an irreducible element from $\langle \az({H^{(0)}}'(\KK^\perp))\rangle$, 
it follows that ${H^{(0)}}'(x_i^{(0)})=H^{(0)}(\cc_{2i})=P^k$ and ${H^{(0)}}'(x_j^{(1)})$ belong to the same cyclic subgroup.

By definition $(x_i^{(0)}, x_j^{(1)})\in E_d(\Gamma_0)$, for all $x_j^{(1)}\in \KK$, $i=1,\dots,n$. From the definition of $\KK$ and the induction hypothesis on the tower $(\Ts_1,\HH_1)$, we have that ${H^{(1)}}'(\varphi_0(\cc))=H^{(1)}(\pi(v_0,v_1)(\cc))=P^k \lra {H^{(1)}}'(\KK)$. Since solutions are $P$-periodic, it follows that 
$$
{H^{(0)}}'(x_i^{(0)})=H^{(0)}(\cc_{2i})=P^l \lra {H^{(1)}}'(\KK)={H^{(0)}}'(\KK),
$$
where $\tau_0(\cc_{2i})=x_i^{(0)}$, $i=1,\dots, |C^{(2)}|=n$.

Assume that for $x_j^{(1)}\in \HH_1$ we have that ${H^{(0)}}'(x_j^{(1)}) \lra {H^{(0)}}'(x_i^{(0)})$ for some $i=1,\dots,n$. Since ${H^{(0)}}'(x_i^{(0)})=H^{(0)}(\cc_{2i})=P^l$, it follows that 
$$
{H^{(0)}}'(x_j^{(1)})={H^{(1)}}'(x_j^{(1)}) \lra H^{(0)}(\pi(v_0,v_1)(h(\cc)))={H^{(0)}}'(\varphi_0(h(\cc)))={H^{(1)}}'(\varphi(h(\cc)))=P^l. 
$$
By the induction hypothesis IH1 on $\HH_1$, it follows that $x_j^{(1)} \in \BA(\varphi_0(h(\cc)))=\KK$. Therefore, $(x_j^{(1)},x_i^{(0)})\in E_d(\Gamma_0)$.
\end{proof}

\begin{lem}\label{lem:singatau0ind}
The homomorphism $\tau_0$ satisfies the induction hypothesis {\rm IH2}: for all $h_i^{(0)}$, if $\tau_0(h_i^{(0)})=x_{i1}\dots x_{ik}$, then for a fundamental sequence of solutions we have that $\BA(H'(x_{ij})) \supset \BA(H_i^{0})$ and $\bigcap \limits_{j=1}^k \BA(H'(x_{ij}))= \BA(H_i^{(0)})$.
\end{lem}
\begin{proof}
Recall that if 
$$
S=\{h_k^{(0)}\mid h_k^{(0)} \in \sigma, \sigma \notin \mathcal P\} \cup \{h_k^{(0)} \mid h_k^{(0)} \in T \} \cup h^{(0)}(C^{(1)}), 
$$
then, for $h_i^{(0)} \in \langle S \rangle$ and for all solutions of the fundamental sequence $H^{(0)}$ we have that $H^{(0)}(h_i^{(0)})=H^{(1)}(\pi(v_0,v_1)(h_i^{(0)}))$. 

If $h_i^{(0)} \in \langle S \rangle$, then $\tau_0(h_i^{(0)})=\tau_1 \pi(v_0,v_1) (h_i^{(0)})$. If $\pi(v_0,v_1)(h_i^{(0)})=h_{i_1}^{(1)}\dots h_{i_m}^{(1)}$, then since solutions of the generalised equations define graphical equalities it follows that for every solution $H^{(1)}$ of the fundamental sequence, we have that $H^{(1)}(\pi(v_0,v_1)(h_i^{(0)}))= H^{(1)}_{i_1}\dots H^{(1)}_{i_m}$ and so 
$$
\BA(H^{(1)}(h_{i_j}^{(1)})) \supset \BA(H^{(1)}(\pi(v_0,v_1)(h_i^{(0)})))=\BA(H_i^{(0)}). 
$$
Since, by induction assumption, the statement of the lemma holds for $\tau_1$, we conclude that the statement holds for $\tau_0(h_i^{(0)})=\tau_1 \pi(v_0,v_1) (h_i^{(0)})$, for all $h_i^{(0)} \in \langle S \rangle$.

If $h_k^{(0)}\notin \langle S\rangle$, then $h_k^{(0)} \in \mathcal P$ and 
$$
h_k^{(0)}=h^{(0)}(\p(v_0,v))v_{k1}(C^{(1)}) v_{k2}(C^{(2)})h^{(0)}(\p(v_0, v'))
$$
for some 
$$
v_{k1}(C^{(1)})\in \langle h^{(0)}(C^{(1)})\rangle, \ v_{k2}(C^{(2)}) \in \langle h^{(0)}(C^{(2)})\rangle \hbox{ and } h^{(0)}(\p(v_0,v)),h^{(0)}(\p(v_0,v')) \in \langle h(e) \mid e\in T\rangle. 
$$
In this case, 
$$
\tau_0(h_k^{(0)})= \varphi_0\left( h^{(0)}(\p(v_0,v))v_{k1}(C^{(1)})\right) v_{k2}(x^{(0)})\varphi_0(h^{(0)}(\p(v_0, v')).
$$
Since every solution $H^{(0)}$ of the fundamental sequence is $P$-periodic, it follows that $\az(H_i^{(0)}) \subset \az(P)$, for all $h_i^{(0)}\in \sigma, \sigma\in \mathcal P$. Hence, for all $x_i$ in the words $\varphi_0(h^{(0)}(\p(v_0,v))v_{k1}(C^{(1)}))$ and $\varphi_0(h^{(0)}(\p(v_0, v')))$, by the induction hypothesis IH2 on $\tau_1$, we can conclude that $\BA({H^{(0)}}'(x_i))\supset \BA(P)=\BA(H_k^{(0)})$. Furthermore, 
$$
\az({H^{(0)}}'(x_i^{(0)}))= \az(H^{(0)}(\cc_{2i}))=\az(P^k) = \az(H_k^{(0)})
$$ 
and so $\BA( {H^{(0)}}'(x_i^{(0)})) = \BA(H_k^{(0)})$ and, in particular, 
$$
\bigcap\limits_{x_i\ in \ \tau_0(h_k^{(0)})}  \BA({H^{(0)}}'(x_{i}))= \BA(H_k^{(0)}).
$$
\end{proof}

\bigskip

Assume now that the periodic structure is singular of type (b), i.e. the set $C^{(2)}$ has exactly one element and there exists a cycle $\cc=\cc_{e_0} \in C^{(1)}$, $h^{(0)}(e_0) \notin \mathcal P$ such that $h^{(0)}(\cc_{e_0})=w(h^{(0)})\ne 1$ in $G_{\Omega_0}$. Since $\cc\in C^{(1)}$, for all solutions $H^{(0)}=\varphi\pi(v_0,v_1)H^{(1)}$ of $\Omega_0$ we have that $H^{(0)}(\cc)=H^{(1)}(\pi(v_0,v_1)(h(\cc)))$. Since $G_{R(\Omega_0)}$ is separated by $\GG$, there exists $H^{(0)}$ so that $1\ne H^{(0)}(\cc)= H^{(1)}(\pi(v_0,v_1)(h(\cc)))$, hence $\pi(v_0,v_1)(h(\cc))\ne 1$ in $G_{R(\Omega_1)}$. Let $w(x')=\varphi_0(h^{(0)}(\cc_{e_0}))$ be the image of $h^{(0)}(\cc_{e_0})$ in $\Ts_1$, and consider the word $w(x^{(1)})$ in $\HH_1$. Set  $\KK$ to be $\BA_{\HH_1}(w(x^{(1)}))$.

Define the graph $\Gamma_0$ as follows. If the canonical parabolic subgroup $\KK^\perp$ is non-abelian, set:
\begin{itemize}
\item $V(\Gamma_0)=V(\Gamma_1) \cup \{x^{(0)}\}$; 
\item $E_c(\Gamma_0)=E_c(\Gamma_1)$ and 
\item $E_d(\Gamma_0)=E_d(\Gamma_1)\cup \{ (x^{(0)}, x_j^{(1)}) \mid \hbox{ for all } x_j^{(1)} \in \KK\}$.
\end{itemize} 
If the canonical parabolic subgroup $\KK^\perp$ is abelian, set:
\begin{itemize}
\item $V(\Gamma_0)=V(\Gamma_1) \cup \{x^{(0)}\}$; 
\item $E_c(\Gamma_0)=E_c(\Gamma_1)\cup \{ (x^{(0)},x_j^{(0)})\mid x_j^{(1)}\in \KK^\perp\}$ and 
\item $E_d(\Gamma_0)=E_d(\Gamma_1)\cup \{ (x^{(0)}, x_j^{(1)}) \mid \hbox{ for all } x_j^{(1)} \in \KK\}$.
\end{itemize} 

Now, the group $\HH_0$ is defined to be $\GG(\Gamma_0)$.

Let $S_0'$ be the following set of relations: 
$$
[x^{(0)}, C_{\Ts_1}(\varphi_0(h^{(0)}(\cc_{e_0})))]=1,
$$
and set $\Ts_0$ to be the quotient $\factor{\HH_0}{\ncl\langle S_0\rangle}$, where $S_0=S_1 \cup S_0'$.

Recall that, by definition of the generating set $\bar x$ and since the generalised equation is periodised (the periodic structure being singular), the set $\{h(e) \mid e\in T\} \cup \{h^{(0)}(C^{(1)}), h^{(0)}(C^{(2)})\}$ is a basis of the free group on the alphabet $\{h_k^{(0)} \in \sigma, \sigma \in \mathcal P\}$. Furthermore, any $h_k^{(0)}\in \sigma, \sigma \in \mathcal P$ so that $h_k^{(0)}=h(e)$, $e\notin T$, $e:v\to v'$, is expressed in the generating set $\{h(e) \mid e\in T\} \cup \{h^{(0)}(C^{(1)}), h^{(0)}(C^{(2)})\}$ as 
$$
h_k^{(0)}=h^{(0)}(\p(v_0,v))^{-1} v_{k1}(C^{(1)}) v_{k2}(C^{(2)})h^{(0)}(\p(v_0, v')),
$$
for some $v_{k1}(C^{(1)})\in \langle h^{(0)}(C^{(1)})\rangle$ and some $v_{k2}(C^{(2)}) \in \langle h^{(0)}(C^{(2)})\rangle$.

\begin{lem} \label{lem:singbtauhom}
The map $\tau_0$, 
$$
h_k^{(0)}\mapsto 
\left\{ 
\begin{array}{lll}
\varphi_0(h_k^{(0)}),& \hbox{ for all $h_k^{(0)} \in \sigma, \sigma \notin \mathcal P$,} \\
\varphi_0(h_k^{(0)}),& \hbox{ for all $h_k^{(0)}=h(e), e\in T$,} \\
\varphi_0(h^{(0)}(\p(v_0, v))^{-1}) \varphi_0(v_{k1}(C^{(1)})) v_{k2}(x^{(0)}) \varphi_0(h^{(0)}(\p(v_0,v'))), & \hbox{ for all $h_k^{(0)}=h(e), e \notin T$,}
\end{array}
\right.
$$
extends to a homomorphism from $G_{\Omega_0}$ to $\Ts_0$.
\end{lem} 
\begin{proof}
Proof is analogous to the proof of Lemma \ref{lem:singatauhom}
\end{proof}

\begin{lem}\label{lem:singbdiagcom}
The homomorphism $\tau_0:G_{\Omega_0} \to \Ts_0$ makes Diagram {\rm(\ref{diag:period})} commutative.
\end{lem}
\begin{proof}
Proof is analogous to the proof of  Lemma \ref{lem:singadiagcom}
\end{proof}

\begin{lem}\label{lem:singbKcool}
The subgroup $\KK=\BA(w(x^{(1)}))$ of $\HH_0$ is $E_d(\Gamma_0)$-\cool.
\end{lem}
\begin{proof}
Proof is analogous to the proof of  Lemma \ref{lem:singaKcool}
\end{proof}

\begin{lem}\label{lem:singbH0ind}
The group $\HH_0$ satisfies the induction hypothesis {\rm IH1}: for all $x_i,x_j \in \HH_0$, we have that $(x_i,x_j) \in E_d(\Gamma_0)$ if and only if $H'(x_i) \lra H'(x_j)$ for all homomorphisms $H'$ induced by solutions $H$ from the fundamental sequence. Furthermore, we have that if $(x_i,x_j) \in E_c(\Gamma_0)$, then $H'(x_i)$ and $H'(x_j)$ belong to a cyclic subgroup for all homomorphisms $H'$ induced by solutions $H$ from the fundamental sequence.
\end{lem}
\begin{proof}
Proof is analogous to the proof of Lemma \ref{lem:qH0ind1}
\end{proof}

\begin{lem}\label{lem:singbtau0ind}
The homomorphism $\tau_0$ satisfies the induction hypothesis {\rm IH2}: for all $h_i^{(0)}$, if $\tau_0(h_i^{(0)})=x_{i1}\dots x_{ik}$, then for a fundamental sequence of solutions we have that $\BA({H^{(0)}}'(x_{ij})) \supset \BA(H_i^{(0)})$ and $\bigcap \limits_{j=1}^k \BA({H^{(0)}}'(x_{ij}))= \BA(H_i^{(0)})$.
\end{lem}
\begin{proof}
Proof is analogous to the proof of  Lemma \ref{lem:singatau0ind}
\end{proof}

We summarise the results of this section in the following

\begin{prop} \label{prop:sing}
Let $\Omega=\langle \Upsilon, \Re_{\Upsilon}\rangle$ be a generalised equation singular or strongly singular with respect to a periodic structure and let $\Re_\Upsilon$ be completed. Then, there exists a graph tower $(\Ts, \HH)$ and a homomorphism $\tau$ from $G_{\Omega}$ to $\Ts$ such that for all solutions of the fundamental sequence, there exists a homomorphism from $\Ts$ to $\GG$ that makes Diagram {\rm(\ref{diag:period})} commutative.
\end{prop}

\subsection{Regular periodic structures}

By \cite[Lemma 6.19]{CKpc},  given a generalised equation $\Omega=\langle \Upsilon, \Re_\Upsilon\rangle$ with no boundary connections, periodised with respect to a connected regular periodic structure $\langle \mathcal P,R \rangle$ and any periodic solution $H$ of $\Omega$ such that $\mathcal P(H,P)= \langle \mathcal P,R\rangle$ and so that $H$ is minimal with respect to the trivial group of automorphisms, either for all $k$, $1 \le k \le \rho$ we have $|H_k| \le 2\rho |P|$, or there exists a cycle $\cc \in \pi_1(\Gamma,v_{\Gamma})$ so that $H(\cc) = P^n$, where $1 \le n \le 2\rho$.

Suppose first that the fundamental sequence falls under the conditions of the first assumption, i.e. all the items from the periodic structure are of bounded length. Then, in a bounded number of steps $c$, all the items, bases and sections from the periodic structure will be transferred onto the sections that do not belong to the periodic structure. Let $\Omega_c$ be the corresponding generalised equation. Then, every solution of the fundamental sequence is a solution of $\Omega_c$. In this case, the graph tower for $\Omega$, is the graph tower associated to $\Omega_c$.

Suppose now that the fundamental sequence falls under the conditions of the second assumption. Then, there exists a cycle $\cc \in \pi_1(\Gamma,v_{\Gamma})$ so that $H(\cc) = P^n$, where $1 \le n \le 2\rho$. Let $w=\varphi_0(h(\cc))$  be the image of $h(\cc)$ in $\Ts_1$, $w=w(y^{(1)})$ and consider the word $w(x^{(1)})$ in $\HH_1$. Set  $\KK$ to be $\BA_{\HH_1}(w(x^{(1)}))$.

Let $x^{(0)}=\{x_1^{(0)}, \dots, x_{\m}^{(0)} \}$. Define the graph $\Gamma_0$ as follows. If the canonical parabolic subgroup $\KK^\perp$ is non-abelian, set:
\begin{itemize}
\item $V(\Gamma_0)=V(\Gamma_1) \cup \{x_1^{(0)},\dots, x_\m^{(0)} \}$; 
\item $E_c(\Gamma_0)=E_c(\Gamma_1)\cup \{ (x_i^{(0)},x_j^{(0)})\mid 1\le i<j\le \m)\}$ and 
\item $E_d(\Gamma_0)=E_d(\Gamma_1)\cup \{ (x_i^{(0)}, x_j^{(1)}) \mid \hbox{ for all } x_j^{(1)} \in \KK, i=1,\dots,\m\}$.
\end{itemize} 
If the canonical parabolic subgroup $\KK^\perp$ is abelian, set:
\begin{itemize}
\item $V(\Gamma_0)=V(\Gamma_1) \cup \{x_1^{(0)},\dots, x_\m^{(0)} \}$; 
\item $E_c(\Gamma_0)=E_c(\Gamma_1)\cup \{ (x_i^{(0)},x_j^{(0)})\mid 1\le i<j\le \m\} \cup \{(x_i^{(0)},x_j^{(1)})\mid x_j^{(1)}\in \KK^\perp, 1\le i\le \m\}$ and 
\item $E_d(\Gamma_0)=E_d(\Gamma_1)\cup \{ (x_i^{(0)}, x_j^{(1)}) \mid \hbox{ for all } x_j^{(1)} \in \KK, i=1,\dots,\m\}$.
\end{itemize} 

Then, the group $\HH_0$ is defined to be $\GG(\Gamma_0)$.

Let $S_0'$ be the set of relations defined as follows. For every equation $u_{ie}^{h(e_i)} =z_{ie}$ (we use notation from \cite[Chapter 6]{CKpc}), we write
$$
[x_i^{(0)}, C_{\Ts_1}(\varphi_0(u_{ie}))]=1,\ i=1, \dots, \m
$$
and set $\Ts_0$ to be the quotient $\factor{\HH_0}{S_0}$, where $S_0=S_1 \cup S_0'$.

Recall that, by definition of the generating set $\bar x$ associated to the periodic structure and since the generalised equation is periodised, any $h_k^{(0)}\in \sigma, \sigma \in \mathcal P$ such that $h_k^{(0)}=h(e)$, $e\notin T$, $e:v\to v'$ is expressed in the generating set $\{h(e) \mid e\in T\} \cup \{h^{(0)}(C^{(1)}), h^{(0)}(C^{(2)})\}$ as 
$$
h_k^{(0)}=h^{(0)}(\p(v_0,v))^{-1} v_{k1}(C^{(1)}) v_{k2}(C^{(2)})h^{(0)}(\p(v_0, v'))
$$
for some $v_{k1}(C^{(1)})\in \langle h^{(0)}(C^{(1)})\rangle$ and some $v_{k2}(C^{(2)}) \in \langle h^{(0)}(C^{(2)})\rangle$. Furthermore, the paths $\p(v_0,v)$ and $\p(v_0,v')$ can be represented as $\p_1e_{i_1}\p_2\dots \p_r e_{i_r} \p_{r+1}$ and $\p_1' e_{j_1}\p_2'\dots \p_{r'}' e_{j_{r'}} \p_{r'+1}'$, correspondingly, where $\p_s, \p_{s'}'$ are paths in $T_0$, $1\le s \le r+1$, $1\le s' \le r'+1$ and $e_{i_t}, e_{j_{t'}} \in \{e_1,\dots, e_\m\}$, $1\le t\le r$, $1\le t' \le r'$.

\begin{lem}
The map $\tau_0'$
$$
h_k^{(0)}\mapsto 
\left\{ 
\begin{array}{lll}
\varphi_0(h_k^{(0)}), & \hbox{for all $h_k^{(0)} \in \sigma, \sigma \notin \mathcal P$ and $h_k^{(0)} \in T \cap \Sh$,}\\
\varphi_0(h^{(0)}(\cc)), & \hbox{for all $\cc \in C^{(1)}\cup C^{(2)}$,}\\
x_i^{(0)}\varphi_0(h(e_i)), & \hbox{for all $h_k^{(0)}=h(e_i)$, $i=1,\dots,\m$,}
\end{array}
\right.
$$
induces a map $\tau_0$ from the set $\{h^{(0)}\}$ which extends to a homomorphism from $G_{\Omega_0}$ to $\Ts_0$.
\end{lem} 
\begin{proof}
By \cite[Lemma 6.14]{CKpc}, the set 
$$
\{h_k^{(0)}\mid h_k^{(0)} \in \sigma, \sigma \notin \mathcal P\} \cup \{h_k^{(0)} \mid h_k^{(0)} \in T \cap \Sh\} \cup h^{(0)}(C^{(1)}) \cup h^{(0)}(C^{(2)}) \cup \{h^{(0)}(e_1),\dots, h^{(0)}(e_\m )\} 
$$
is a generating set of $G_{\Omega_0}$.

Notice that the map $\tau_0$ on the set 
$$
S=\{h_k^{(0)}\mid h_k^{(0)} \in \sigma, \sigma \notin \mathcal P\} \cup \{h_k^{(0)} \mid h_k^{(0)} \in T\cap \Sh \} \cup h^{(0)}(C^{(1)}) \cup h^{(0)}(C^{(2)})
$$ 
coincides with the composition of the epimorphism $\pi(v_0,v_1)$ and the homomorphism $\tau_1$, and so it extends to a homomorphism on the subgroup $\langle S\rangle$ of  $G_{\Omega_0}$. In particular, the map $\tau_0$ coincides with $\varphi_0$ on the subgroup generated by all the items $h_k^{(0)}\notin \mathcal P$.

By \cite[Lemma 6.14]{CKpc}, since $\Omega_0$ is periodised, i.e. $[h^{(0)}(\cc_1), h^{(0)}(\cc_2)]=1$ for all cycles $\cc_1, \cc_2 \in C^{(1)}\cup C^{(2)}$, it follows that all the words $h^{(0)}(\mu)^{-\varepsilon(\mu)} h^{(0)}(\Delta(\mu))^{\varepsilon(\Delta(\mu))}$ belong to the subgroup generated by $S$ and so 
$$
\tau_0(h^{(0)}(\mu)^{-\varepsilon(\mu)} h^{(0)}(\Delta(\mu))^{\varepsilon(\Delta(\mu))})=1. 
$$
We note that the words $[h^{(0)}(\cc_1), h^{(0)}(\cc_2)]$ are trivial elements in the subgroup generated by $S$ for all cycles $\cc_1, \cc_2 \in C^{(1)}\cup C^{(2)}$. Therefore, in order to show that the map  $\tau_0$ extends to a homomorphism from $G_{\Upsilon_0}$ to $\Ts_0$ it suffices to show that $\tau_0(u_{ei}^{h^{(0)}(e_i)}z_{ei}^{-1})=1$ in $\Ts_0$, $1\le i \le \m$. Indeed, 
$$
\tau_0(u_{ei}^{h^{(0)}(e_i)}z_{ei}^{-1})=\varphi_0(u_{ei})^{x_i^{(0)}\varphi_0(h^{(0)}(e_i))}\varphi_0(z_{ei}^{-1}).
$$
Since from the definition of $S_0$, $[x_i^{(0)},\varphi_0(u_{ei})]=1$, it follows that 
$$
\tau_0(u_{ei}^{h^{(0)}(e_i)}z_{ei}^{-1})=\varphi_0(u_{ei}^{h^{(0)}(e_i)}z_{ei}^{-1})=1.
$$

To prove that $\tau_0$ extends to a homomorphism from $G_{\Omega_0}$ to $\Ts_0$, it is left to show that $\tau_0([h_i^{(0)}, h_j^{(0)}])=1$, for all $h_i^{(0)}, h_j^{(0)}$ such that $\Re_{\Upsilon_0}(h_i^{(0)}, h_j^{(0)})$. In fact, since the map extends to a homomorphism on the subgroup generated by $S$ we only need to check that $\tau_0([h_i^{(0)}, h_j^{(0)}])=1$, for all $h_j^{(0)}\in \mathcal P$ so that $\Re_{\Upsilon_0}(h_i^{(0)}, h_j^{(0)})$. For every $P$-periodic solution $H^{(0)}$, we have that $\az(H_i^{(0)}) \subset \az(H_j^{(0)})$, for all $h_i^{(0)} \in  \sigma, \sigma \in \mathcal P$, $h_j^{(0)} \in \mathcal P$. Recall that if $\Re_{\Upsilon_0}(h_i^{(0)}, h_j^{(0)})$, then we have that $H_i^{(0)} \lra H_j^{(0)}$ for all solutions $H^{(0)}$ of $\Omega_0$. Therefore, if $\Re_{\Upsilon_0}(h_i^{(0)}, h_j^{(0)})$ and $h_j^{(0)} \in \mathcal P$, then $h_i^{(0)} \in \sigma, \sigma \notin \mathcal P$. 

Furthermore, since the periodic structure is not strongly singular and the set $\Re_{\Upsilon_0}$ is completed, it follows that for all $h_j^{(0)} \in \mathcal P$ and for every $h_i^{(0)}$ such that $\Re_{\Upsilon_0}(h_i^{(0)}, h_j^{(0)})$,  one has that $\Re_{\Upsilon_0}(h_i^{(0)}, h_k^{(0)})$, for all $h_j^{(0)} \in \sigma, \sigma \in \mathcal P$.

By properties of the elementary transformations, we have that $\Re_{\Upsilon_1}(h_r^{(1)}, h_s^{(1)})$, for all the items $h_r^{(1)}$ which appear in the word $\pi(v_0,v_1)(h_i^{(0)})$ and all items $h_s^{(1)}$ from the word $\pi(v_0,v_1)(h_k^{(0)})$, $h_k^{(0)}\in \sigma, \sigma \in \mathcal P$. By the induction hypothesis IH2 on $\tau_1$, for any $y_m^{(1)}$ which appears in the word $\varphi_0(h_i^{(0)})$, any $y_n^{(1)}$ from the word $\varphi_0(h_k^{(0)})$ and any solution $H^{(1)}$ of $\Omega_1$ from the fundamental sequence,  we have that $\BA({H^{(1)}}'(y_m^{(1)})) \supset \BA(H^{(0)}_i)$ and $\BA({H^{(1)}}'(y_n^{(1)})) \supset \BA(H_k^{(0)})$, hence ${H^{(1)}}'(y_m^{(1)}) \lra {H^{(1)}}'(y_n^{(1)})$. By the induction hypothesis IH1 on $\HH_1$, we have that $(x_n^{(1)},x_m^{(1)})\in E_d(\Gamma_1)$ and so $\varphi_0(h_i^{(0)}) \lra \varphi_0(h_k^{(0)})$ in $\HH_1$, for all $h_k^{(0)}\in \sigma,\sigma \in \mathcal P$ and, hence, in particular, $\varphi_0(h_i^{(0)})$ belongs to the subgroup $\KK$ of $\HH_1$. Therefore, if $\Re_{\Upsilon_0}(h_i^{(0)}, h_j^{(0)})$ and $h_j^{(0)}\in \mathcal P$, then $\tau_0([h_i^{(0)}, h_j^{(0)}])=1$.
\end{proof}

\begin{lem}\label{lem:regdiagcom}
The homomorphism $\tau_0:G_{\Omega_0} \to \Ts_0$ makes Diagram {\rm (\ref{diag:period})} commutative.
\end{lem}
\begin{proof}
Proof is analogous to the one of Lemma \ref{lem:singadiagcom}.
\end{proof}

\begin{lem}\label{lem:regKcool}
The subgroup $\KK=\BA(w(x^{(1)}))$ of $\HH_1$ is $E_d(\Gamma_0)$-\cool.
\end{lem}
\begin{proof}
Proof is analogous to the one of Lemma \ref{lem:singaKcool}
\end{proof}

\begin{lem}\label{lem:regH0ind}
The group $\HH_0$ satisfies the induction hypothesis {\rm IH1}: for every $x_i,x_j \in \HH_0$, we have that $(x_i,x_j) \in E_d(\Gamma_0)$ if and only if $H'(x_i) \lra H'(x_j)$ for all homomorphisms $H'$ induced by solutions $H$ from the fundamental sequence. Furthermore, if $(x_i,x_j) \in E_c(\Gamma_0)$, then $H'(x_i) ,H'(x_j)$ belong to a cyclic subgroup for all homomorphisms $H'$ induced by solutions $H$ from the fundamental sequence.
\end{lem}
\begin{proof}
Proof is analogous to the one of Lemma \ref{lem:qH0ind1}
\end{proof}

\begin{lem}\label{lem:regtau0ind}
The homomorphism $\tau_0$ satisfies the induction hypothesis {\rm IH2}: for all $h_i^{(0)}$, if $\tau_0(h_i^{(0)})=x_{i1}\dots x_{ik}$, then for a fundamental sequence of solutions we have that $\BA({H^{(0)}}'(x_{ij})) \supset \BA(H_i^{(0)})$ and $\bigcap \limits_{j=1}^k \BA({H^{(0)}}'(x_{ij}))= \BA(H_i^{(0)})$.
\end{lem}
\begin{proof}
Proof is analogous to the one of Lemma \ref{lem:singatau0ind}
\end{proof}

We summarize the results of this section in the following
\begin{prop} \label{prop:reg}
Let $\Omega=\langle \Upsilon, \Re_{\Upsilon}\rangle$ be a generalised equation regular with respect to a periodic structure and let $\Re_\Upsilon$ be completed. Then, there exists a graph tower $(\Ts, \HH)$ and a homomorphism $\tau$ from $G_{\Omega}$ to $\Ts$ such that for all solutions of the fundamental sequence, there exists a homomorphism from $\Ts$ to $\GG$ that makes Diagram {\rm(\ref{diag:period})} commutative.
\end{prop}

\subsection{Case 15}

Suppose that the infinite branch is of type 15. Let $\Omega_0 \to \Omega_1 \to\dots\to \Omega_N$ be the branch of the tree $T_0(\Omega_{0})$ such that all the epimorphisms $\pi(v_i,v_{i+1}): G_{R(\Omega_i)}\to G_{R(\Omega_{i+1})}$ are isomorphisms, $i=0,\dots, N-2$ and $\pi(v_{N-1},v_{N}):G_{R(\Omega_{N-1})} \to G_{R(\Omega_N)}$ is a proper epimorphism. It follows by construction of the tree $T(\Omega_0)$, see \cite{CKpc}, that one can subdivide the branch $\Omega_0 \to \Omega_1 \to\dots\to \Omega_N$ as follows. Choose $0<n_0<n_1<\dots<n_k=N$ so that
\begin{itemize}
\item the quadratic part of $\Omega_j$ is non-trivial for all $j<n_0$ and $\Omega_{n_0}$ has trivial quadratic part, i.e. in the process the quadratic part of $\Omega_0$ is transferred to the non-quadratic part of $\Omega_{n_0}$ and the generalised equation $\Omega_{n_0}$ is the first one in the branch with trivial quadratic part;
\item every solution $H$ of the fundamental sequence is $P_l$-periodic in the section $[1,i_l]$ of the generalised equation $\Omega_{n_l}$ and $\Omega_{n_l}$ is periodised with respect to a regular periodic structure $\P(H,P_l)$. Furthermore, the section $[1,i_l]$ is transferred in the process from $\Omega_{n_l}$ to $\Omega_{n_{l+1}}$, $k=0,\dots, N-1$.
\end{itemize}

We now use induction on $n_k$ to show that there exists a graph tower $\Ts_0$  and homomorphisms making Diagram (\ref{diag:period}) commutative for all solutions of a fundamental sequence.

Suppose first that $k=1$. Then, the existence of the graph tower $\Ts_0$ for $G_{\Omega_0}$ and the required homomorphisms follows by Proposition \ref{prop:quad}. Assume now that $k>1$. By construction, declaring the section $[1,i_{k-1}]$ to be the active section of $\Omega_{n_{k-1}}$, the generalised equation $\Omega_{n_{k-1}}$ is periodised with respect to a regular periodic structure $\P(H,P_{k-1})$. Hence, by Proposition \ref{prop:reg}, there exists a graph tower $\Ts_{k-1}$ for $\Omega_{n_{k-1}}$ that satisfies the required properties. This proves the base of induction. 

Assume by induction that there exists a graph tower $\Ts_2$ for $G_{\Omega_{n_2}}$. Then, again using Proposition \ref{prop:reg}, we conclude that there exists a graph tower $\Ts_1$ and the required homomorphisms for $G_{\Omega_{n_1}}$.

Finally, since there exists a tower $\Ts_1$ for $G_{\Omega_{n_1}}$, by Proposition \ref{prop:quad}, there exists a graph tower $\Ts_0$ satisfying the required properties for the group $G_{\Omega_{n_0}}$. 

We arrive at the following
\begin{prop} \label{prop:case15}
Let $\Omega=\langle \Upsilon, \Re_{\Upsilon}\rangle$ be a generalised equation of the infinite branch of type $15$ and let $\Re_\Upsilon$ be completed. Then, there exist a graph tower $(\Ts, \HH)$ and a homomorphism $\tau$ from $G_{\Omega}$ to $\Ts$ such that for all solutions of the fundamental sequence, there exists a homomorphism from $\Ts$ to $\GG$ that makes Diagram {\rm(\ref{diag:period})} commutative.
\end{prop}

\subsection{Linear case}

Suppose that the infinite branch is of the type 7-10. Then, by \cite[Lemma 7.10]{CKpc}, there exists a generalised equation that repeats in the branch infinitely many times. Without loss of generality, we may assume that $\Omega$ is the generalised equation that repeats infinitely many times.

In the case of free groups, if an item is covered only once, one can delete the item together with the base that covers it and its dual. On the level of groups, this corresponds to a Tietze transformation: deletion of a generator and a relation in which this generator occurs only once. Repeating this operation, one eliminates all the ``eliminable'' bases and obtains the, so called, kernel $\Ker(\Omega)$ of the corresponding generalised equation $\Omega$, see \cite{Razborov2, KhMNull}. One can then show that $F_{R(\Omega)}=F_{R(\Ker(\Omega))}*F(Z)$, see \cite{Razborov2, KhMNull}.

In our case, if an item is covered only once and one deletes the item along with the base that covers it and its dual, then one looses the ``partially commutative structure'' in the following sense. The item $h_i$ which has been deleted together with a relation $h_i=w_i=w_i(h_1,\dots, h_{i-1}, h_{i+1}, \dots h_{\rho-1})$, has commutation constraints associated to it, hence $h_i$ appears in the commutation relations. Therefore, the presentation we obtain after deleting the generator $h_i$ and the relation $h_i=w_i$, contains relations of the form $[h_j, w_i]$ (note that, originally, all the commutation relations are commutation of the generators, hence, the term ``partially commutative structure''), see \cite[Section 4.3]{CKpc} for further discussion. It follows that, one cannot use these Tietze transformations to exhibit the splitting.

By analogy with the quadratic case, one could try to study the dynamics of the infinite branch and try to gain control on the constraints and (at best) conclude that the relations $[w_i,h_j]$ are consequences of the other relations, i.e. that for each $h_k$ that appears in $w_i$, there is a relation $[h_k,h_j]$.

Unfortunately, the dynamics of the linear branch is much harder than that of the quadratic case and we did not manage to advance in this direction (in fact, as we will see from the approach we are about to present, this might not be a sensible option).

Our approach is to study this case by ``forcing'' it to behave as the cases we know how to deal with: the quadratic case and the general case. In order to carry out this new analysis of the branch, we do the contrary of what is done in the free group case: we use Tietze transformations to \emph{introduce} new (non-active) generators and relations.  Our goal is to transform the generalised equation  to a new one, where every item is covered at least twice and then study the obtained generalised equation as in the quadratic and general cases. 

As one might expect, such a naive attempt does not (directly) work, as, in general, after some transformations, the generalised equation can fall again into one of the types 7-10. 

At this point, tribes again  play a crucial role. We show that one can ensure that every time the generalised equation falls into one of the types 7-10, the minimal tribe of the generalised equation \emph{strictly} increases. Since $\GG$ is finitely generated,  there exists a constant $N(\GG)$ that bounds any sequence of strictly dominating tribes. Furthermore, if at some point, the minimal tribe of a generalised equation turns out to be the maximal tribe, then all the items of the generalised equation, in fact, belong to this maximal tribe. In this case, we can proceed as in the free group case and eliminate the item and the corresponding relation, since the commutation relations $[w_i,h_k]$ are indeed a consequence of the relations $[h_j,h_k]$ for all $h_j$ which occur in $w_i$.

Hence, after introducing fictitious bases to make the generalised equation of type 7-10 into a generalised equation where every item is covered at least twice and running the process for the obtained generalised equation, one of the following alternatives occurs:
\begin{enumerate}
\item[A.] the process terminates in finitely many steps;
\item[B.] the obtained generalised equation defines an infinite branch of type 12 or 15, thus we can analyse it;
\item[C.] after finitely many steps, the generalised equation is again transformed into a generalised equation of the type 7-10, but the minimal tribe of the latter generalised equation strictly dominates the minimal tribe of the original generalised equation.
\end{enumerate}
We can have case C at most $N(\GG)$ times. Hence, after finitely many instances of case C, one has one of the cases A, B or the generalised equation is again of type 7-10, but we can perform the Tietze transformation, eliminate the item $h_i$ together with the relation $h_i=w_i$ and preserve the ``partially commutative structure''. All these alternatives can be analysed and hence the linear case can be tackled.

We now turn our attention to the formal analysis of the linear case. We define a new generalised equation $\Omega^a=\langle \Upsilon^a, \Re_{\Upsilon^a}\rangle$ constructed from $\Omega=\langle \Upsilon, \Re_{\Upsilon}\rangle$ as follows. For every item $h_i$ which is covered exactly once, we introduce 
\begin{itemize}
\item a base $\lambda_{i}$, so that $\varepsilon(\lambda_i)=1$, $\alpha(\lambda_i)=i$, $\beta(\lambda_i)=i+1$,
\item an item $h_{d_i}$ (along with the boundary and re-enumerate all the boundaries as appropriate) in the non-active part of $\Omega$ and,
\item the dual base $\Delta(\lambda_i)$ of $\lambda_i$ so that $\varepsilon(\Delta(\lambda_i))=1$, $\alpha(\Delta(\lambda_i))=d_i$ and $\beta(\Delta(\lambda_i))=d_i+1$. 
\end{itemize}

We set $\Re_{\Upsilon^a}(h_{d_i})=\Re_\Upsilon(h_i)$. We call the bases $\lambda_{i}$ \emph{auxiliary}. Denote the obtained generalised equation by $\Omega^a$. 

It is obvious that the group $G_{\Omega^a}$ is isomorphic to $G_{\Omega}$. Furthermore, $\tp(\Omega^a)\ge 11$. We apply the process to the generalised equation $\Omega^a$ and follow the branch of the fundamental sequence (recall that, by Lemma \ref{lem:discfam}, there is only one such branch). 

Assume first that one of the auxiliary bases $\lambda_i$ is a carrier base of the generalised equation $\Omega_{M(\lambda_i)}$. In this case, we $\mu$-tie the boundary $\beta(\lambda_i)$ in all the bases $\mu$, $\mu\ne \lambda_i$ that contains $\beta(\lambda_i)$ and close the section $\sigma(\lambda_i)=[1,\beta(\lambda_i)]$. 

We now apply $\D 5$ to the obtained generalised equation and, using $\lambda_i$ as the carrier, we transfer all the bases from $\sigma(\lambda_i)$ to the non-active part. Notice that, since the base $\lambda_i$ is the carrier and the section $\sigma(\lambda_i)$ is closed, the pair of dual bases $(\lambda_i, \Delta(\lambda_i))$ is removed and the number of auxiliary bases in the active part of the obtained generalised equation has decreased. The non-active section $\sigma(\Delta(\lambda_i))$ of $\Omega_{M(\lambda_i)}$ defines a (non-active) closed section in $\Omega_j$, for all $j>{M(\lambda_i)}$, that we still denote by $\sigma(\Delta(\lambda_i))$.

It follows from the process that the branch of the tree $T(\Omega^a)$ is either finite and the leaf is of type 1 (it is a proper quotient) or of type 2 (there are no active sections), or the branch is infinite of type 12 or of type 15.

\subsubsection*{Case I} If the branch of $T(\Omega^a)$ is finite and the leaf is of type 1, then any solution of fundamental sequence is a solution of the leaf. By induction hypothesis on the height of proper quotients, we can conclude that there exists a graph tower with the required properties.

\subsubsection*{Case II} Assume now that the branch is finite and the leaf $\Omega_k'$ is of type 2. Let $\Omega_{m(\lambda_i)}$ be the generalised equation obtained from $\Omega_k'$ by declaring one of the sections $\sigma(\Delta(\lambda_i))$ active and all the other sections non-active.

\subsubsection*{Case II.1} Suppose first that every active item in $\Omega_{m(\lambda_i)}$ is covered at least twice. In this case, we argue  by induction on the number of auxiliary bases and prove the existence of a graph tower for $\Omega_{m(\lambda_i)}$. If there are no auxiliary (active) bases in $\Omega_{m(\lambda_i)}$, then, by Propositions  \ref{prop:sing}, \ref{prop:quad} and \ref{prop:case15}, there exists a graph tower for $G_{\Omega_{m(\lambda_i)}}$. Since the number of auxiliary bases of $\Omega_{m(\lambda_i)}$ is strictly less than the number of auxiliary bases of $\Omega_k$ (as the base $\lambda_i$ is eliminated in $\Omega_{m(\lambda_i)}$), then we conclude by induction that there exists a tower for $\Omega_{m(\lambda_i)}$.

\subsubsection*{Case II.2} Suppose next that there is an item $h_j$ in $\sigma(\Delta(\lambda_i))$ which is covered exactly once, say by the base $\nu$. Note that the tribes of all the items in $\sigma(\Delta(\lambda_i))$ dominate the tribe $t(\lambda_i)$.  To simplify the notation we write $\Omega_{m(\lambda_i)}=\Omega$.

If $h_j$ belongs to the tribe $t(\lambda_i)$, then since the tribes of all the items in $\sigma(\Delta(\lambda_i))$ dominate the tribe $t(\lambda_i)$, it follows that $\nu$ (and so $\Delta(\nu)$) belong to the tribe $t(\lambda_i)$. The base $\nu$ is eliminable and we eliminate it together with the item $h_j$ as follows.

We have that the variable $h_j$ occurs only once in $\Upsilon$: in the equation $h(\nu)^{\varepsilon(\nu)}=h(\Delta(\nu))^{\varepsilon(\Delta(\nu))}$ corresponding to the base $\nu$, denote this equation by $s_\nu$.  Therefore, in the group $G_{\Omega}$ the relation $s_\nu$ can be written as $h_j = w_\nu$, 
$$
h_j=h_{j-1}^{-1}\cdots h_{\alpha(\nu)}^{-1} h(\Delta(\nu))h_{\beta(\nu)-1}^{-1}\cdots h_{j+1}^{-1}. 
$$
Consider the generalised equation $\Upsilon'$ obtained from $\Upsilon$ by deleting the equation $s_\nu$ and the item $h_j$. The presentation of $G_{\Upsilon'}$ is obtained from the presentation of $G_\Upsilon$ using a Tietze transformation. Thus, these groups are isomorphic. We define the relation $\Re_{\Upsilon'}$ as the restriction of $\Re_\Upsilon$ to the set $h\setminus\{h_j\}$, and $\Omega'=\langle\Upsilon',\Re_{\Upsilon'}\rangle$.

It follows that
$$
G_{\Omega}\simeq \factor{G[h_1,\dots,h_{j-1},h_{j+1},\dots,h_\rho]}{\Omega'\cup\{[h_k,w_\nu]\mid \Re_\Upsilon(h_j,h_k)\})}.
$$
Note that if $\Re_\Upsilon(h_j, h_k)$, then, since all the items covered by $\nu$ or $\Delta(\nu)$ belong to tribes that dominate $t(\lambda_i)$, it follows that $\Re_\Upsilon(h_i,h_k)$ for all $i=\alpha(\nu),\dots, j-1,j+1, \dots, \beta(\nu)-1,\alpha(\Delta(\nu)),\dots, \beta(\Delta(\nu))-1$. We conclude that the groups $G_{\Omega}$ and $G_{\Omega'}$ are isomorphic and we denote this isomorphism by $i:G_{\Omega}\to G_{\Omega'}$.

Notice that the number of eliminable bases of $\Omega'$ is strictly lower than the number of eliminable bases of $\Omega$. Furthermore, if the number of eliminable bases of $\Omega'$ is $0$, then $\Omega'=\Ker(\Omega)$. Since for the kernel $\Ker(\Omega)$ the linear case can not further occur, it follows that the fundamental sequence factors either through a finite branch or through an infinite branch of type 12 or type 15. Then, by Propositions \ref{prop:sing}, \ref{prop:quad} and \ref{prop:case15}, there exists a graph tower for $\Omega'$. We conclude by induction on the number of eliminable bases that there exists a graph tower $\Ts'$ (the group $\HH'$ and the homomorphism $\tau'$) for the group $G_{\Omega'}$.

We now show that the triple $\Ts'$, $\HH'$ and $i \tau'$ is a graph tower associated to $G_{\Omega}$. Since $i$ is an isomorphism, it implies that Diagram (\ref{diag:period}) is commutative. We are left to show that $i\tau'$ satisfies the induction hypothesis IH1 and IH2. Since  $\HH'$ satisfies the property IH1 for the homomorphisms induced by solutions of $\Omega'$ and every solution of $\Omega$ induces a solution of $\Omega'$, it follows that $\HH'$ satisfies the assumption IH1 for the family of solutions of $\Omega$.  Let $h_j$ be the item of $\Omega$ which has been eliminated. It suffices to prove that for all $x_i$ in $i\tau'(h_j)$ and any solution from the fundamental sequence we have that $\BA(H'(x_i)) \supset \BA(H_j)$. The latter follows since $h_j = w_\nu$, $h_j$ belongs to the tribe $t(\lambda_i)$, any $h_k$ in $w_\nu$ dominates the tribe $t(\lambda_i)$ (i.e. for any solution from the fundamental sequence $\BA(H_k) \subset \BA(H_j)$) and $\tau_1$ satisfies the induction hypothesis IH2 for the items $h_k$ in $\Omega'$.

\subsubsection*{Case II.3} Suppose finally that all the active items that are covered once in $\Omega_{m(\lambda_i)}$ belong to tribes that strictly dominate $t(\lambda_i)$. In this case, instead of $\Omega_{m(\lambda_i)}$, we consider the generalised equation $\Omega_{m(\lambda_i)}^a$: for every active item that is covered only once we introduce a new pair of auxiliary bases, and apply the argument given above for $\Omega^a$ to $\Omega_{m(\lambda_i)}^a$. Note that all the new auxiliary bases in $\Omega_{m(\lambda_i)}^a$ belong to tribes that strictly dominate $t(\lambda_i)$ and that the length of a sequence of strictly dominating tribes is bounded above. Furthermore, if the tribe $t(\lambda_i)$ is maximal, then all the items belong to this tribe. In particular, if the tribe is maximal and there are items covered just once, they are eliminable and as we have shown in Case II.2, in this case, there exists the corresponding graph tower. Therefore, by (increasing) induction on the length of the sequence of strictly dominating tribes for $t(\lambda_i)$, we prove the existence of the corresponding graph tower.

\subsubsection*{Case III} Suppose that the fundamental branch in $T(\Omega^a)$ is infinite, we show that in this case, the branch is of type 12 or 15. By construction, every auxiliary base of $\Omega^a$ can be the carrier at most once. It follows that if the fundamental branch in $\Omega^a$ is infinite, then there exists $N$ so that for all $n>N$ the carrier of $\Omega_n$ is not an auxiliary base. Furthermore, if the carrier of $\Omega_n$ has its dual in the non-active part, then the complexity of the generalised equation $\Omega_{n+1}$ is strictly lower than the complexity of $\Omega_n$. 

We conclude that in the infinite branch there exists a generalised equation $\Omega_{k}$ so that 
\begin{itemize}
\item for all $n>k$, the carrier base of $\Omega_n$ is neither auxiliary nor does its dual belong to the non-active part of $\Omega_n$;
\item if $\tp(\Omega_k)=12$, then $\Omega_k$ repeats in the infinite branch infinitely many times and $h_1^{(k)}$ belongs to a minimal tribe. 
\end{itemize}
Using the description of the cases 12 and 15, we conclude that there exists a generalised equation $\Omega_k'$ so that $G_{R(\Omega_k')}$ is a proper quotient of $G_{R(\Omega_k)}$ and such that every solution of the fundamental sequence is a composition of an automorphism of type 12 and 15, correspondingly, the epimorphism from $G_{R(\Omega_k)}$ to $G_{R(\Omega_k')}$ and a solution of $\Omega_k'$. By induction hypothesis there exists a graph tower for $G_{\Omega_k'}$ and by Propositions \ref{prop:sing}, \ref{prop:quad} and \ref{prop:case15} one can construct a tower for $G_{\Omega_k}$.

Finally, we summarize the results of this whole section in the theorem below.
\begin{thm}
Let $G$ be a limit group over $\GG$ and let $\Omega_0\to\dots \to\Omega_q$ be the fundamental branch for $G$. Then, there exists a homomorphism from $G_{\Omega_0}$ to a graph tower $\Ts_0$ which makes Diagram {\rm(\ref{diag:period})} commutative. Furthermore, if the group $G$ is given by its finite radical presentation, then the graph tower $\Ts_0$ and the corresponding homomorphism can be constructed effectively.
\end{thm}

\section{Graph towers are discriminated by $\GG$}\label{sec:11}

In Section \ref{sec:10}, we proved that given a limit group $G$ over $\GG$ and a fundamental branch $\Omega_0\to\dots \to \Omega_q$, one can effectively construct a graph tower $\Ts_0$ and a homomorphism $\tau_0:G_{\Omega_0}\to \Ts_0$ that makes Diagram (\ref{diag:period}) commutative.  In this section, we show that the graph tower $\Ts_0$ is discriminated by $\GG$ by the family of homomorphisms induced by the fundamental sequence. We use this fact to conclude that every limit group over $\GG$ is a subgroup of a graph tower. Furthermore, if $G$ is given by its finite radical presentation, then the graph tower $\Ts_0$ and the embedding can be constructed effectively.

\begin{thm}\label{thm:towerdisc}
Let $G$ be a limit group over $\GG$ and let $(\Ts,\HH)$ be the graph tower associated to the fundamental branch of $G$. Then, the graph tower $\Ts$ is discriminated by $\GG$ by the family of homomorphisms induced by solutions that factor through the fundamental branch.
\end{thm}
\begin{proof}
We prove the statement by induction on the height of the graph tower. Suppose that the height of the graph tower is $0$. Then, $\Ts$ is the partially commutative group associated to a leaf of the tree $T_{\sol}(\Omega)$ for some generalised equation $\Omega$ and so, by \cite[Proposition 9.1]{CKpc}, the graph tower $\Ts$ is discriminated by $\GG$ by the family of solutions of the generalised equation.

Suppose that if $\Ts$ is a graph tower of height less than or equal to $l-1$, then $\Ts$ is discriminated by $\GG$ by the family of homomorphisms induced by solutions of the fundamental sequence.

Let $\Ts_0$ be a  graph tower of height $l$. By Lemma \ref{lem:prtower}, the graph tower $\Ts_0$ has a decomposition as an amalgamated product. We now analyse each of the possible decompositions. 

Notice that if $\Ts_0$ is of type a1), then it only appears in the construction of the leaves of $T_{\sol}$, i.e. the graph tower $\Ts_1$ coincides with the associated partially commutative group $\GG(\Gamma_1)$. In this case, since $\KK^\perp$ is directly indecomposable, it follows from \cite[Corollary 2.11]{CKpc} that the family of solutions $H$ of the fundamental sequence which maps the variables $x^{(0)}$ into irreducible elements of the subgroup $\KK^\perp$ is a discriminating family.

\medskip

Assume now that the graph tower $\Ts_0$ is of type a2). This case appears either in the construction of a leaf, or in the presence of a periodic structure or in the particular case when the quadratic equation in the normal form does not satisfy the properties $\circledast$ and $\circledast\circledast$, see Definition \ref{defn:tower} (and $\KK^\perp$ is $E_c(\Gamma)$-abelian). 

If it appears in the construction of a leaf, the result follows analogously to the case a1).

Let us now consider the remaining cases. We begin with an observation. Let $\GG_\KK$ be the canonical parabolic subgroup such that for the fundamental sequence of solutions $H$ we have that $\langle \az(H'(\KK))\rangle = \GG_\KK$. Since $\GG_\KK^\perp \lra \GG_\KK$, by the induction hypothesis IH1, it follows that $\GG_\KK^\perp < \KK^\perp$. Since $\KK^\perp$ is $E_c(\Gamma_1)$-abelian, so is $\GG_\KK^\perp$. Since for the discriminating family of solutions the subgroup defined by $H'(\KK^\perp)$ is cyclic and $\langle\az(H'(\GG_\KK^\perp))\rangle=\GG_\KK^\perp$, it follows that $\GG_\KK^\perp$ is a cyclic canonical parabolic subgroup generated by the generator $a$. Let $C_\GG(a)=\langle a\rangle \times \BA(a)$. Since $\langle a\rangle =\GG_{\KK^\perp} \lra \GG_\KK$, we have that $\GG_\KK<\BA(a)$. On the other hand, since $\BA(a) \lra H'(\KK^\perp)$, by the induction hypothesis IH1, it follows that $\BA(a)< \KK$ and so $\BA(a)<\GG_\KK$. Hence $\GG_\KK=\BA(a)$ and $\GG_{\KK^\perp}=\langle a\rangle$. Furthermore, since $C_\GG(a)<H'(C_{\Ts_1}(\KK^\perp))<C_\GG(H'(\KK^\perp))< C_\GG(a)$ for all homomorphisms $H'$ induced by the discriminating family, we conclude that $C_\GG(a)=\GG_{C_{\Ts_1}(\KK^\perp)}$.

Since $\Ts_0$ is an amalgamated product, it follows that every non-trivial element $e$ of $\Ts_0$ can be written in the reduced form as $a_1b_1a_2b_2\cdots a_k b_k a_{k+1}$, where $a_2, \dots, a_{k+1} \notin  C_{\Ts_1}(\KK^\perp)$ and $b_1,\dots, b_k$ are words in the generators $\{x_1^l,\dots, x_{m_l}^l\}$. Hence, there exist $w_i \in \KK^\perp$ such that $[a_i,w_i]\ne 1$, $i=2,\dots, k+1$.

Since, by induction hypothesis, $\Ts_1$ is discriminated by $\GG$ by the family of homomorphisms induced by the set of solutions, there exists a homomorphism ${H^{(1)}}'$ which is injective on the finite set $\{[a_i,w_i] \mid i=2,\dots, k+1\}$. Therefore,  ${H^{(1)}}'([a_i,w_i])=[{H^{(1)}}'(a_i),a^{l_i}]\ne 1$. Hence ${H^{(1)}}'(a_i)\notin C_\GG(a)$ and, in particular, using the above observation, ${H^{(1)}}'(a_i)\notin \GG_{C_{\Ts_1}(\KK^\perp)}$.

The homomorphism ${H^{(1)}}'$ induces a homomorphism ${H^{(1)}}''$ from $\Ts_0$ to the group 
$$
A=\GG \ast_{\GG_{C_{\Ts_1}(\KK^\perp)}} \left(\GG_{C_{\Ts_1}(\KK^\perp)} \times \langle x_1^l,\dots, x_{m_l}^l\mid [x_i^l,x_j^l], 1\le i<j\le n \rangle\right), 
$$
i.e. ${H^{(1)}}''$ is a homomorphism from $\Ts_0$ to the extension of the centraliser of the element $a$: 
$$
A=\GG \ast_{C_\GG(a)} \langle C_\GG(a), x_1^l,\dots,x_{m_l}^l \mid [C_\GG(a),x_i^l]=1, [x_i^l,x_j^l]=1 \rangle.
$$
Moreover, the image $H'(a_1)b_{1}\dots b_{k} H'(a_{k+1})$ of $e$ in $\GG$, is a reduced element of the amalgamated product $A$ since $H'(a_i) \notin \GG_{ C_{\Ts_1}(\KK^\perp)}$, $i=2,\dots,k+1$.

Furthermore, in all the cases under consideration (for periodic structures and for exceptional subcases of the quadratic case), we have that any solution ${H^{(0)}}$ of the fundamental sequence is the composition of a canonical automorphism $\phi$ of $G_{R(\Omega_0)}$, the epimorphism $\pi(v_0,v_1)$ and a solution ${H^{(1)}}$. Moreover, notice that any such automorphism $\phi$ induces an automorphism of $A$ that fixes $\GG$. Hence  the family of solutions ${H^{(0)}}$ of the fundamental sequence defines a family of homomorphisms from $A$ to $\GG$ that are obtained from a (non-trivial) solution (induced by $\pi(v_0,v_1) {H^{(1)}}$) by pre-composing with canonical automorphisms.

By \cite[Lemma 4.17]{CK1} and \cite[Corollary 2.11]{CKpc}, the group $A$ (an extension of a centraliser of the canonical generator $a$) is discriminated by $\GG$ by this family of homomorphisms. Therefore, we conclude that $\Ts_0$ is discriminated by $\GG$ by the family of homomorphisms induced by the solutions of the fundamental sequence.

\medskip

Assume that $\Ts_0$ is of type b1) or b2). These cases occur in the presence of a periodic structure or in the particular case when the quadratic equation in the normal form does not satisfy the properties $\circledast$ and $\circledast\circledast$ from Definition \ref{defn:tower}. By Remarks \ref{rem:blockperiod} and \ref{rem:blocks}, the homomorphisms $H'$ induced by solutions of the fundamental sequence (only in the case b1)) satisfy that $H'(u)$ and $H'(x_i^l)$ are irreducible elements of $\GG_{\KK^\perp}$. Now the proof is analogous to the case a2).

\medskip

Suppose that $\Ts_0$ is of type c). This case only occurs when the quadratic equation satisfies  properties $\circledast$ and $\circledast\circledast$ from Definition \ref{defn:tower}.

\subsubsection*{Claim.} The subgroup of $\Ts_1$ generated by $C_{\Ts_1}(\KK^\perp)$ and $\varphi_0(\lambda_i)$, where $\lambda_i$ is a quadratic-coefficient base, $i=2g+1,\dots,m$ is the direct product
$$
\langle C_{\Ts_1}(\KK^\perp), \varphi_0(\lambda_{2g+1}),\dots, \varphi_0(\lambda_m)\rangle =  C_{\Ts_1}(\KK^\perp)\times \langle \varphi_0(\lambda_{2g+1}),\dots, \varphi_0(\lambda_m)\rangle.
$$

Let us prove the claim. In fact, we show that the subgroup $C$ generated by $\KK^\perp$ and $C_{\Ts_1}(\KK^\perp)$ is a direct product. Then,  since $\varphi_0(\lambda_i)\in \KK^\perp$, the claim will follow. By Remark \ref{rem:blocks}, for the family of homomorphisms $H'$ induced by the set of solutions of the fundamental sequence, we have that for all minimal items $h_i$, the image $H'_i$ is a block element of $\GG_{\KK^\perp}$ and $\BA(H'_i)=\BA(\GG_{\KK^\perp})$. In particular, it follows that $\GG_{\KK^\perp}$ is a (non-abelian) directly indecomposable canonical parabolic subgroup. By the description of centralisers in partially commutative groups, it follows that $\GG_{\KK^\perp} \cap C_\GG(\GG_{\KK^\perp})=1$ and so $\GG_{\KK^\perp} \lra \GG_{C_{\Ts_1}(\KK^\perp)}$. By induction hypothesis, we have that the subgroup $C$ is discriminated by the subgroup $\GG_{\KK^\perp} \times \GG_{C_{\Ts_1}(\KK^\perp)}$ and so $C$ is the direct product of $\KK^\perp$ and $C_{\Ts_1}(\KK^\perp)$. Indeed since $C=\KK^\perp \cdot C_{\Ts_1}(\KK^\perp)$ and $[\KK^\perp,C_{\Ts_1}(\KK^\perp)]=1$, we are left to show that $\KK^\perp \cap C_{\Ts_1}(\KK^\perp)=1$. Suppose that there exist $u \in \KK^\perp$ and $v\in C_{\Ts_1}(\KK^\perp)$ so that $u=v$ in $C$. Then, $H'(u)\in \GG_{\KK^\perp}$, $H'(v)\in \GG_{C_{\Ts_1}(\KK^\perp)}$ and $H'(u)=H'(v)$. Since $\GG_{\KK^\perp}\lra \GG_{C_{\Ts_1}(\KK^\perp)}$, so $H'(u)=H'(v)=1$ for all homomorphisms $H'$. It follows that $u=v=1$.

The strategy to prove that $\Ts_0$ is discriminated by $\GG$ is as follows. As above, to prove that $\Ts_0$ is discriminated, it suffices to prove that so is  
$$
\GG \ast_{H'(C_{\Ts_1}(\KK^\perp))\times  E} \left(\langle E,x_1^l,\dots, x_n^l \mid W\rangle \times H'(C_{\Ts_1}(\KK^\perp))\right),
$$
where $E=H'(\langle \varphi_0(\lambda_{2g+1}),\dots, \varphi_0(\lambda_m)\rangle)$. To prove the latter, we follow the standard argument for limit groups over free groups, see for instance \cite[Proposition 4.22]{Wilton}. 

There are two adjustments required for the argument in \cite{Wilton} to go through. Firstly, free groups satisfy the BP property, see \cite[Lemma 4.13]{Wilton} and secondly, limit groups over free groups are CSA. In general, partially commutative groups do not have the BP property (see \cite{Blath}). In our case, since, by induction, we assume that $\Ts_1$ is discriminated by the family of homomorphisms induced by a fundamental sequence, it follows that the images of $v_i, w_j \in \Ts_1$ under the discriminating family are block elements (see Remark \ref{rem:blocks}). Furthermore, by Lemma \ref{lem:eired}, there exists a non-trivial curve $U$ which is mapped to irreducible elements for all solutions of a fundamental sequence. We prove in \cite[Lemma 4.17]{CKpc} that we do have the BP property with respect to irreducible elements in partially commutative groups. Therefore, the argument given in \cite[Proposition 4.22]{Wilton} goes through, using \cite[Lemma 4.17]{CKpc} instead of \cite[Lemma 4.13]{Wilton}.

The argument of \cite{Wilton} uses the CSA property in \cite[Example 4.21]{Wilton}. In general partially commutative group do not have the CSA property, but we do can apply the argument from \cite[Example 4.21]{Wilton} in our context. Let $\Sigma$, $\zeta$, $S$, $S'$ and $t$ be as in \cite[Example 4.21]{Wilton}. By assumption, we have that $H'(S)$ is non-abelian and we want to show that $H'(S')$ is also non-abelian. 

By contradiction, assume that $H'(S')$ is abelian. By induction hypothesis, the discriminating family of homomorphisms $H'$ induced by the solutions has the property that if $h_i$ is a minimal item, then $H'_i$ is a block element and there exists an irreducible element $H(U)\in H'(S')$. It follows that if the image $H'(S')$ is abelian, then it is cyclic. In partially commutative groups, cyclic subgroups satisfy the property that if $H'(S') \cap H'({S'})^{H'(t)}$ is non-trivial, then $H'(t)$ commutes with $H'(S')$ and hence $H'(S)$ is abelian, deriving a contradiction. 

Now the argument is identical to the one given in \cite[Proposition 4.22]{Wilton}. 
\end{proof}

\begin{cor} \label{cor:limit}
Let $G$ be a limit group over $\GG$. Then, $G$ is a subgroup of a graph tower $(\Ts,\HH)$. Furthermore if $G$ is given by its finite radical presentation, then the graph tower $(\Ts, \HH)$ and the embedding can be effectively constructed.
\end{cor}
\begin{proof}
Since $\Ts$ is discriminated by $\GG$ by the family of homomorphisms induced by a fundamental sequence of solutions associated to $G$ and since Diagram (\ref{diag:period}) is commutative, it follows that $\ker(G_{\Omega}\to G_{R(\Omega)})<\ker \tau$. Hence, by the universal property of the quotient, $\tau$ induces a homomorphism $\tau'$ from $G_{R(\Omega)}$ to $\Ts$. Since the fundamental sequence that discriminates $G$ into $\GG$ factors through $\Ts$, it follows that the composition of  $\pi:G \to G_{R(\Omega)}$ and $\tau$ is an embedding of $G$ into $\Ts$.
\end{proof}

\begin{thm} \label{thm:main}
Let $G$ be a finitely generated group. The group $G$ acts essentially freely co-specially on a real cubing if and only if it is a subgroup of a graph tower.
\end{thm}
\begin{proof}
Follows from Theorem \ref{thm:gen} and Corollary \ref{cor:limit}.
\end{proof}

In the case of free actions, the above theorem results in the following corollary, which can be likened to Rips' theorem on free actions on real trees.
\begin{cor} \label{cor:main}
A finitely generated group $G$ acts freely, essentially freely and co-specially on a real cubing if and only if $G$ is a subgroup of the graph product of free abelian and {\rm(}non-exceptional{\rm)} surface groups.

In particular, if the real cubing is a real tree, then $G$ is a {\rm(}subgroup of{\rm)} the free product of free abelian groups and {\rm(}non-exceptional{\rm)} surface groups.
\end{cor}

We now record some properties of limit groups over partially commutative groups.

\begin{prop} \
\begin{enumerate}
\item Every limit group over $\GG$ is torsion free.
\item Every 2-generated limit group is either free or free abelian.
\item Every solvable subgroup of a limit group is abelian.
\item Let $\mathcal G$ be an algebraic group over $\BR$ in which $\GG$ embeds. Then, for any limit group $G$ over $\GG$ there exists an embedding $G \hookrightarrow\mathcal  G$. In particular, $G$ embeds into $\SL_n(\BR)$.  The natural map $G \to \PSL_n(\BR)$ is also an embedding.
\item 
Let $\GG$ be a partially commutative group and let $(\Ts,\HH)$ be a graph tower associated to a limit group $G$ over $\GG$. Then, abelian subgroups of $\Ts$  {\rm(}and thus of $G${\rm)} are free, and there is a uniform bound on their rank.
\end{enumerate}
\end{prop}
\begin{proof}
All but the last statement follow from definition using general arguments, see, for example,  \cite{Wilton}. To prove the last one we use the fact that every limit group over $\GG$ is a subgroup of a graph tower.

First notice that since, by Theorem \ref{thm:towerdisc}, the group $\Ts$ is discriminated by $\GG$, then any abelian subgroup of $\Ts$ is torsion-free. Let $A$ be an abelian subgroup of $\Ts$. Let us prove that the rank of $A$ is uniformly bounded by induction on the height of $\Ts$. If the height is $0$, then the graph tower coincides with the partially commutative group $\GG$ and hence the rank of an abelian subgroup is bounded by the number of vertices in a maximal clique of the graph $\Gamma$ or, more coarsely, it is bounded by the number of generators of $\GG$. 

Let $\Ts=\Ts^l$ be of height $l$. By Lemma \ref{lem:prtower}, the group $\Ts$ admits a decomposition as an amalgamated product. Let $T$ be the Bass-Serre tree associated to the decomposition of $\Ts$. Notice that the rank of an abelian subgroup of $\Ts^{l-1}$ is uniformly bounded by $N_{l-1}$ by induction hypothesis. Furthermore, the other vertex group is a direct product of a subgroup of $\Ts^{l-1}$ and either a free, or a free abelian or a surface group. Hence in all the cases, we conclude that the rank of an abelian subgroup of this vertex group is uniformly bounded by $N_{l-1}+m_l$. If the subgroup $A$ fixes a vertex of $T$ then $A$ is a subgroup of a (conjugate) of a vertex group and so its rank is uniformly bounded by $N_{l-1}+m_l$. Otherwise, $A$ fixes a line $T_A$ in $T$, on which it acts by translations. The quotient $\Delta = A\backslash T_A$ is topologically a circle; after some collapses $\Delta$ is an HNN-extension; so the rank of $A$ is uniformly bounded by $N_{l-1}+m_l+1$.
\end{proof}

\section{Irreducible components}
If a group $G$ is equationally Noetherian, then every algebraic set $V$ in $G^n$ is a finite union of its irreducible components. Using the duality between the categories of algebraic sets and coordinate groups, one can conclude  that if $G$ is equationally Noetherian, then any coordinate group is the subdirect product of the direct product of the coordinate groups of the irreducible components, \cite{BMR1}.

Since partially commutative groups are equationally Noetherian, any finitely generated residually $\GG$ group is a subdirect product of the direct product of finitely many limit groups over $\GG$.

In the case of free groups, Kharlampovich and Miasnikov, \cite{KhMNull}, used the Makanin-Razborov process to describe an embedding of a residually free group into the direct product of limit groups. In a recent work on the structure of finitely presented residually free groups, Bridson, Howie, Miller and Short, \cite{BHMS}, gave a different (canonical) construction of such an embedding.

The aim of this section is to generalise the forementioned results and show that given a finitely generated residually $\GG$ group $G$, one can effectively construct an embedding of $G$ into the direct product of limit groups over $\GG$. Namely, we prove the following
\begin{thm} \label{thm:sollast}
Let $\GG$ be a partially commutative group and let $G$ be a finitely generated residually $\GG$ group. Then, one can effectively construct finitely many fully residually $\GG$ graph towers $\Ts_1,\dots,\Ts_k$ and homomorphisms $p_i$ from $G$ to $\Ts_i$, $i=1,\dots,k$, so that any homomorphism from $G$ to $\GG$ factors through a graph tower $\Ts_i$, for some $i=1,\dots,k$, i.e. for any homomorphism $\varphi:G\to \GG$ there exist $i\in \{1,\dots,k\}$ and a homomorphism $\varphi_i:\Ts_i\to \GG$ so that $\varphi=p_i\varphi_i$. In particular, $G$ is a subgroup of the direct product of the graph towers $\Ts_i$, $i=1,\dots,k$ and a subdirect product of the direct product of groups $p_i(G)<\Ts_i$, $i=1,\dots, k$.
\end{thm}
Notice that $p_i(G)<\Ts_i$, $i=1,\dots, k$, are finitely generated subgroups of the graph towers $\Ts_i$. Since the graph towers $\Ts_i$ are discriminated by $\GG$, it follows that $p_i(G)$, $i=1,\dots, k$ are limit groups over $\GG$.

Since partially commutative groups are equationally Noetherian, any finitely generated residually $\GG$ group admits a finite radical presentation.

\begin{cor}
Let $G$ be a finitely generated residually $\GG$ group given by its finite radical presentation. Then, one can effectively construct an embedding of $G$ into the direct product of finitely many limit groups over $\GG$.
\end{cor}

It is worthwhile mentioning that, in general, the decomposition  we construct is not minimal, i.e. $G$ might be presented as the subdirect product of a direct product of (strictly) less than $k$ limit groups.
\begin{proof}
For a limit group over $\GG$, we have effectively constructed a graph tower associated to the fundamental branch, see Corollary \ref{cor:limit}. In the proof, the fact that the fundamental sequence is a discriminating family is  used only in and when referring to Lemma \ref{lem:discfam}, i.e. when given a decomposition of a discriminating family of homomorphism $\{\varphi_i\}$ into a finite union of families of homomorphisms $\{\varphi_{i1}\}, \dots, \{\varphi_{il}\}$ (with a specific property), we conclude that then one of the subfamilies is a discriminating family.

Let $G$ be a finitely generated residually $\GG$ group and let $T_{\sol}(G)$ be the Makanin-Razborov diagram constructed in \cite{CKpc}. We construct a tree $T_{\sol}'(G)$ from the tree $T_{\sol}(G)$ as follows. We follow the proof given in Section \ref{sec:10}, but for each branch $B$ of the tree $T_{\sol}(G)$, if a family of homomorphisms $\{\varphi_i\}$ of the branch $B$ is decomposed into a finite union $\{\varphi_{i1}\}, \dots, \{\varphi_{il}\}$ of subfamilies (with certain properties), we replace the branch $B$ by $l$ identical branches $B_1, \dots, B_l$ and require that the homomorphisms that factor through the branch $B_j$ are homomorphisms that factor through $B$ and satisfy the properties satisfied by the family $\{\varphi_{ij}\}$. We denote the new tree by $T_{\sol}'(G)$.

Now, for each branch $B: \Omega_{0,B}\to \Omega_{1,B}\to \dots\to \Omega_{n, B}$ of the tree $T_{\sol}'(G)$,  as in Section \ref{sec:10},  we construct a  graph tower $\Ts_{B}$ and a homomorphism $\tau_{B}: G_{\Omega_{0,B}} \to \Ts_{B}$ that makes Diagram (\ref{diag:period}) commutative.

An argument, analogous to the one given in Section \ref{sec:11}, shows that each graph tower constructed for a branch of the tree $T_{\sol}'(G)$ is discriminated by $\GG$.

Let $\pi_{B}$ be the homomorphism from $G$ to $G_{\Omega_{0,B}}$ constructed in \cite[Section 3.3]{CKpc} and set $p_{B}:G\to \Ts_{B}$ to be $\pi_{B} \tau_{B}$. Since the tree $T_{\sol}'(G)$ describes all the homomorphisms from $G$ to $\GG$ (see \cite{CKpc}), by the commutativity of Diagram (\ref{diag:period}), we conclude that every homomorphism from $G$ to $\GG$ factors through some graph tower $\Ts_{B}$.

Furthermore, since the group $G$ is residually $\GG$, it follows that the homomorphisms $p_{B}:G\to \Ts_{B}$ induce an embedding of $G$ into the direct product of graph towers $\Ts_{B}$, where $B$ runs over all the branches of the tree $T_{\sol}'(G)$.
\end{proof}

We finish with the following 
\begin{cor}
For any finite system of equations $S(X)=1$ over a partially commutative group $\GG$, one can find effectively a finite family of graph towers  $\Ts_1,\dots,\Ts_k$, $\Ts_i=\factor{\GG[Y_i]}{S_i}$ and word mappings $p_i: V(S_i)\to V(S)$ such that for every $b\in V(S)$ there exist $i$ and $c\in V(S_i)$ for which $b=p(c)$, i.e.
$$
V(S)= p_1(V(S_1)) \cup\dots\cup p_k(V(S_k))
$$
and all sets $p_i(S_i)$ are irreducible; moreover, every irreducible component of $V(S)$ can be obtained as a closure of some $p_i(V(S_i))$ in the Zariski topology.
\end{cor}

\end{document}